\documentclass[final]{siamart1116}

\usepackage{graphicx} 
\ifpdf
  \DeclareGraphicsExtensions{.eps,.pdf,.png,.jpg}
\else
  \DeclareGraphicsExtensions{.eps}
\fi

\ifpdf
\usepackage{lmodern}
\fi
\usepackage{subfigure}

\headers{Adaptive LSFEMs for Elliptic Equations in Non-Divergence Form}
{W. Qiu and S. Zhang}

\title{Adaptive First-Order System Least-Squares Finite Element Methods
for Second Order Elliptic Equations in Non-Divergence Form
}

\author{
  Weifeng Qiu \thanks{Department of Mathematics, City University of Hong Kong, Kowloon Tong, Hong Kong SAR, China
    (\email{weifeqiu@cityu.edu.hk}). The author is supported by a Grant from the Research Grants Council of the Hong Kong Special Administrative Region, China (Project No. CityU 11304017)}
  \and
  Shun Zhang\thanks{Department of Mathematics, City University of Hong Kong, Kowloon Tong, Hong Kong SAR, China
   (\email{shun.zhang@cityu.edu.hk})}
}

\usepackage{graphics}
\usepackage{epsfig}

\usepackage{enumerate}

\renewcommand{\theequation}{\thesection.\arabic{equation}}

\newtheorem{thm}{Theorem}[section]

\newtheorem{lem}[thm]{Lemma}

\newtheorem{rem}[thm]{Remark}

\newtheorem{assumption}[thm]{\textit{Assumption}}

\begin{document}
\newcommand{\BX}{{\bf X}}
\newcommand{\cv}{{\cal V}}
\newcommand{\cW}{{\cal W}}
\newcommand{\co}{{\cal O}}

\def\abstract{
\advance \rightskip by 10mm
\advance \leftskip by 10mm
\vspace{-0.8em}
\noindent
\small{\bf Abstract.}
}
\def\endabstract{\par\normalsize\rm}

\def\Xint#1{\mathchoice
{\XXint\displaystyle\textstyle{#1}}%
{\XXint\textstyle\scriptstyle{#1}}%
{\XXint\scriptstyle\scriptscriptstyle{#1}}%
{\XXint\scriptscriptstyle\scriptscriptstyle{#1}}%
\!\int}
\def\XXint#1#2#3{{\setbox0=\hbox{$#1{#2#3}{\int}$}
\vcenter{\hbox{$#2#3$}}\kern-.5\wd0}}
\def\ddashint{\Xint=}
\def\dashint{\Xint-}

\def\a{\alpha}
\def\b{\beta}
\def\d{\delta}\def\D{\Delta}
\def\e{\epsilon}
\def\g{\gamma}\def\G{\Gamma}
\def\k{\kappa}
\def\lam{\lambda}\def\Lam{\Lambda}
\renewcommand\o{\omega}\renewcommand\O{\Omega}
\def\s{\sigma}\def\S{\Sigma}
\renewcommand\t{\theta}\def\vt{\vartheta}
\newcommand{\vphi}{\varphi}
\def\z{\zeta}

\newcommand{\tsigma}{\tilde{\s}}
\newcommand{\tbsigma}{\tilde{\bsigma}}
\def\te{\tilde{\e}}
\def\tu{\tilde{u}}

\newcommand{\bchi}{\mbox{\boldmath$\chi$}}
\newcommand{\bdelta}{\mbox{\boldmath$\delta$}}
\newcommand{\bepsilon}{\mbox{\boldmath$\epsilon$}}
\newcommand{\bfeta}{\mbox{\boldmath$\eta$}}
\newcommand{\bgamma}{\mbox{\boldmath$\gamma$}}
\newcommand{\bomega}{\mbox{\boldmath$\omega$}}
\newcommand{\bvphi}{\mbox{\boldmath$\varphi$}}
\newcommand{\bphi}{\mbox{\boldmath$\phi$}}
\newcommand{\bPhi}{\mbox{\boldmath$\Phi$}}
\newcommand{\bpsi}{\mbox{\boldmath$\psi$}}
\newcommand{\bPsi}{\mbox{\boldmath$\Psi$}}
\newcommand{\bsigma}{\mbox{\boldmath$\sigma$}}
\newcommand{\btau}{\mbox{\boldmath$\tau$}}
\newcommand{\bxi}{\mbox{\boldmath$\xi$}}
\newcommand{\brho}{\mbox{\boldmath$\rho$}}
\newcommand{\bbeta}{\mbox{\boldmath$\beta$}}
\newcommand{\bzeta}{\mbox{\boldmath$\zeta$}}

\def\bk{\boldsymbol{\kappa}}
\def\bmu{\boldsymbol\mu}
\def\bxi{\boldsymbol{\xi}}
\def\bz{\boldsymbol{\zeta}}

\def\ba{{\bf a}}
\def\bb{{\bf b}}
\def\bc{{\bf c}}
\def\be{{\bf e}}
\def\bff{{\bf f}}
\def\bg{{\bf g}}
\def\bn{{\bf n}}
\def\bp{{\bf p}}
\def\bq{{\bf q}}
\def\bs{{\bf s}}
\def\bt{{\bf t}}
\def\bu{{\bf u}}
\def\bv{{\bf v}}
\def\bw{{\bf w}}
\def\bx{{\bf x}}
\def\by{{\bf y}}
\def\bzz{{\bf z}}

\def\bD{{\bf D}}
\def\bE{{\bf E}}
\def\bF{{\bf F}}
\def\bH{{\bf H}}
\def\bJ{{\bf J}}
\def\bV{{\bf V}}
\def\bU{{\bf U}}
\def\bW{{\bf W}}
\def\bX{{\bf X}}
\def\bY{{\bf Y}}

\def\cA{{\cal A}}
\def\cC{{\cal C}}
\def\cD{{\cal D}}
\def\cE{{\cal E}}
\def\cF{{\cal F}}
\def\cG{{\cal G}}
\def\cI{{\cal I}}
\def\cJ{{\cal J}}
\def\cK{{\cal K}}
\def\cL{{\cal L}}
\def\cO{{\cal O}}
\def\cP{{\cal P}}
\def\cQ{{\cal Q}}
\def\cR{{\cal R}}
\def\cS{{\cal \Sigma}}
\def\cT{{\cal T}}
\def\cU{{\cal U}}
\def\cV{{\cal V}}

\def\scT{{_\cT}}
\def\sD{{_D}}
\def\sE{{_E}}
\def\sF{{_F}}
\def\sFz{{_{F_z}}}
\def\sK{{_K}}
\def\sI{{_I}}
\def\sb{{_b}}
\def\sN{{_N}}

\def\curl{{{\bf curl} \ }}
\def\rot{{\mbox{rot}\ }}
\def\BPI{{\bf \Pi}}

\def\cth{\cT_h}
\def\ctH{\cT_H}

\def\tJ{\tilde{\J}}

\def\hK{\widehat{K}}
\def\hx{\widehat{x}}
\def\hy{\widehat{y}}
\def\bhv{\widehat{\bv}}

\def\l{\ell}
\def\bl{\boldsymbol{\ell}}
\def\col{\colon}
\def\f12{\frac12}
\def\dfrac{\displaystyle\frac}
\def\dint{\displaystyle\int}
\def\nab{\nabla}
\def\p{\partial}
\def\sm{\setminus}
\def\dsum{\displaystyle\sum}
\newcommand{\pp}[2]{\frac{\partial {#1}}{\partial {#2}}}
\def\bzero{{\bf 0}}

\def\divv{\nab\cdot}
\def\divx{\nab_x\cdot}
\def\divtx{\nab_{t,x}\cdot}
\def\nabx{\nab_x}

\newcommand{\grad}{\nabla}
\newcommand{\curlt}{{\nabla \times}}
\newcommand{\gperp}{\nabla^{\perp}}
\newcommand{\gradt}{\nabla\cdot}

\def\forallqq{\quad\forall\,}
\def\aph{A^{1/2}}
\def\amh{A^{-1/2}}

\def\osc{{\rm osc \, }}

\def\Im{{\rm Im}}
\newcommand{\tr}{{\rm tr}}
\def\divvr{{\rm div}}
\def\curllr{{\rm curl}}
\def\curll{{\rm curl}}
\def\curl{{\bf curl}}
\newcommand{\bgrad}{{\bf grad}}
\newcommand\diam{\mathrm{diam\,}}
\renewcommand\Im{\mathrm{Im\,}}
\def\Span{\mbox{Span}}
\def\supp{\mbox{supp\,}}
\newcommand{\trace}{{\rm trace}}

\newcommand{\tri}{|\!|\!|}
\newcommand{\ljump}{\lbrack\!\lbrack}
\newcommand{\rjump}{\rbrack\!\rbrack}
\newcommand{\bdm}{\begin{displaymath}}
\newcommand{\edm}{\end{displaymath}}
\newcommand{\beq}{\begin{equation}}
\newcommand{\eeq}{\end{equation}}
\newcommand{\beqa}{\begin{eqnarray}}
\newcommand{\eeqa}{\end{eqnarray}}
\newcommand{\beqas}{\begin{eqnarray*}}
\newcommand{\eeqas}{\end{eqnarray*}}
\newcommand{\ul}{\underline}
\newcommand{\wh}{\widehat}
\newcommand{\la}{\langle}
\newcommand{\ra}{\rangle}

\newcommand{\Lt}{L^2(\Omega)}
\newcommand{\Lts}{L^2(\Omega)^2}
\newcommand{\Ltc}{L^2(\Omega)^3}
\newcommand{\Ho}{H^1(\Omega)}
\newcommand{\Hoh}{H^1(\wh{\Omega})}
\newcommand{\Hoi}{H^1(\Omega_i)}
\newcommand{\Hos}{H^1(\Omega)^2}
\newcommand{\Hoc}{H^1(\Omega)^3}
\newcommand{\Hoch}{H^1(\wh{\Omega})^3}
\newcommand{\Hoci}{H^1(\Omega_i)^3}
\newcommand{\Hoz}{H^1_0(\Omega)}
\newcommand{\Ht}{H^2(\Omega)}
\newcommand{\Hti}{H^2(\Omega_i)}
\newcommand{\Hts}{H^2(\Omega)^2}
\newcommand{\Htc}{H^2(\Omega)^3}
\newcommand{\Htz}{H^0(\Omega)}
\newcommand{\Hh}{H^{1/2}(\Gamma)}
\newcommand{\Hhi}{H^{1/2}(\Gamma_i)}
\newcommand{\Hmh}{H^{-1/2}(\Gamma)}
\newcommand{\Hdiv}{H(\divvr;\,\Omega)}
\newcommand{\Hdivh}{H(\divv;\,\wh \Omega)}
\newcommand{\hcurl}{H(\curl\,A;\,\Omega)}
\newcommand{\Hcurl}{H(\curll\,A;\,\Omega)}
\newcommand{\Hcrl}{H(\curll\,;\,\Omega)}
\newcommand{\hcrl}{H(\curl\,;\,\Omega)}
\newcommand{\Hcrlh}{H(\curll\,;\,\wh\Omega)}
\newcommand{\hcrlh}{H(\curl\,;\,\wh\Omega)}
\newcommand{\Wdiv}{\BW_0(\mbox{\divv}\,;\,\Omega)}
\newcommand{\Wcurl}{\BW_0(\mbox{\curl}\,A;\,\Omega)}
\newcommand{\WcrossV}{\BW \times V}

\def\calS{{\cal S}}
\def\calT{{\cal T}}
\def\cB{{\cal B}}
\def\cH{{\cal H}}
\def\ba{{\mathbf{a}}}
\def\cM{{\cal M}}
\def\cN{{\mathcal{N}}}
\def\cE{{\mathcal{E}}}
\def\cT{{\mathcal{T}}}
\def\cJ{{\mathcal{J}}}
\def\cL{{\mathcal{L}}}

\def\bE{{\bf E}}
\def\bS{{\bf S}}
\def\br{{\bf r}}
\def\bW{{\bf W}}
\def\bLambda{{\bf \Lambda}}

\newcommand{\lJump}{[\![}
\newcommand{\rJump}{]\!]}
\newcommand{\jump}[1]{[\![ #1]\!]}

\newcommand{\st}{\tilde{\bsigma}}
\newcommand{\sh}{\hat{\bsigma}}
\newcommand{\rd}{\brho^{\Delta}}

\newcommand{\WH}{W\!H}
\newcommand{\NE}{N\!E}

\newcommand{\ND}{N\!D}
\newcommand{\BDM}{B\!D\!M}

\newcommand{\sT}{{_T}}
\newcommand{\sRT}{{_{RT}}}
\newcommand{\sBDM}{{_{BDM}}}
\newcommand{\sWH}{{_{WH}}}
\newcommand{\sND}{{_{ND}}}
\newcommand{\sV}{_\cV}

\newcommand{\dd}{\underline{{\mathbf d}}}
\newcommand{\C}{\rm I\kern-.5emC}
\newcommand{\R}{\rm I\kern-.19emR}
\newcommand{\W}{{\mathbf W}}
\def\3bar{{|\hspace{-.02in}|\hspace{-.02in}|}}
\newcommand{\A}{{\mathcal A}}

\newcommand*\cnorm[1]{\left\vert\!\left[\!\left]#1\right[\!\right]\!\right\vert}

\newcommand{\ratio}{{\,:\,}}     
\maketitle
\begin{abstract}
This paper studies adaptive first-order least-squares finite element methods for second-order elliptic partial differential equations in non-divergence form. Unlike the classical finite element method which uses weak formulations of PDEs not applicable for the non-divergence equation, the first-order least-squares formulations naturally have stable weak forms without using integration by parts, allow simple finite element approximation spaces, and have build-in a posteriori error estimators  for adaptive mesh refinements. 

The non-divergence equation is first written as a system of first-order equations by introducing the gradient as a new variable. Then two versions of least-squares finite element methods using simple $C^0$ finite elements are developed in the paper, one is the $L^2$-LSFEM which uses linear elements, the other is the weighted-LSFEM with a mesh-dependent weight to ensure the optimal convergence. Under a very mild assumption that the PDE has a unique solution, optimal a priori and a posteriori error estimates are proved. With an extra assumption on the operator regularity which is weaker than traditionally assumed, convergences in standard norms for the weighted-LSFEM are also discussed. $L^2$-error estimates are  derived for both formulations. We perform extensive numerical experiments for smooth, non-smooth, and even degenerate coefficients on smooth and singular solutions to test the accuracy and efficiency of the proposed methods.
\end{abstract}



\section{Introduction}\label{intro}
In this paper we consider finite element approximations of the following elliptic PDE in non-divergence form:

\begin{align}
 \label{eq_nondiv}
-A : D^2 u &= f \quad \mbox{in} \,\ \O, \\ \nonumber
                        u &= 0 \quad \mbox{on} \,\ \p\O.
\end{align}
Here, the domain $\O \in \R^d$ is an open and bounded polytope for $d =2$ or $3$, the coefficient matrix $A = A(x) \in L^{\infty}(\O)^{d\times d}$ is a positive definite matrix with eigenvalues bounded by $\lambda >0$ below and $\Lambda >0$ above on $\O$, but not necessarily differentiable. The righthand side $f$ is assumed in $L^2(\O)$.

Note that when $A\in [C^1(\O)]^{n\times n}$, we have the following equation in divergence form, where $\gradt A$ is taken row-wise:
\begin{align}
 \label{eq_div}
-\gradt(A \nabla u) + (\gradt A) \cdot \nabla u &= f \quad \mbox{in} \,\ \O, \\ \nonumber
                        u &= 0 \quad \mbox{on} \,\ \p\O.
\end{align}

The elliptic PDE in non-divergence form arises in the linearization of fully nonlinear PDEs, for example, stochastic control problems, nonlinear elasticity, and mathematical finance. The matrix $A(x)$ is not smooth nor even continuous in many such cases. For example, for fully nonlinear PDEs solving by $C^0$ finite element methods, the coefficient matrix of its linearization is possibly only element-wisely smooth if the coefficients containing derivatives of the numerical solution. 

Since the matrix $A$ is not differentiable, the standard weak notion of elliptic equation is not applicable.  The existence and uniqueness of equations of non-divergence form are often based on the classical or strong senses of the solutions, see discussions in \cite{NZ:18,NW:19}. These PDE theories often assume that the domain $\O$ is convex, the boundary is sufficiently smooth, or some other restrictive conditions on the smoothness of $A$. For a discontinuous $A$, there are possibilities that the solution is non-unique, see an example given in \cite{NZ:18}. It is worth to mention that these available theoretical PDE results are all sufficient theories. For example, since the Poisson equation is also an example of the non-divergence equation with $A=I$, the existence and uniqueness condition of the equation \eqref{eq_nondiv} dependence on the domain can be very weak.

There are several numerical methods are available for the problem in non-divergence form. Based on discrete Calderon-Zygmond estimates, Feng, Neilan, and co-authors developed finite element methods for problems with a continuous coefficient matrix in \cite{FHN:17,FNS:18,Nei:17}. For equations with discontinuous coefficients satisfying the Cordes condition, a discontinuous Galerkin method \cite{SS:13}, a mixed method \cite{Gal:17}, and a non-symmetric method \cite{NW:19} are developed. A weak Galerkin method is developed by Wang and Wang in \cite{WW:18}. The analysis of these papers mostly assumes the full $H^2$ regularity of the operator and studies the $H^2$-error estimates of the approximations. In some sense, these methods keep the non-divergence operator second order and borrow techniques from variational fourth-order problems. Nochetto and Zhang \cite{NZ:18} studied a two-scale method, which is based on the integro-differential approach and focuses on $L^{\infty}$ error estimates.

Traditionally, the finite element method is based on the variational formulation of an elliptic equation, where the integration by parts plays an essential role. The integration by parts can shift a derivative from the trial variable to the test variable, thus reduces the differential order of the operator.  For \eqref{eq_nondiv}, the integration by parts is not available. Luckily, there is another natural method to reduce the differential order of a PDE operator by introducing another auxiliary variable. We can reduce the second order equation into a system of a first order equation by using the new auxiliary variable. Normally, for the first-order system, we can two approaches. One is the mixed method which also involving the integration by parts and has difficulties to ensure the stability. The other method is the least-squares finite element method (LSFEM). The first-order system least-squares principle first re-write the PDE into a first-order system, then define an artificial, externally defined energy-type principle. The energy functional can be defined as summations of weighted residuals of the system. With the first-order least-squares functionals, corresponding LSFEMs can be defined. No integration by parts is needed to define the least-squares principle and thus the LSFEM, thus the first-order system LSFEM is ideal for the second order elliptic equation of non-divergence form.

Beside the obvious advantage of non-requirement of integration by parts, the LSFEM has other advantages. First, the least-squares weak formulation and its associated LSFEM using conforming finite element spaces are automatically coercive as long as the first-order system is well-posed. This is a significant advantage over other numerical methods since the well-posedness theory of the equation in non-divergence form is in general only sufficient. On the other hand, even without a rigorous mathematical proof, the elliptic equation in non-divergence form is often a result of some physical process that we are sure that a unique solution exists. Thus, in LSFEMs developed in this paper, we can reduce the condition of the PDE into a simple well-posedness without specifing the condition explicitly.

The other advantages of LSFEMs include conforming discretizations lead to stable and, ultimately, optimally accurate methods, the resulting algebraic problems are symmetric, positive definite, and can be solved by standard and robust iterative methods including multigrid methods.

The last important advantage of the LSFEMs is that it has a build-in a posteriori error estimator. The solution is probably singular due to the geometry of the domain or the coefficient matrix. Also, for problems like reaction diffusion equations, interior or boundary layers appear. To solve these problems efficiently, the a posteriori error estimator and adaptive mesh refinement algorithm are necessary.

In this paper, by introducing the gradient as an auxiliary variable, we first write the equation in non-divergence form into a system of first-order equations, then develop two least-squares minimization principles and two corresponding LSFEMs: one is based on an $L^2$-norm square sum of the residuals and the other is based on a mesh-size weighted $L^2$-norm square sum  of the residuals. The two methods are called $L^2$-LSFEM and weighted-LSFEM, respectively. For the $L^2$-LSFEM, simpliest linear $C^0$-finite elements are used to approximate both the solution and the gradient. For the weighted-LSFEM, the $C^0$-finite element of degree $k$, $k\geq 2$ is used to approximate the solution, while the degree $k-1$ $C^0$-finite element is used to approximate the gradient. Under the very weak assumption that the coefficient and domain is good enough to guarantee the existence and uniqueness of a solution, we show both continuous least-squares weak forms and their corresponding discrete problems are well-posed. A priori and a posteriori error estimates with respect to the least-squares norms are then discussed. 

Numerical methods for non-divergence equation often use the following operator regularity assumption
$$
\|u\|_2\leq C\|A\ratio D^2 u\|_0
$$
to derive stability and error estimates. Unlike these papers, for the weighted-LSFEM, the error estimates of error in the $H^1$-norm and the discrete broken $H^2$-norm are investigated with a weaker assumption, see our discussion in section 4.2. Under stronger regularity assumptions, we show that the $L^2$-norm of the error of the solution is one order higher than the least-squares norm of the error, providing the approximation degree for the solution is at least three. For the $L^2$-LSFEM, we show the optimal $L^2$ and $H^1$-error estimates with a solution regularity assumption. We perform extensive numerical experiments for smooth, non-smooth, and even degenerate coefficients on smooth and singular solutions to test the accuracy and efficiency of the proposed methods. With uniform refinements, we show the convergence orders match with the theory. With adaptive mesh refinements, optimal convergences results are obtained for singular solutions.

The LSFEM is well developed for the elliptic equations in divergence form, see for example, \cite{Cai:94,Cai:97,BLP:97,BG:09,CFZ:15,CQ:17}. A posteriori error estimates and adaptivity algorithms based on LSFEMs can be founded in \cite{BMM:97,CZ:10}. Compared the the LSFEMs for the elliptic equation in divergence form, the non-divergence equation has many differences in the stability analysis and choices of the finite element sub-spaces due to the non-divergence structure. We remark these differences in the various places of the paper as comparisons.

In a summary, the LSFEMs developed in this paper have several advantages compared to existing numerical methods: they are automatically stable under very mild assumption; they are easy to program due to that only simple Lagrange finite elements without jump terms are used; adaptive algorithms with the build-in a posteriori error estimators  can handle problems with singular solutions or layers; under a condition on the operator regularity which is weaker than traditionally assumed,  error estimates in standard norms are proved.

There are two least-squares finite element methods available for the non-divergence equation. None of them use a first-order reformulation. The paper \cite{Gal:17} uses a second order least-squares formulation with $C^1$-finite element approximations. The simple method developed by Ye and Mu in \cite{MY:17} uses $C^0$ -finite element spaces with orders higher than two and penalize the continuity of the solution and the normal component of the flux.  

The remaining parts of this article are as follows: section 2 defines the first-order system least-squares weak problems and discusses their stabilities; section 3 presents the corresponding LSFEMs and their a priori and a posteriori error estimates in least-squares norms. Error estimates in other norms are discussed in sections 4 and 5 for the weighted and $L^2$ versions of methods, separately. Numerical experiments are presented in section 6. 

Standard notation on function spaces applies throughout this article. Norms of functions in Lebesgue
and Sobolev space $H^{k}(\omega)$ ($L^2(\omega) = H^0(\omega)$) are denoted by $\|\cdot\|_{k,\omega}$. The subscript $\omega$ is omitted when $\omega=\O$. The inner product of real-valued $d\times d$ matrices $A\ratio B$ is denoted by $A\ratio B = \sum_{i,j}^d a_{ij}b_{ij}$. We use $D^2 v$ to denote the Hessian of $v$.
\section{First-Order System Least-Squares Weak Problems}
\setcounter{equation}{0}

\subsection{Existence and uniqueness assumption}
Define the solution space of \eqref{eq_nondiv}:
\beq \label{defV}
V := \{ v\in H^1_0(\O)\mbox{  and  } A\ratio  D^2 v \in L^2(\O)\}.
\eeq
Notice that the space $V$ is weaker then $H^1_0(\O)\cap H^2(\O)$, since we can expect that even though an individual $\p_{ij}^2 v$, $i,j=1,\cdots, d$ is not in $L^2(\O)$, due to cancelation or good properties of $A$, $A\ratio D^2 v$ belongs to $L^2(\O)$. For example, let $A=I$ and $w\in H^1_0(\O)$ is the solution of the Poisson equation $-\Delta w=1$ on an L-shaped domain $\O = (-1,1)^2 \backslash [0,1)\times (-1,0]$, then $w \in V$ but clearly $w\not\in H^1_0(\O)\cap H^2(\O)$.

We first state the assumption of the existence and uniqueness of the solution.
\begin{assumption} \label{assump_EU}
{\bf (Existence and uniqueness of the solution of the elliptic equation in non-divergence form)}
Assume that  the coefficient matrix $A$ and the domain $\O$ are nice enough, such that the equation \eqref{eq_nondiv} has a unique solution $u\in V$ for any $f\in L^2(\O)$. 
\end{assumption}
\begin{rem}
There are various theories to ensure the existence and uniqueness of  the equation, for example:
\begin{enumerate}
\item (Classical solution \cite{GT:01}) A classical solution exists if $A$ is H\"older continuous and if $\p \O$ is sufficiently smooth.

\item (Strong solution \cite{CFL:93}) If $A\in VMO(\O)\cap L^\infty(\O)$, a vanishing mean oscillation matrix with a uniform VMO-modulus of continuity, and if $\O$ is of class $C^{1,1}$, then there exists a unique solution $u\in W^{2,p}$ to the problem.

\item ($H^2$ solution \cite{MPS:00}) If the domain is convex and if $A$ satisfies the Cordes condition \eqref{Cordes}, then there exists a solution $u\in H^2(\O)$.
\end{enumerate}
More detailed discussions can be found in the introduction of \cite{NZ:18}. We do find these theories are only sufficient theories. Besides many examples of Poisson equations on non-convex domains, the Test 3 problem from \cite{FHN:17}( see also our numerical test 6.4) is an example that the matrix $A$ is not uniformly elliptic but a unique solution still exists.
\end{rem}

\subsection{Least-squares problems}

Introduce a new gradient variable $\bsigma = \nabla u$, we have the following first-order system:
\beq
\label{1stordersys}
\left\{
\begin{array}{rllll}
\bsigma-\nabla u &=&0, &  \mbox{in } \O, \\[2mm]
-A : \nabla \bsigma &=&f,&  \mbox{in } \O,\\[2mm]
u&=&0,&  \mbox{on } \p\O.
\end{array}
\right.
\eeq
It is clear that $u\in H_0^1(\O)$. For the gradient $\bsigma$, the appropriate solution space is:
$$
Q := \{ \btau \in L^2(\O)^d : A\ratio  \nabla \btau \in L^2(\O)\}.
$$
\begin{rem}
As a comparison, consider the equation in divergence form 
\begin{align}
 \label{eqdiv}
-\gradt(A \nabla u)  &= f \quad \mbox{in} \,\ \O, \\ \nonumber
                        u &= 0 \quad \mbox{on} \,\ \p\O.
\end{align}
It is well known that the flux $- A\nabla u \in H(\divvr;\O)$ for $f\in L^2(\O)$. The space $H(\divvr;\O)$ is well studied \cite{GR:86,BBF:13}. One of its property is that the normal component of its member vector function is continuous across the interfaces in the weak sense. Also, the negative divergence operator is the the dual operator of the gradient. These properties play important roles in the design of least-squares finite element methods for elliptic equations in divergence form. More importantly, $H(\divvr;\O)$-conforming finite element spaces such as Raviart-Thomas (RT) finite element space \cite{BBF:13} with good approximation properties are also well known.

For the elliptic equation non-divergence form \eqref{eq_nondiv}, the property of the space $Q$ is barely known. Similar to the space $V$, for a vector function $\btau \in Q$, we can expect that even though an individual $\p \btau/\p x_i$, $i=1,\cdots, d$ is not in $L^2(\O)$, due to cancelation or good properties of $A$, $A\ratio \nabla \btau$ may belong to $L^2(\O)$. But if we want to design a numerical method for a general non-divergence elliptic equation, we basically cannot assume or use any of these information, and we can not design an $A$-intrinsic finite element subspace of $Q$ as $H(\divvr)$-conforming finite elements.   
\end{rem}

Let $\cT = \{K\}$ be a triangulation of $\O$ using simplicial elements. The mesh $\cT$ is assumed to be shape-regular, but it does not to be quasi-uniform. Let $h_K$ be the diameter of the element $K\in \cT$.

We introduce two versions of least-squares functionals:
\begin{eqnarray}
\cJ_h(v,\btau; f) &:= & \sum_{K\in\cT} h_K^2
\|f + A\ratio  \nabla \btau \|_{0,K}^2 +\|\btau-\nabla v \|_{0}^2, \;\; \forall (v,\btau) \in H^1_0(\O)\times Q,\\
\cJ_0(v,\btau; f) &:= & \|f + A\ratio  \nabla \btau \|_{0}^2 +\|\btau-\nabla v \|_{0}^2, \quad \forall (v,\btau) \in H^1_0(\O)\times Q.
\end{eqnarray}
The functionals $\cJ_h$ and $\cJ_0$  are called the weighted version and the $L^2$ version, respectively. We use the notation $\cJ$ to denote both $\cJ_h$ and $\cJ_0$ when two formulations can be presented in a unified framework and no confusion is caused.

The least-squares minimization problem is: seek $(u,\bsigma) \in H^1_0(\O)\times Q$, such that
\beq \label{LS}
\cJ(u,\bsigma; f) = \inf_{(v,\btau) \in H^1_0(\O)\times Q} \cJ(v,\btau; f).
\eeq
The corresponding Euler-Lagrange formulations are: seek $(u,\bsigma) \in H^1_0(\O)\times Q$, such that
\beq \label{LS_EL_h}
a_h((u,\bsigma),(v,\btau)) = -\sum_{K\in\cT} h_K^2(f, A\ratio \nabla \btau )_K, \quad \forall (v,\btau) \in H^1_0(\O)\times Q,
\eeq
and find $(u,\bsigma) \in H^1_0(\O)\times Q$, such that
\beq \label{LS_EL_L2}
a_0((u,\bsigma),(v,\btau)) = -(f, A\ratio \nabla \btau ), \quad \forall (v,\btau) \in H^1_0(\O)\times Q,
\eeq
where for all $ (w,\brho)$ and $(v,\btau) \in H^1_0(\O)\times Q$, the bilinear forms are defined as
\begin{eqnarray*}
a_h((w,\brho),(v,\btau)) &:= & (\brho-\nabla w, \btau-\nabla v) + \sum_{K\in\cT}h_K^2(A\ratio \nabla \brho, A\ratio \nabla \btau )_K,\\
\mbox{and}\quad a_0((w,\brho),(v,\btau)) &:= & (\brho-\nabla w, \btau-\nabla v) + (A\ratio \nabla \brho, A\ratio \nabla \btau ).
\end{eqnarray*}

\begin{rem}
The least-squares formulations can be easily extended to more general cases, for example, non-homogeneous Dirichlet boundary conditions and equations with convection and advection terms. 

For example, for the general elliptic equation
\beq
-A : D^2 u + \bbeta \cdot \nabla u + c u= f
\eeq
with homogeneous boundary condition, let $\bsigma =\nabla u$, the $L^2$ least-squares functional can be defined as:
$$
\cJ_0(v,\btau; f) :=  \|f + A\ratio  \nabla \btau -\bbeta \cdot \btau - cv\|_{0}^2 +\|\btau-\nabla v \|_{0}^2, \quad \forall (v,\btau) \in H^1_0(\O)\times Q.
$$

To have a better robustness with respect to the coefficients, coefficient-weighted versions can also be used. For example, define the $L^2$ least-squares functional as:
$$
\cJ_0(v,\btau; f) :=  \|\gamma(f + A\ratio  \nabla \btau -\bbeta \cdot \btau - c v)\|_{0}^2 +\|A^{1/2}(\btau-\nabla v) \|_{0}^2, \quad \forall (v,\btau) \in H^1_0(\O)\times Q,
$$
where $\gamma >0$ is a weight defined as a function of the coefficients $A$, $\bbeta$, and $c$.

The $h$-weighted least-squares functional can be defined similarly. 
\end{rem}

\begin{rem}
For the weighted functional $\cJ_h$, the $h$-weight is on the term $\|f+A\ratio \nabla \btau\|_0$, similarly, we can also use, 
$$
\tilde{\cJ}_h(v,\btau; f) :=  
\|f + A\ratio  \nabla \btau \|_{0}^2 + \sum_{K\in\cT} h_K^{-2}\|\btau-\nabla v \|_{0,K}^2, \;\; \forall (v,\btau) \in H^1_0(\O)\times Q,
$$
as the weighted least-squares functional. For a uniform mesh, $\cJ_h(v,\btau; f)$ and $\tilde{\cJ}_h(v,\btau; f)$ are equivalent. But for an adaptively refined mesh, $\cJ_h(v,\btau; f)$ beahives more like a minimization problem with respect to the $H^1$-norm of $u$ while $\tilde{\cJ}_h(v,\btau; f)$ is more like an optimization with respect to the $H^2$-norm. We prefer the $\cJ_h(v,\btau; f)$ version in this paper since the minimum requirement of $u$ is $H^1(\O)$ not $H^2(\O)$. Earlier discussion on the mesh-dependent least-squares methods can be found in \cite{AKS:85}.
\end{rem}

\begin{lem} \label{norm1}
The following are norms for $(v,\btau) \in H^1_0(\O)\times Q$:
\begin{eqnarray*}
\tri (v,\btau) \tri_h^2 &:=& 
\sum_{K\in\cT}h_K^2\|A\ratio  \nabla \btau \|_{0,K}^2 +\|\btau-\nabla v\|_0^2, \\
\mbox{and}\quad \tri (v,\btau) \tri_0^2 &:=& 
\|A\ratio  \nabla \btau \|_{0}^2 +\|\btau-\nabla v\|_0^2.
\end{eqnarray*}
\end{lem}
We use $\tri (v,\btau) \tri$ to denote both versions when no confusion is caused.
\begin{proof}
To prove that $\tri (v,\btau) \tri$ defines a norm on $H^1_0(\O)\times Q$, we only need to check conditions of a norm definition.
  
The linearity and the triangle inequality are obvious for $\tri (v,\btau) \tri$.

If $\tri (v,\btau) \tri =0$, due to the fact $\btau \in Q$, we have $A\ratio  \nabla \btau \in L^2(\O)$, thus
$$
A\ratio  \nabla \btau =0 \quad\mbox{and}\quad \btau =\nabla v,
$$
in the $L^2$ sense. This means, $v\in V$, and
$$
A\ratio  D^2 v=0 \mbox{  in  } \O, \quad v=0 \mbox{ on }\p\O,
$$
is true in the  $L^2$ sense. By the existence and uniqueness of assumption of the solution Assumption \ref{assump_EU}, $v=0$ and $\btau=0$. The norm $\tri \cdot\tri$ is then well defined for both the weighted and $L^2$ versions of definition.
\end{proof}
\begin{rem}
The condition $\btau \in Q$ is essential to the definition. This condition has the same role as the requirement of flux in $H(\divvr;\O)$ for the equation in divergence form, which implicitly implies some weak continuity condition of its member functions. 
\end{rem}

\begin{rem}
It is also clear that
\begin{eqnarray*}
\tri (v,\btau) \tri_{h,K}^2 &:=& 
h_K^2\|A\ratio  \nabla \btau \|_{0,K}^2 +\|\btau-\nabla v \|_{0,K}^2 \\
  \mbox{and}\quad
\tri (v,\btau) \tri_{0,K}^2 &:=& 
\|A\ratio  \nabla \btau \|_{0,K}^2 +\|\btau-\nabla v \|_{0,K}^2,
\end{eqnarray*}
are semi-norms on an element $K\in\cT$.
\end{rem}

\begin{lem} The bilinear form $a=a_h$ or $a_0$ is continuous and coercive:
\begin{eqnarray}
a((w,\brho),(v,\btau)) &\leq & \tri (w,\brho) \tri \tri (v,\btau) \tri, \quad \forall (w,\brho) \mbox{ and } (v,\btau) \in H^1_0(\O)\times Q, \\[1mm]
a((v,\btau),(v,\btau)) &= & \tri (w,\brho) \tri^2, \quad \forall  (v,\btau) \in H^1_0(\O)\times Q.
\end{eqnarray}
\end{lem}
The lemma can be easily proved by a simple computation.

\begin{thm} \label{EU1}
Assume that $f\in L^2(\O)$, the coefficient matrix $A$ and the domain $\O$ are nice enough such that Assumption \ref{assump_EU} is true, then the least-squares problem (\ref{LS}) has a unique solution $(u,\bsigma) \in H^1_0(\O)\times Q$.
\end{thm}
\begin{proof}
To prove the existence, for $f\in L^2(\O)$, by Assumption \ref{assump_EU}, there exists a unique $u\in H_0^1(\O)\subset V$ solving the equation. Let $\bsigma =\nabla u$, it is easy to that $\bsigma \in Q$ and $A:\nabla \bsigma +f =0$, thus the least-squares functional \eqref{LS} has a minimizer $(u,\bsigma)\in H_0^1(\O)\times Q$ with minimum value zero. The minimizer $(u,\bsigma) \in H_0^1(\O)\times Q$ is then the solution of the least-squares problem (\ref{LS}) and its corresponding Euler-Lagrange equation. The uniqueness is a simple consequence of the fact $\tri (v,\btau) \tri$ is a norm.
\end{proof}

\begin{rem}
The above proof can be easily generalize to the case that the Drichelet boundary condition is not homogeneous.

The above argument to show the existence and uniqueness of the least-squares formulation is useful when the existences and uniqueness of the PDE is obtained from various non-variational techniques. A similar argument is used to prove the stability of least-squares formulations for the linear transport equation in \cite{LZ:18}.
\end{rem}

\begin{rem}
Here, the assumption of the coefficient $A$ is quite weak. The matrix $A$ does not need to be in $C^{0}(\O)^{d\times d}$, it can even be degenerate as long as  Assumption \ref{assump_EU} still holds.
\end{rem}

\begin{rem}
For the elliptic equation in divergence form \eqref{eqdiv}, traditionally there are two forms on least-squares functionals \cite{Cai:94,BLP:97}:
\begin{eqnarray*}
\cL_0(v,\btau; f) &:= & \|\gradt \btau -f \|_{0}^2 +\|A^{-1/2}\btau+A^{1/2}\nabla v \|_{0}^2, \quad \forall (v,\btau) \in H^1_0(\O)\times H(\divvr;\O),\\
\cL_{-1}(v,\btau; f) &:= & 
\|\gradt \btau -f \|_{-1}^2 +\|A^{-1/2}\btau+A^{1/2}\nabla v \|_{0}^2, \quad \forall (v,\btau) \in H^1_0(\O)\times H(\divvr;\O).
\end{eqnarray*}
A norm equivalence can be proved: there exists positive constant $C_1$ and $C_2$, such that for $(\btau,v)\in H(\divvr;\O) \times H^1(\O)$:
$$
C_1 (\|\btau\|_{H(\divvr)}^2 + \|v\|_1^2) \leq \cL_0(v,\btau; 0) \leq C_2 (\|\btau\|_{H(\divvr)}^2 + \|v\|_1^2)
$$
and
$$
C_1 (\|\btau\|_{0}^2 + \|v\|_1^2) \leq \cL_{-1}(v,\btau; 0) \leq C_2 (\|\btau\|_{0}^2 + \|v\|_1^2).
$$
Our $L^2$ least-squares functional essentially is a modification of $\cL_0$. But due to the lack of the differentiability of $A$, we cannot prove the following norm equivalence:
\beq \label{equvalience}
C_1 (\|\btau\|_0^2 + \|A\ratio \nabla \btau\|_0^2  + \|v\|_1^2) \leq \cJ_0(v,\btau; 0) \leq C_2 (\|\btau\|_0^2 + \|A\ratio \nabla \btau\|_0^2  + \|v\|_1^2).
\eeq
On the other hand, if $A$ is smooth enough, then \eqref{eq_nondiv} can be written in the divergence form as \eqref{eq_div}, we do can prove \eqref{equvalience} using the same technique for the equation divergence form with similar arguments in \cite{CFZ:15,Ku:07}.

However, for our least-squares functionals, we do have a one-sided bound, which can be easily proved for $(v,\btau) \in H^1_0(\O)\times Q$:
\begin{eqnarray}
C \tri (v,\btau)\tri_0 &\leq & \|A\ratio \nabla \btau\|_0 + \|\btau\|_0 + \|\nabla v\|_0, \\[2mm]
C \tri (v,\btau)\tri_h &\leq & \sum_{K\in\cT}h_K\|A\ratio \nabla \btau\|_{0,K} + \|\btau\|_0 + \|\nabla v\|_0.
\end{eqnarray}
We do not use the minus-$H^1$ norm version in this paper due to its complicated discrete implementation, in stead, we choose a weighted mesh-dependent version to simplify the implementation and keep an optimal order of convergence. 
\end{rem}

\section{Least-Squares Finite Element Methods}
\setcounter{equation}{0}
In this section, LSFEMs based on the least-squares minimization problems are developed. The a priori and a posteriori error estimates with respect to the least-squares norms $\tri (\cdot,\cdot)\tri$ are derived.

\subsection{Least-squares finite element methods}
For an element $K\in \cT$ and an integer $k\geq 0$, let $P_k(K)$ the space of polynomials with degrees less than or equal to $k$. Define the finite element spaces $S_k$ and $S_{k,0}$, $k\geq 1$, as follows:
$$
S_k  :=\{ v \in H^1(\O) \colon
			v|_K \in P_k(K)\,\, \forall \, K\in \cT\}
\quad\mbox{and}\quad
S_{k,0} := S_k \cap H_0^1(\O).
$$
We define the LSFEMs are follows. 

\noindent ({\bf Weighted-LSFEM Problem})
Seek $(u_h, \bsigma_h)\in S_{k,0}\times S_{k-1}^d$, $k\geq 2$, such that
\beq \label{LSFEM_h}
\cJ_h(u_h,\bsigma_h;f) = \inf_{(v,\btau) \in S_{k,0}\times S_{k-1}^d} \cJ_h(v,\btau; f).
\eeq
Or equivalently, find $(u_h, \bsigma_h)\in S_{k,0}\times S_{k-1}^d$, $k\geq 2$, such that
\beq \label{LSFEM_EL_h}
a_h((u_h, \bsigma_h),(v_h, \btau_h)) = -\sum_{K\in\cT} h_K^2 (f, A\ratio  \nabla \btau_h), \quad \forall (v_h,\btau_h) \in S_{k,0}\times S_{k-1}^d.
\eeq

\noindent ({\bf $L^2$-LSFEM Problem})
Seek $(u_h, \bsigma_h)\in S_{1,0}\times S_{1}^d$, such that
\beq \label{LSFEM_L2}
\cJ_0(u_h,\bsigma_h;f) = \inf_{(v,\btau) \in S_{1,0}\times S_{1}^d} \cJ_0(v,\btau; f).
\eeq
Or equivalently, find $(u_h, \bsigma_h)\in S_{1,0}\times S_{1}^d$, such that
\beq \label{LSFEM_EL_L2}
a_0((u_h, \bsigma_h),(v_h, \btau_h)) = -(f, A\ratio  \nabla \btau_h), \quad \forall (v,\btau_h) \in S_{1,0}\times S_{1}^d.
\eeq

The existence and uniqueness of the LSFEM problems are obvious from the facts that
$$
 S_{k}^d \subset H^1(\O)^d  =\{\btau \in L^2(\O)^d : \nabla \btau \in L^2(\O)^d\} \subset Q
$$ and $S_{k,0}\subset H^1_0(\O)$.

\begin{rem}
For the approximation space of $u\in H^1(\O)$, the $H^1$-conforming space is an obvious good choice. For the approximation space of $Q$, we use $S^d  \subset H^1(\O)^d \subset Q$. The space $H^1(\O)^d$ is more restrictive than $Q$, but it has a simple conforming finite element space. The $H^1$-conforming space is not the best choice if further information of $A$ is known. For example, in the case of the equation of divergence form, it is well known that $\nabla u \in H(\mbox{\curl};\O)$ and $A\nabla u\in H(\divvr;\O)$. Similarly for the non-divergence equation, for problems with low regularity and possible discontinuous coefficients, $\nabla u$ may not be in $H^1(\O)^d$, then in some extreme case we will have even though $u_h$ is identical to the exact solution $u$, $\bsigma_h$ cannot equal to $\nabla u$ at the same time, since $\nabla u$ may not be in $H^1(\O)^d$. Such cases will pose problems of a posteriori error estimation and adaptive mesh refinements, see the discussions in \cite{CZ:09} and \cite{CHZ:17} for the failure of the classical Zienkiewicz-Zhu error estimator, which recovers $\nabla u$ in $S^d$ to construct the a posteriori error estimator. Even though we do have this concern, due to the lack of information of $A$ and to keep the method suitable for a general coefficient matrix $A$, the $C^0$-conforming space to approximate $\bsigma = \nabla u$ is still a reasonable choice.  

It is also worth to mention that, we also do not know whether $A\nabla u \in H(\divvr;\O)$ or not for the non-divergence equation. In the the formulations suggested in \cite{FHN:17} and \cite{MY:17}, the normal jump of $A\nabla u$ are used. Thus, these formulations have a similar possible inconsistency as our methods when applied to problems with less smoothness, where the exact jump of $A\nabla u \cdot \bn$ across element interfaces may not be zero for an exact $u$, and have a possibility to introduce an extra error.
\end{rem}

\subsection{A priori error estimates}
\begin{thm} \label{thm_bestapp}
(Cea's lemma type of result)
Let $(u, \bsigma)$ be the solution of least-squares variational problem \eqref{LS}. Let $(u_h, \bsigma_h)\in S_{k,0}\times S_{k-1}^d$, $k\geq 2$, be the solution of the weight-LSFEM problem \eqref{LSFEM_h},  the following best approximation result holds:
\beq
\tri (u-u_h, \bsigma-\bsigma_h)\tri_h \leq \inf _{(v_h,\btau_h) \in  S_{k,0}\times S_{k-1}^d}
\tri(u-v_h, \bsigma-\bsigma_h)\tri_h.
\eeq
Let $(u_h, \bsigma_h)\in S_{1,0}\times S_{1}^d$ be the solution of the $L^2$-LSFEM problem \eqref{LSFEM_L2}, the following best approximation result holds:
\beq
\tri (u-u_h, \bsigma-\bsigma_h)\tri_0 \leq \inf _{(v_h,\btau_h) \in  S_{1,0}\times S_{1}^d}
\tri(u-v_h, \bsigma-\bsigma_h)\tri_0.
\eeq
\end{thm}
\begin{proof}
The proof of the best approximation result is standard.
\end{proof}

\begin{thm} \label{thm_LSFEM_apriori_h}
(A priori error estimate for the weighted-LSFEM)
Assume the solution 
$u\in H^{r+1}(\O)$, for some $r\geq 1$, and  $(u_h, \bsigma_h)\in S_{k,0}\times S_{k-1}^d$, $k\geq 2$ be the solution of the weighted LSFEM problem \eqref{LSFEM_h}, then there exists a constant $C>0$ independent of the mesh size $h$, such that
\begin{align}
\tri (u-u_h, \bsigma-\bsigma_h)\tri_h  \leq C  h^{\min(k,r+1)} \|u\|_{\min(k,r+1)}.
\end{align}
\end{thm}
\begin{proof}
It is easy to see that
\beq \label{local_ineq_h}
C\tri (v,\btau) \tri_h \leq  \sum_{K\in\cT} h_K\|A\ratio  \nabla \btau\|_{0,K} +\|\nabla v\|_0 + \|\btau\|_0. 
\eeq
Then the a priori result is a direct consequence of Theorem \ref{thm_bestapp} and the approximation properties of functions in $S_k$.
\end{proof}

\begin{thm} \label{thm_LSFEM_apriori_L2}
(A priori error estimate for the $L^2$-LSFEM)
Assume the solution $u\in H^{3}(\O)$ and  $(u_h, \bsigma_h)\in S_{1,0}\times S_{1}^d$  be the solution of the $L^2$-LSFEM problem \eqref{LSFEM_L2}, then there exists a constant $C>0$ independent of the mesh size $h$, such that
\begin{align}
\tri (u-u_h, \bsigma-\bsigma_h)\tri_0  \leq C  h \|u\|_{3}.
\end{align}
\end{thm}
\begin{proof}
We have
\beq \label{local_ineq_L2}
C\tri (v,\btau) \tri_0 \leq  \|A\ratio  \nabla \btau\|_{0} +\|\nabla v\|_0 + \|\btau\|_0. 
\eeq
Then let $v_h$ be the interpolation of $u$ in $S_{1,0}$ and $\btau_h$ be the interpolation of $\nabla u$ in $S_{1}^d$, by the approximation properties of $S_1$ and the fact that $\bsigma = \nabla u$, we have 
\begin{eqnarray*}
C\tri (u-u_h,\bsigma-\bsigma_h) \tri_0 &\leq & \|A\ratio  \nabla (\bsigma-\btau_h)\|_{0} + \|\bsigma-\btau_h\|_0 +\|\nabla (u-v_h)\|_0 \\
	&\leq &Ch\|\nabla u\|_2 + Ch\|u\|_2 +Ch^2\|\nabla u\|_2 \leq C h\|u\|_3.
\end{eqnarray*}
The theorem is proved.
\end{proof}

\begin{rem}
From the a priori estimates, we can clearly see that with respect the least-squares norm,  the weighted version is optimal when the regularity is high and a suitable high order finite element pair is used. For the $L^2$-LSFEM, the optimal interpolation order for the $L^2$-norm of $\bsigma$ is $2$, which is one order high that the other two components, and thus sub-optimal, thus high order approximations of the $L^2$-LSFEM is not suggested. But the $L^2$-LSFEM can use the simplest linear conforming finite element space for $u$, and has reasonable approximation orders, for example, we will find the $L^2$-estimates of $u-u_h$ is the same order as the the $S_{2,0}\times S_{1}^d$ weighted-LSFEM if assuming enough smoothness of the coefficient $A$ and the solution, see Theorems \ref{L2_h} and  Theorems \ref{L2_0} and our numerical tests. 
\end{rem}

\begin{rem}
In the a priori error estimates, in order to get a convergence order, we assume that the regularity of $u$ is at least $H^2$ or $H^3$. To discuss the convergence without a high regularity, we first assume that the coefficient matrix $A$ is nice enough such that for any $\epsilon >0$, there exists a $\bsigma^{\epsilon} \in C^{\infty}(\O)^d$, such that 
\beq \label{eps}
\|\bsigma -\bsigma^{\epsilon}\|_0 \leq \epsilon \quad \mbox{and}\quad \|f-A\ratio  \nabla \bsigma^{\epsilon}\|_0 \leq \epsilon.
\eeq
Note that this condition is weaker than $\|\bsigma - \bsigma^{\epsilon}\|_1 \leq \epsilon$, since $\nabla \bsigma$ may not be in $L^2$. If the assumption \eqref{eps} is true, then a convergence result is easy to prove for $u \in H^1_0(\O)$ only. 

For the special case that $A|_K$, the restriction of the coefficient matrix on each $K\in\cT$, is a constant matrix, then $A\ratio \nabla \btau_h$ contains a piecewise constant on each element $K$, we have 
$$
\|f-A\ratio \nabla \btau_h \|_{0,K} \leq C h_K^r \|f\|_{r,K}, \quad \mbox{for  } 0<r\leq 1.
$$
This can be used to establish the convergence for a solution with a low regularity.
   
For the standard LSFEM for the divergence, such problems do not exist since Raviart-Thomas element is used to approximation $\bsigma_d = -A\nabla u$, and the regularity requirement is on $f =\gradt \bsigma_d$, which is weaker than the requirement on $D^2 u$.  Again, this is due to the special structure of the divergence equation, which is not available for the general non-divergence equation.    
\end{rem}

\subsection{A posteriori error estimates}
The least-squares functional can be used to define the following fully computable a posteriori  local indicator and global error estimator.

\subsubsection{Weighted-LSFEM}
Let $(u, \bsigma)$ be the solution of least-squares variational problem \eqref{LS}, and $(u_h, \bsigma_h)\in S_{k,0}\times S_{k-1}^d$, $k\geq 2$ be the solution of the weighted-LSFEM problem \eqref{LSFEM_h}, then define:
\begin{eqnarray*}
\eta_{h,K}^2 &:=&  h_K^2\|f+A\ratio  \nabla \bsigma_h\|_{0,K}^2 +\|\bsigma_h-\nabla u_h\|_{0,K}^2,
\quad \forall K\in \cT, \\[2mm]
\mbox{and} \quad \eta_h^2 &:=& \sum_{K\in\cT}\eta_{h,K}^2 =  \sum_{K\in\cT}h_K^2\|f+A\ratio  \nabla \bsigma_h\|_{0,K}^2 +\|\bsigma_h-\nabla u_h\|_0^2.
\end{eqnarray*}

\begin{thm} \label{poster_h}
The a posteriori error estimator $\eta_h$ is exact with respect to the least-squares norm $\tri \cdot\tri_h$:
$$
\eta_h = \tri (u-u_h, \bsigma-\bsigma_h) \tri_h \quad \mbox{and} \quad \eta_{h,K} = \tri (u-u_h, \bsigma-\bsigma_h) \tri_{h,K}.
$$
The following local efficiency bound is also true with a constant $C>0$ independent of the mesh size $h$:
\beq \label{eff_h}
C \eta_{h,K} \leq h_K\|A :\nabla(\bsigma-\bsigma_h)\|_{0,K} + \|\bsigma-\bsigma_h\|_{0,K} + \|\nabla(u-u_h)\|_{0,K}, \quad \forall K\in\cT.
\eeq
\end{thm}
\begin{proof}
Using $\bsigma -\nabla u =0$ and $f=-A\ratio  \nabla \bsigma$, we obtain,
\begin{eqnarray*}
\eta_h^2 &=&
\sum_{K\in\cT}h_K^2\|A\ratio  \nabla (\bsigma-\bsigma_h)\|_{0,K}^2 +\|\bsigma_h - \bsigma -\nabla (u_h -u)\|_0^2
=\tri (u-u_h, \bsigma-\bsigma_h) \tri_h^2.
\end{eqnarray*}
The proof of the local exactness is identical.

The locally efficiency \eqref{eff_h} is a direct result of a local version of  \eqref{local_ineq_h}.
\end{proof}

\subsubsection{$L^2$-LSFEM}
For the $L^2$-LSFEM, the a posteriori error estimator can be defined accordingly, and the corresponding results can be proved in a similar fashion.

Let $(u, \bsigma)$ be the solution of least-squares variational problem \eqref{LS}, and $(u_h, \bsigma_h)\in S_{1,0}\times S_{1}^d$ be the solution of the $L^2$-LSFEM problem \eqref{LSFEM_L2}, define:
$$
\eta_{0,K}^2 :=  \|f+A\ratio  \nabla \bsigma_h\|_{0,K}^2 +\|\bsigma_h-\nabla u_h\|_{0,K}^2,
\quad \forall K\in \cT,
$$
and
$$
\eta_0^2 := \sum_{K\in\cT}\eta_{0,K}^2 =  \|f+A\ratio  \nabla \bsigma_h\|_{0,K}^2 +\|\bsigma_h-\nabla u_h\|_0^2.
$$
\begin{thm}
The a posteriori error estimator $\eta_0$ is exact with respect to the least-squares norm $\tri \cdot\tri_0$:
$$
\eta_0 = \tri (u-u_h, \bsigma-\bsigma_h) \tri_0 \quad \mbox{and} \quad \eta_{0,K} = \tri (u-u_h, \bsigma-\bsigma_h) \tri_{0,K}.
$$
The following local efficiency bound is also true with $C>0$ independent of the mesh size $h$:
\beq \label{eff_0}
C \eta_{0,K} \leq \|A :\nabla(\bsigma-\bsigma_h)\|_{0,K} + \|\bsigma-\bsigma_h\|_{0,K} + \|\nabla(u-u_h)\|_{0,K}, \quad \forall K\in\cT.
\eeq
\end{thm}

\begin{rem}
For all the a priori and a posteriori results in this section, we only assume Assumption \ref{assump_EU} is true, $A\in L^{\infty}(\O)^{d\times d}$, but $A$ can be possibly degenerate.
\end{rem}
\section{More A Priori Error Estimates for the Weighted-LSFEM}

\subsection{$A$ and $h$-weighted broken $H^2$-norm estimate}
For the weighted-LSFEM, $k\geq 2$ polynomial spaces are used to approximate $u$, and due to the non-divergence structure of the equation, the weighted-LSFEM can also be viewed as a method via an approximation of the $D^2$ operator. Thus, in this subsection, we derive an $A$ and $h$-weighted broken $H^2$-norm estimate of the numerical solution.

We use the following notation to denote a mesh-dependent norm:
$$
\|h v\|_0 := \left(\sum_{K\in\cT} h_K^2 \|v\|_{0,K}^2\right)^{1/2}, \quad \forall v\in L^2(\O).
$$
For example, 
$\|hA\ratio D_h^2 v_h\|_0 = 
\left(\sum_{K\in\cT} h_K^2 \|A\ratio  D^2 v_h\|_{0,K}^2\right)^{1/2}
$.

In two dimensions, we define $\tilde{V}_{k+2,2d}$ to be the $C^1$-conforming finite element space of degree $k+2$, $k\geq 2$ on $\cT$, which is the high-order version of the classical Hsieh-Clough-Tocher macro-element \cite{DDPS:79}.
In three dimensions, let $\tilde{V}_{h,3d}$ be the classical  $C^1$-conforming piecewise cubic Hsieh-Clough-Tocher macro-element space associated with the mesh $\cT$, see \cite{Ciarlet:78,NW:19}.

The proof of the following lemma can be found in \cite{BGS:10,GHV:11} for the two dimensional case and in \cite{NW:19} for the  three dimensional case. Although in these papers, the result are all presented in the global setting, a careful look into their proofs will find  the result is true locally due to the shape regularity assumption.
\begin{lem} \label{E_h} For any function $v_h \in S_{k,0}$ with $2\leq k \leq 3$ if the space dimension $d=3$, and $k\geq 2$ for $d=2$, there exists an averaging linear map $E_h: S_k \rightarrow \tilde{V}_{k+2,2d}\cap H_0^1(\O)$ in two dimensions and $E_h: S_k \rightarrow \tilde{V}_{h,3d}\cap H_0^1(\O)$ in three dimensions, the following estimate is true:
\beq \label{HCTinterpolation}
\|v_h - E_h v_h\|_{0,K} \leq C h_K^{3/2} \sum_{F\in \cE_K}\|\jump{\nabla v_h\cdot \bn}\|_{0,F}, \quad \forall K\in \cT,
\eeq
where $\cE_K$ is the collection of interior edges on elements that shares a common nodes with the element $K$ in two dimensions, and is the collection of interior faces on elements that shares a common vertex or a common edge middle point with the element $K$ in three dimensions.
\end{lem}

\begin{lem}
There exists a constant $C>0$ independent of the mesh size $h$, such that, for any  $(v_h, \btau_h)\in S_{k,0}\times S_{k-1}^d$, with $2\leq k \leq 3$ if the space dimension $d=3$, and $k\geq 2$ for $d=2$, the following estimates are true:
\beq \label{ineq_vp}
\|\nabla(v_h - E_h v_h)\|_{0} \leq C \|\btau_h-\nabla v_h\|_{0}, \quad
\|D_h^2(v_h - E_h v_h)\|_{0} \leq C h^{-1}\|\btau_h-\nabla v_h\|_{0},
\eeq
\beq \label{ineq_vp2}
\mbox{and}\quad \|\btau_h - \nabla( E_h v_h)\|_{0} \leq C \|\btau_h-\nabla v_h\|_{0}.
\eeq
\end{lem}
\begin{proof}
To prove \eqref{ineq_vp}, since $\btau_h \in (H^1(\O))^d$, we have $\jump{\btau_h\cdot\bn}_F =0$ on an interior edge(2D)/face(3D) $F$, then by the discrete trace inequality:
$$
\|\jump{\nabla v_h\cdot \bn}\|_{0,F} = \|\jump{\nabla v_h-\btau_h}\cdot \bn\|_{0,F}
\leq  C h_F^{-1/2} \|\nabla v_h-\btau_h\|_{0,\o_F},
$$
where $\omega_F$ is the collection of two elements that share the common $F$. Thus, by \eqref{HCTinterpolation}, we have
\beq
\|v_h - E_h v_h\|_{0,K} \leq C h_K \sum_{F\in\cE_K} \|\nabla v_h-\btau_h\|_{0,\o_F}.
\eeq
Combined with the inverse inequalities $\|\nabla(v_h - E_h v_h)\|_{0,K} \leq Ch_K^{-1}\|v_h - E_h v_h\|_{0,K}$ , $\|D_h^2 (v_h - E_h v_h)\|_{0,K} \leq Ch_K^{-2}\|v_h - E_h v_h\|_{0,K}$, and the shape regularity of the mesh $\cT$, \eqref{ineq_vp} is proved.

The inequality of \eqref{ineq_vp2} is a simple consequence of the first inequality of \eqref{ineq_vp} and the triangle inequality. 
\end{proof}

\begin{lem}\label{lem_ineq}
The following inequality holds for all $(v_h, \btau_h)\in S_{k,0}\times S_{k-1}^d$, with $2\leq k \leq 3$ if the space dimension $d=3$, and $k\geq 2$ for $d=2$:
\beq \label{ineq_lem36}
\|hA\ratio D_h^2 v_h\|_0 
\leq C(\|hA\ratio \nabla \btau_h\|_0 + \|\btau_h -\nabla v_h\|_0) \leq C\tri(v_h,\btau_h)\tri_h.
\eeq
\end{lem}
\begin{proof}
Let $\tilde{v}_h = E_hv_h$, by the triangle inequality, we have
\beq
 \|hA\ratio D_h^2 v_h\|_0 \leq  \|hA\ratio D_h^2 (v_h-\tilde{v}_h)\|_0 + \|hA\ratio D^2 \tilde{v}_h\|_0 .
\eeq
For the term $\|hA\ratio D_h^2 (v_h-\tilde{v}_h)\|_0$, by the inverse estimate, the fact that $A \in L^{\infty}(\O)^{d\times d}$, and the second inequality in \eqref{ineq_vp}, 
we have
\beq \label{vh-tvh}
\|hA\ratio D_h^2 (v_h-\tilde{v}_h)\|_0  \leq C \|\nabla (v_h-\tilde{v}_h)\|_0\leq C\|\btau_h-\nabla v_h\|_{0}.
\eeq
On the other hand, by the triangle inequality, 
$$
 \|hA\ratio D^2 \tilde{v}_h\|_0  \leq \|hA\ratio \nabla \btau_h\|_0	+ \|hA\ratio \nabla (\btau_h-\nabla  \tilde{v}_h)\|_0.
$$
By the inverse estimate, the fact that  $A \in L^{\infty}(\O)^{d\times d}$, and the inequality \eqref{ineq_vp2}, 
we have,
$$
\|hA\ratio \nabla (\btau_h-\nabla  \tilde{v}_h)\|_0
\leq C \|\btau_h-\nabla \tilde{v}_h\|_0 \leq C\|\btau_h-\nabla v_h\|_0.
$$
Combined the above results, we prove the first inequality of the lemma. The second inequality is a consequence of a simple calculation.
\end{proof}

\begin{thm} \label{thm_apriori1}
Assume the solution 
$u\in H^{r+1}(\O)$, for some $r>1$, and $(u_h, \bsigma_h) \in S_{k,0}\times S_{k-1}^d$, with $2\leq k \leq 3$ if the space dimension $d=3$, and $k\geq 2$ for $d=2$, be the solution of the weighted-LSFEM \eqref{LSFEM_h}, then there exists a constant $C>0$ independent of the mesh size $h$, such that the following error estimate is true:
\beq \label{cnorm}
\|hA\ratio D_h^2 (u-u_h)\|_0 = \left(\sum_{K\in\cT} h_K^2 \|A\ratio  D^2 (u-u_h)\|_{0,K}^2\right)^{1/2} \leq C h^{\min(k,r+1)} \|u\|_{\min(k,r+1)}.
\eeq
\end{thm}
\begin{proof}
By the triangle inequality, for an arbitrary $v_h \in S_{k,0}$, 
$$
\|hA\ratio D_h^2 (u-u_h)\|_0 \leq \|hA\ratio D_h^2 (u-v_h)\|_0 +\|hA\ratio D_h^2 (u_h-v_h)\|_0.
$$
By \eqref{ineq_lem36}, for an arbitrary $\btau_h \in S_{k-1}^d$, we have 
$$
C\|hA\ratio D_h^2 (u_h-v_h)\|_0 \leq \tri(u_h-v_h,\bsigma_h-\btau_h)\tri_h.
$$
On the other hand, by the orthogonality result
$$
a_h((u-u_h,\bsigma-\bsigma_h), (u_h-v_h,\bsigma_h-\btau_h))=0, \quad \forall (v_h,\btau_h) \in S_{k,0}\times S_{k-1}^d,
$$
we have
\begin{eqnarray*}
\tri(u_h-v_h,\bsigma_h-\btau_h)\tri_h^2 &= &a_h((u_h-v_h,\bsigma_h-\btau_h), (u_h-v_h,\bsigma_h-\btau_h)) \\
&=&  a_h((u-v_h,\bsigma-\bsigma_h), (u_h-v_h,\bsigma_h-\btau_h)) \\
&\leq& \tri(u-v_h,\bsigma-\bsigma_h)\tri_h \tri(u_h-v_h,\bsigma_h-\btau_h)\tri_h.
\end{eqnarray*}
Thus, 
$$
\tri(u_h-v_h,\bsigma_h-\btau_h)\tri_h \leq \tri(u-v_h,\bsigma-\bsigma_h)\tri_h, \quad \forall (v_h,\bsigma_h) \in S_{k,0}\times S_{k-1}^d.
$$
Combined the above results, we have
$$
\|hA\ratio D_h^2 (u-u_h)\|_0  \leq \inf_{(v_h,\bsigma_h)\in S_{k,0}\times S_{k-1}^d} ( \|hA\ratio D_h^2 (u-v_h)\|_0+ C\tri(u-v_h,\bsigma-\bsigma_h)\tri_h).
$$
By the approximation properties of functions in $S_k$, we have 
$$
\|hA\ratio D_h^2 (u-u_h)\|_0  \leq C h^{\min(k,r+1)} \|u\|_{\min(k,r+1)} .
$$
\end{proof}

\subsection{Error estimates based on the assumption of the non-divergence operator}

In this subsection, we prove error estimates based on the following assumption on the non-divergence operator.
\begin{assumption}
{\bf (Assumption on the non-divergence operator)} We assume that for the given $A$ and the domain $\O$, the following result holds for the operator $A\ratio D^2 v$ in non-divergence form:
\beq \label{regularity}
\|v\|_{1+\delta,\O} \leq  C \|A\ratio D^2 v\|_{0,\O}, \quad  \forall v\in V, \mbox{ for some } \delta \in [0,1],
\eeq
where the space $V$ is defined in \eqref{defV}.

For a special case $\delta =1$, it is
\beq \label{H2regularity}
\|v\|_{2,\O} \leq  C \|A\ratio D^2 v\|_{0,\O}, \quad \forall v\in H_0^1(\O)\cap H^2(\O).
\eeq
\end{assumption}
\begin{rem}
The strong assumption \eqref{H2regularity} is widely used in the proofs of stability of numerical methods and convergence analysis in the papers of \cite{FHN:17, FNS:18, WW:18}. 

To guarantee that \eqref{H2regularity} is true, we first assume that $\O$ is  a bounded convex domain, then it is known that the following Miranda-Talenti inequality holds
$$
\|D^2 v\|_0 \leq \|\Delta v\|_0, \quad \forall v\in H_0^1(\O)\cap H^2(\O).
$$
If we further assume that the matrix $A$ satisfies the so-called Cordes condition: there exists some $\epsilon \in (0,1]$ such that
\beq \label{Cordes}
|A|^2/(\mbox{tr}(A))^2 \leq 1/(d-1+\epsilon),
\eeq
then the strong assumption \eqref{H2regularity} can be proved.

For the case $d=2$, uniformly elliptic of $A$ implies the Codes condition, while for $d=3$, the PDE in non-divergence form may be ill-posed due to the lack of the condition. In \cite{Gal:17}, it is shown that for the bounded convex domain $\O$, \eqref{H2regularity} is true, with the constant $C$ depending on the matrix $A$ and the Codes condition.

In this paper, we assume the domain can be non-convex, thus only a weaker assumption on the operator \eqref{regularity} holds.
\end{rem}
\begin{thm} \label{thm_htc_inter}
Assume that the assumption of the operator \eqref{regularity} is true, the mesh is quasi-uniform,  $(u_h, \bsigma_h) \in S_{k,0}\times S_{k-1}^d$, with $k= 2$ or $3$ if the space dimension $d=3$, and $k \geq 2$ for $d=2$, is the weighted-LSFEM solution, then for $u\in H^{1+r}(\O)$, $r\geq 1$, we have
\beq
\|u -E_h u_h\|_{1+\delta} \leq  C h^{\min(k-1,r)} \|u\|_{\min(k,r+1)}, \quad \mbox{for some } \delta \in [0,1].
\eeq
Specifically, for the weakest case $\delta =0$, we have
\beq
\|u -E_h u_h\|_{1} \leq  C h^{\min(k-1,r)} \|u\|_{\min(k,r+1)}.
\eeq
and for the strongest case, $\delta =1$,
\beq
\|u -E_h u_h\|_{2} \leq  C h^{\min(k-1,r)} \|u\|_{\min(k,r+1)}.
\eeq
Here, $E_h$ is the HTC-element averaging operator defined in Lemma \ref{E_h}.
\end{thm}
\begin{proof}
Let $\tilde{u}_h = E_h u_h$, then by the assumption of the operator and the triangle inequality,
\begin{eqnarray*}
C\|u -\tilde{u}_h\|_{1+\delta,\O} \leq \|A\ratio D^2(u -\tilde{u}_h)\|_{0}
 \leq \|A\ratio D_h^2(u -u_h)\|_{0}+\|A\ratio D_h^2(u_h -\tilde{u}_h)\|_{0}.
\end{eqnarray*}
For the first term on the righthand side, by Theorem \ref{thm_apriori1}, $k\geq 2$, and the fact the mesh is quasi-uniform, we have 
$$
\|A\ratio D_h^2(u -u_h)\|_{0,\O} \leq C h^{\min(k-1,r)} \|u\|_{\min(k,r+1)}.
$$
For the second term, by the same argument as in \eqref{vh-tvh}, 
\begin{eqnarray*}
\|A\ratio D^2(u_h -\tilde{u}_h)\|_{0,\O} &\leq & C h^{-1}\|\bsigma_h - \nabla u_h\|_0.
\end{eqnarray*}
Then by the fact that $\bsigma = \nabla u$, and Theorem \ref{thm_LSFEM_apriori_h},
\begin{eqnarray*}
\|\bsigma_h - \nabla u_h\|_0 & \leq & \|\bsigma_h - \bsigma + \nabla (u_h-u)\|_0 \leq   \tri(u-v_h,\bsigma-\bsigma_h)\tri_h \\
&\leq &   C h^{\min(k,r+1)} \|u\|_{\min(k,r+1)},.
\end{eqnarray*}
Combined the above estimates, we proved the theorem.
\end{proof}

\begin{thm} \label{thm_apriori_codes}
Assume that the assumption of the operator \eqref{regularity} is true, the mesh is quasi-uniform, the solution 
$u\in H^{r+1}(\O)$, for some $r>1$, and the numerical solution $(u_h, \bsigma_h) \in S_{k,0}\times S_{k-1}^d$, with $2\leq k \leq 3$ if the space dimension $d=3$, and $k\geq 2$ for $d=2$, be the solution of the weighted-LSFEM \eqref{LSFEM_h}, then we also have the following $H^1$-estimate:
\beq \label{H1error}
\|\nabla (u-u_h)||_0 \leq C h^{\min(k-1,r)} \|u\|_{\min(k,r+1)}.
\eeq
If we further assume that \eqref{H2regularity} is true, then the following broken $H^2$-norm estimate is also true:
\beq \label{brokenH2}
\|D^2_h(u-u_h)||_0 \leq C h^{\min(k-1,r)} \|u\|_{\min(k,r+1)}.
\eeq
\end{thm}
\begin{proof}
Let $\tilde{u}_h = E_h u_h$, then by the triangle inequality, 
$$
\|\nabla (u -u_h)\|_{0,\O} \leq \|\nabla (u -\tilde{u}_h)\|_{0,\O}+\|\nabla (\tilde{u}_h -u_h)\|_{0,\O}.
$$
The first term is good by Theorem \ref{thm_htc_inter}. For the second term, by \eqref{ineq_vp} and the a priori error estimate result of Theorem \ref{thm_LSFEM_apriori_h},
\begin{eqnarray*}
\|\nabla (\tilde{u}_h -u_h)\|_{0,\O} &\leq & C\|\bsigma_h-\nabla u_h\|_0 \leq C \tri (u-u_h,\bsigma-\bsigma_h) \tri_h \\ &\leq &  C h^{\min(k,r+1)} \|u\|_{\min(k,r+1)}.
\end{eqnarray*}
Then we prove \eqref{H1error}.

For  \eqref{brokenH2}, similarly, we have
$$
\|D_h^2(u -u_h)\|_{0,\O} \leq \|D^2(u -\tilde{u}_h)\|_{0,\O}+\|D_h^2(\tilde{u}_h -u_h)\|_{0,\O},
$$
and
\begin{eqnarray*}
\|D_h^2(\tilde{u}_h -u_h)\|_{0,\O} &\leq & C h^{-1} \|\bsigma_h-\nabla u_h\|_0 \leq C h^{-1} \tri (u-u_h,\bsigma-\bsigma_h) \tri_h \\ &\leq &  C h^{\min(k-1,r)} \|u\|_{\min(k,r+1)}.
\end{eqnarray*}
The result \eqref{brokenH2} then can be proved by combining the estimates of $\|D^2(u -\tilde{u}_h)\|_{0,\O}$ and $\|D_h^2(\tilde{u}_h -u_h)\|_{0,\O}$.
\end{proof}
\begin{rem}
The $H^1$-error estimate \eqref{H1error} is of course not optimal in the approximation order, but its requirement on the operator and the domain is much weaker than the strong assumption \eqref{H2regularity}. For example, assume that $\O$ is an L-shaped domain and $A\ratio D^2 v = \Delta v$, the result \eqref{H1error} shows that we still have convergence in $H^1$-norm under a weaker operator regularity assumption. 
\end{rem}


\subsection{$L^2$ error estimate}
In this subsection, we discuss the $L^2$-error estimate of the weighted-LSFEM with extra regularity conditions of the equation. The proof is based on a modification of the argument of Cai and Ku \cite{CK:06} for the LSFEM of the elliptic equations in divergence form. The existence of the weight $h$ in the weighted-LSFEM adds extra difficulties to the analysis and requires the polynomial degree to approximate $u$ is at least three. 
 
Denote by 
$$
E = \bsigma -\bsigma_h \quad\mbox{and}\quad e = u-u_h
$$
the respective errors of the gradient and the solution.

Assume that $A$ is smooth enough that the operation $\gradt (\gradt (Az))$ is meaningful for a smooth $z$, where $Az$ is a matrix with items $a_{i,j}z$, and the divergence of a matrix $B$ is a column vector with each item being the divergence of the row of $B$. 

Let $z\in H^1_0(\O)$ be the solution of the following equation:
\beq \label{eq:dual}
\sum_{i,j=1}^d \p_{ij}^2 (a_{i,j} u)=\gradt (\gradt (Az)) = e \mbox{ in }\O, \quad z =0\mbox{ on }\p\O.
\eeq
We assume that both the original non-divergence PDE \eqref{eq_nondiv} and the dual equation \eqref{eq:dual} satisfy the full $H^2$-regularity:
\beq \label{reg_H2}
\|u\|_2 \leq C \|f\|_0 \quad \mbox{and}\quad \|z\|_2 \leq C \|e\|_0.
\eeq
In addition, we also assume that the solution of \eqref{eq_nondiv} satisfies the following stronger regularity assumption:
\beq \label{reg_H4}
\|u\|_4 \leq C \|f\|_2.
\eeq

We should note that in order to make $\gradt (\gradt (Az))$ is well-defined does not require that $A \in C^1(\O)^{d\times d}$, but only $z\in \{v \in L^2(\O):  Av \in H(\divvr;\O)^d, \gradt (Av) \in H(\divvr;\O)\}$. But of course, the high regularity assumptions probably need the smoothness of $A$.

\begin{thm} \label{L2_h} Assuming that the mesh is quasi-uniform with a mesh-size $h$, the regularity assumptions \eqref{reg_H2} and \eqref{reg_H4} are true, and the  weighted-LSFEM solutions $(u_h, \bsigma_h)$ belongs to $S_{k,0}\times S_{k-1}^d$, for $k\geq 3$, we have the following $L^2$-error estimate:
\beq
\|u-u_h\|_0 \leq C h \tri (u-u_h,\bsigma-\bsigma_h) \tri_h.
\eeq
\end{thm}
\begin{proof}
Using the integration by parts, we have
\begin{eqnarray*}
\|e\|_0^2 &=& (e,e) = (e, \gradt (\gradt (Az))) = - (\nabla e, \gradt (Az)) = (E-\nabla e, \gradt (Az))-(E, \gradt (Az)) \\
&=& (E-\nabla e, \gradt (Az))+(\nabla E, Az) = (E-\nabla e, \gradt (Az))+(A\ratio \nabla E, z). 
\end{eqnarray*}
To match with the bilinear form of the weighted-LSFEM, we introduce two sub-auxiliary problems for $w_1\in H_0^1(\O)$ and $w_2\in H_0^1(\O)$:
\beq \label{subproblems_h}
\left\{
\begin{array}{lll}
\bphi_1 -\nabla w_1 &=& 0, \\
A\ratio  \nabla \bphi_1 &=& h^{-2}z,
\end{array}
\right.
\quad \mbox{and} \quad
\left\{
\begin{array}{lll}
\bphi_2 -\nabla w_2 &=& \gradt(A z), \\
A\ratio  \nabla \bphi_2 &=& 0,
\end{array}
\right.
\eeq
which are, in the PDE forms:
\beq \label{subPDE_h}
A\ratio D^2 w_1 = h^{-2}z \quad \mbox{and}\quad A\ratio D^2 w_2 = - A\ratio \nabla (\gradt (Az)).
\eeq
Let $w=w_1+w_2$ and $\bphi=\bphi_1 + \bphi_2$, then
\beq
\left\{
\begin{array}{rlll}
\bphi -\nabla w &=& \gradt(A z), &\mbox{in  } \O\\
A\ratio  \nabla \bphi &=& h^{-2}z, &\mbox{in  } \O,\\
u&=&0 &\mbox{on  } \p\O.
\end{array}
\right.
\eeq
Substitute \eqref{subproblems_h} into the representation of $\|e\|_0^2$: 
\begin{eqnarray*}
\|e\|_0^2
&=&  (E-\nabla e, \gradt (Az))+(A\ratio \nabla E, z) \\
&=&  (E-\nabla e, \bphi_1 -\nabla w_1)+\sum_{K\in\cT}(h_K^2 A\ratio \nabla E, A\ratio \nabla \bphi_1)_K \\[2mm] 
&& + (E-\nabla e, \bphi_2 -\nabla w_2)+\sum_{K\in\cT}(h_K^2 A\ratio \nabla E, A\ratio \nabla \bphi_2)_K \\[2mm]
&=& a_h((e,E), (w_1, \bphi_1)) + a_h((e,E), (w_2, \bphi_2)).
\end{eqnarray*}
Let $(w_{i,h}, \bphi_{i,h}) \in S_{k,0}\times S_{k-1}^d$, for $k\geq 3$ and $i=1$ and $2$ and use the orthogonality of the error equation, we have
\begin{eqnarray}\nonumber 
\|e\|_0^2
&= &a_h((e,E), (w_1-w_{1,h}, \bphi_1-\bphi_{1,h})) + a_h((e,E), (w_2-w_{2,h}, \bphi_2-\bphi_{2,h}))\\
\label{e21h}
&\leq& \tri (e,E) \tri_h (\tri (w_1-w_{1,h}, \bphi_1-\bphi_{1,h}) \tri_h + \tri (w_2-w_{2,h}, \bphi_2-\bphi_{2,h}) \tri_h ).
\end{eqnarray}
By the approximation properties of $S_{k,0}\times S_{k-1}^d$ with $k\geq 3$,
$$
\inf_{(w_{1,h}, \bphi_{1,h}) \in S_{k,0}\times S_{k-1}^d}\tri (w_1-w_{1,h}, \bphi_1-\bphi_{1,h}) \tri_h \leq C h^3(\|w_1\|_{4}+\|\bphi_1\|_3 + \|A\ratio \nabla \bphi_1\|_2).
$$
Since $A\ratio \nabla \bphi_1 = h^{-2}z$, we have
$$
\|A\ratio \nabla \bphi_1\|_2 =  h^{-2}\|z\|_2.
$$
By the regularity assumptions \eqref{reg_H4}, we have
$$
\|\bphi_1\|_3 \leq C \|w_1\|_4  \leq C h^{-2} \|z\|_2.
$$
Combined with the regularity assumption that $\|z\|_2 \leq C \|e\|_0$, we have 
\beq \label{e22h}
\inf_{(w_{1,h}, \bphi_{1,h}) \in S_{k,0}\times S_{k-1}^d}\tri (w_1-w_{1,h}, \bphi_1-\bphi_{1,h}) \tri_h \leq C h \|e\|_0.
\eeq
For the $w_2$ and $\bphi_2$ term, using the approximation properties and the fact $A\ratio \nabla \bphi_2 =0$, we have 
$$
\inf_{(w_{2,h}, \bphi_{2,h}) \in S_{k,0}\times S_{k-1}^d} \tri (w_2-w_{2,h}, \bphi_2-\bphi_{2,h}) \tri_h \leq Ch(\|w_2\|_{2}+\|\bphi_2\|_1).
$$
By the PDE form \eqref{subPDE_h} and using the regularity assumption for the non-divergence PDE \eqref{reg_H2} for two times,
$$
\|w_2\|_2 \leq C \|A\ratio \nabla (\gradt (Az))\|_0 \leq C \|z\|_2 \leq C \|e\|_0,
$$
and by the fact $\bphi_2 = \nabla w_2 + \gradt(Az)$, 
$$
\|\bphi_2\|_1 \leq  C (\|w_2\|_2 + \|\gradt(Az)\|_1) \leq C (\|w_2\|_2 + \|z\|_2)\leq  C \|e\|_0.
$$
Combined the results, we have
\beq \label{e23h}
\inf_{(w_{2,h}, \bphi_{2,h}) \in S_{k,0}\times S_{k-1}^d}\tri (w_2-w_{2,h}, \bphi_2-\bphi_{2,h}) \tri \leq C h \|e\|_0.
\eeq
From \eqref{e21h}, \eqref{e22h}, and \eqref{e23h}, we have
\begin{eqnarray}
\|e\|_0^2 \leq C h \tri (e,E) \tri_h \|e\|_0.
\end{eqnarray}
The theorem is proved.
\end{proof}
\begin{rem}
The result of this theorem requires some high regularity and at least degree three polynomial approximation for $u$. From the numerical experiments, we do find that this degree $3$ requirement is necessary. The proof of the result can be generalized to other $h$-weighted LSFEMs.  
\end{rem}

\section{More A Priori Error Estimates for the $L^2$-LSFEM}
In this section, we discuss the $L^2$-error estimation of the $L^2$-LSFEM with a standard $H^2$-regularity assumoption.  The proof is also based on modification of the proof of Cai and Ku \cite{CK:06} but is simpler than that of the weighted-LSFEM. With the $L^2$-error estimate available, we discuss the $H^1$-norm estimate with the same assumption. 

\begin{thm} \label{L2_0} Assuming that the mesh is quasi-uniform with a mesh-size $h$, the regularity assumptions \eqref{reg_H2} is true, and $(u_h, \bsigma_h)$ is the $S_{1,0}\times S_{1}^d$ $L^2$-LSFEM solution, we have the following $L^2$-  and $H^1$-error estimates:
\begin{eqnarray}
\|u-u_h\|_0 &\leq & C h \tri (u-u_h,\bsigma-\bsigma_h) \tri_0,\\[2mm]
\mbox{and}\quad
\|\nabla(u-u_h)\|_0 &\leq& C  \tri (u-u_h,\bsigma-\bsigma_h) \tri_0 + Ch \|u\|_2.
\end{eqnarray}
\end{thm}
\begin{proof}
We have the same error representation as in the weighted-LSFEM case,
\begin{eqnarray*}
\|e\|_0^2 = (E-\nabla e, \gradt (Az))+(A\ratio \nabla E, z). 
\end{eqnarray*}
To match with the bilinear form of $L^2$-LSFEM, we also introduce two sub-auxiliary problems for $w_1\in H_0^1(\O)$ and $w_2\in H_0^1(\O)$:
\beq \label{subproblems}
\left\{
\begin{array}{lll}
\bphi_1 -\nabla w_1 &=& 0, \\
A\ratio  \nabla \bphi_1 &=& z,
\end{array}
\right.
\quad \mbox{and} \quad
\left\{
\begin{array}{lll}
\bphi_2 -\nabla w_2 &=& \gradt(A z), \\
A\ratio  \nabla \bphi_2 &=& 0,
\end{array}
\right.
\eeq
which are, in the PDE forms:
\beq \label{subPDE}
A\ratio D^2 w_1 = z \quad \mbox{and}\quad A\ratio D^2 w_2 = - A\ratio \nabla (\gradt (Az)).
\eeq
Let $w=w_1+w_2$ and $\bphi=\bphi_1 + \bphi_2$, and
substitute \eqref{subproblems} into the representation of $\|e\|_0^2$:
\begin{eqnarray*}
\|e\|_0^2
&=& a_0((e,E), (w_1, \bphi_1)) + a_0((e,E), (w_2, \bphi_2)).
\end{eqnarray*}
Let $(w_{i,h}, \bphi_{i,h}) \in S_{1,0}\times S_{1}^d$, for $i=1$ and $2$ and use the orthogonality of the error equation, we have
\begin{eqnarray}\nonumber
\|e\|_0^2
&= &a_0((e,E), (w_1-w_{1,h}, \bphi_1-\bphi_{1,h})) + a_0((e,E), (w_2-w_{2,h}, \bphi_2-\bphi_{2,h}))\\
\label{e21_L2}
&\leq& \tri (e,E) \tri_0 (\tri (w_1-w_{1,h}, \bphi_1-\bphi_{1,h}) \tri_0 + \tri (w_2-w_{2,h}, \bphi_2-\bphi_{2,h}) \tri_0).
\end{eqnarray}
By the approximation properties of $S_{1,0}\times S_{1}^d$,
$$
\inf_{(w_{1,h}, \bphi_{1,h}) \in S_{1,0}\times S_1^d} \tri (w_1-w_{1,h}, \bphi_1-\bphi_{1,h}) \tri_0 \leq C h(\|w_1\|_{2}+\|\bphi_1\|_1 + \|A\ratio \nabla \bphi_1\|_1).
$$
Since $A\ratio \nabla \bphi_1 = z$, we have
$$
\|A\ratio \nabla \bphi_1\|_1 \leq \|A\ratio \nabla \bphi_1\|_2 =  \|z\|_2 \leq C \|e\|_0.
$$
By the regularity assumptions \eqref{reg_H2}, we have
$$
\|\bphi_1\|_1 \leq C \|w_1\|_2  \leq C \|z\|_0 \leq C \|z\|_2 \leq  C \|e\|_0.
$$
Then we have
\beq \label{e22_L2}
\inf_{(w_{1,h}, \bphi_{1,h}) \in S_{1,0}\times S_1^d}  \tri (w_1-w_{1,h}, \bphi_1-\bphi_{1,h}) \tri_0 \leq C h \|e\|_0.
\eeq
For the $w_2$ and $\bphi_2$ term, the proof of
\beq \label{e23_L2}
\inf_{(w_{2,h}, \bphi_{2,h}) \in S_{1,0}\times S_1^d}  \tri (w_2-w_{2,h}, \bphi_2-\bphi_{2,h}) \tri_0 \leq C h \|e\|_0.
\eeq
is identical to the  estimate of the same term in the proof in $L^2$-estimate of the weighted-LSFEM.

From \eqref{e21_L2}, \eqref{e22_L2}, and \eqref{e23_L2}, we have
\begin{eqnarray}
\|e\|_0^2 \leq C h \tri (e,E) \tri_0 \|e\|_0,
\end{eqnarray}
thus the $L^2$ error estimate is proved.

To prove the $H^1$-norm error estimate, let $v_h$ be the nodal interpolation of $u$ in $S_{0,1}$, by the triangle inequality, we have
$$
\|\nabla(u-u_h)\|_0 \leq \|\nabla(u-v_h)\|_0 + \|\nabla(u_h-v_h)\|_0. 
$$
By the inverse estimates and the triangle inequality, 
$$
\|\nabla(u_h-v_h)\|_0 \leq  Ch^{-1} \|u_h-v_h\|_0  \leq 
 Ch^{-1} (\|u-u_h\|_0 + \|u-v_h\|_0).
$$
Then, by the $L^2$ estimate of $u-u_h$  and the approximation property of $v_h$,
$$
\|\nabla(u-u_h)\|_0 \leq \|\nabla(u-v_h)\|_0 + Ch^{-1} (\|u-u_h\|_0 + \|u-v_h\|_0)
\leq C ( \tri (e,E) \tri_0 + h\|u\|_2).
$$
The theorem is proved.
\end{proof}

\section{Numerical Experiments}
\setcounter{equation}{0}

In this section, for a variety of problems, we will test them by $S_{1,0}\times S_1^2$ $L^2$-LSFEM and  $S_{2,0}\times S_1^2$ and $S_{3,0}\times S_2^2$ weighted-LSFEMs. We first present in Table \ref{order} a table of the optimal convergence rates with an assumption of both a smooth coefficient matrix and a smooth solution. For the purpose of comparison to $\|h A\ratio D_h^2 (u-u_h)\|_0$, we use $\|h D_h^2 (u-u_h)\|_0$ in stead of the broken $H^2$-norm $\|D_h^2 (u-u_h)\|_0$.

\begin{table}[!htp]
\caption{Convergence orders of LSFEMs with smooth coefficients and solutions, where $E=\bsigma-\bsigma_h$ and $e= u-u_h$.}
\begin{center}
\begin{tabular}{|c|c|c|c|c|c|c|}
\hline
method & $\tri(e,E) \tri$ & $\|e\|_0$ & $\|\nabla e\|_0$  & $\|E\|_0$ & $\|h A\ratio D_h^2 e\|_0$ & $\|h D_h^2 e\|_0$ \\
\hline
$S_{1,0}\times S_1^2$ $L^2$ & 1 & 2&1 & [1,2] & N.A. & N.A.\\
\hline
$S_{2,0}\times S_1^2$ weighted & 2 & 2& 2 & 2 & 2 & 2 \\
\hline
$S_{3,0}\times S_2^2$ weighted & 3 & 4& 3 & 3 & 3 & 3\\
\hline
\end{tabular}
\end{center}
\label{order}
\end{table}%
These convergence orders have different smoothness requirements, at least, theoretically. To get an optimal order of $\tri(u-u_h,\bsigma-\bsigma_h) \tri$ and  $\|hA\mathbin{:} D_h^2 (u-u_h)\|_0$ (for the weighted-LSFEM only), we only require the solution is piecewisely smooth enough while the matrix $A$ can be discontinuous or degenerate; to get an optimal order of the error of solution in the discrete broken $H^2$-norm, of course we need the $H^2$-regularity of the solution, and the matrix $A$ can be discontinuous but cannot be degenerate; we need both the solution and the coefficient matrix to be smooth to get optimal orders of the other norms.

In the paper, we only prove a non-optimal order $H^1$-norm estimate of the error of the weighted-LSFEM. By the same argument as the $L^2$-LSFEM, we can get an optimal $H^1$-norm convergence order for the $S_{k,0}\times S_{k-1}^2$ weighted-LSFEM is $k$, $k\geq 3$, for smooth solutions and coefficients. For the $S_{2,0}\times S_{2-1}^2$ weighted-LSFEM, we observe the order is the best interpolation order $2$, for problems with smooth solutions and coefficients.  

\subsection{Examples with a smooth solution}
Let $r = \sqrt{x^2+y^2}$. We consider several different cases:
$$
A_1 = 
\left(
\begin{array}{cc} 
	r^{1/2}+1  & -r^{1/2}\\
  	-r^{1/2}  &  5r^{1/2}+1 
   \end{array}
\right), \quad
A_2 = 
\left(
\begin{array}{cc} 
	-\frac{5}{\ln(r)}+15  & 1\\
  	1  &  -\frac{1}{\ln(r)}+3 
   \end{array}
\right),
$$
$$
A_3 =
\left(
\begin{array}{cc} 
	2  &  \frac{xy}{|xy|}\\
  	 \frac{xy}{|xy|} &  2 
   \end{array}
\right)
\quad \mbox{and} \quad
 A_4 =
\left(
\begin{array}{cc} 
	|x|^{2/3}  & -|x|^{1/3}|y|^{1/3}\\
  	-|x|^{1/3}|y|^{1/3}  &  |y|^{2/3} 
   \end{array}
\right).
$$
The coefficients of the matrices $A_1$, $A_2$, $A_3$, and $A_4$ are H\"older continuous, uniformly continuous, discontinuous, and degenerate, respectively. The matrix $A_4$ is degenerate since $\det(A_4) =0$, thus the equation is not uniformly elliptic and not satisfying the Cordes condition. The domain is chosen to be $\O = (-1/2,1/2)^2$. We choose the data $f$ for different $A$ such that the solution is 
$$
u(x,y) = \sin(2 \pi x) \sin(2\pi y) e^{x \cos(y)}.
$$

The initial mesh contains four triangles by connecting two diagonals. Eight uniform refinements are performed to generate a series of numerical solutions.

\subsubsection{$L^2$-LSFEM}

\begin{figure}[htp]
\centering 
\subfigure[H\"older continuous coefficient $A_1$]{
\includegraphics[width=0.41\linewidth]{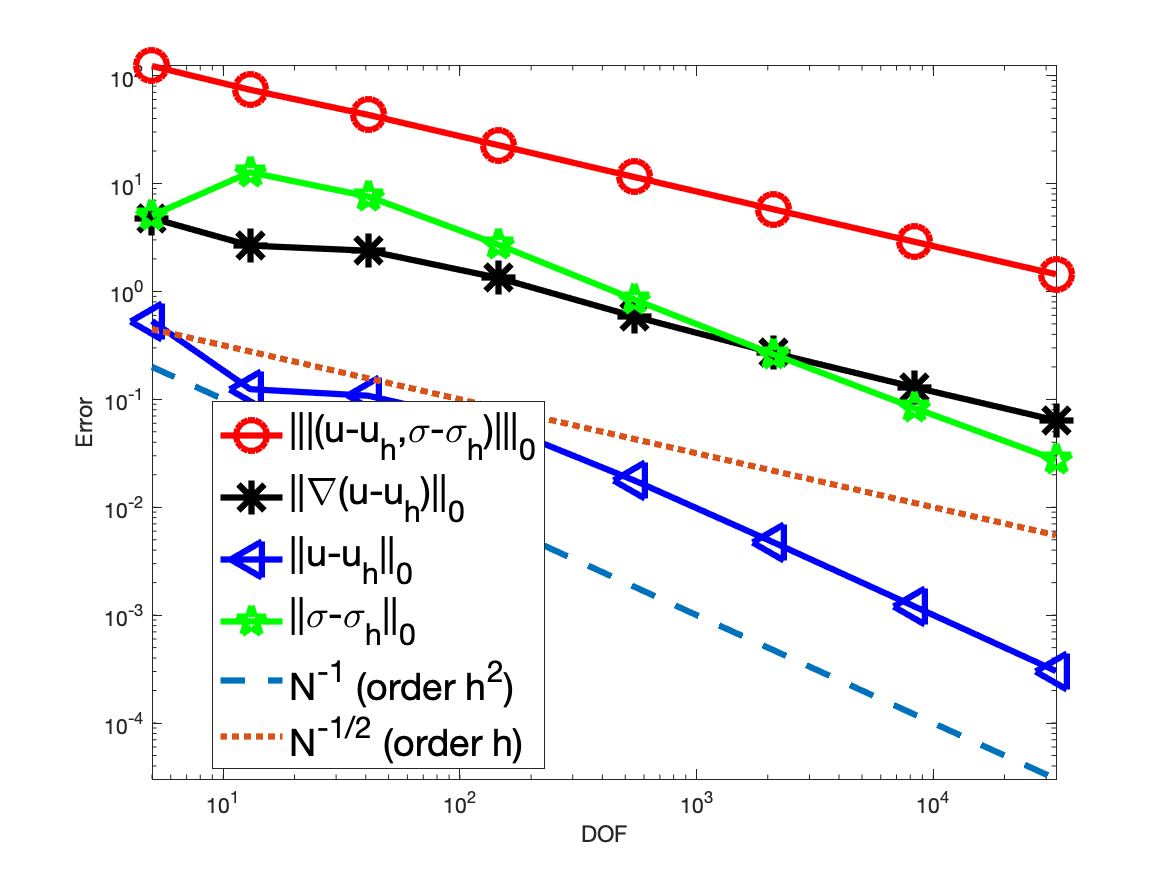}}
\subfigure[Uniformly continuous coefficient $A_2$]{
\includegraphics[width=0.41\linewidth]{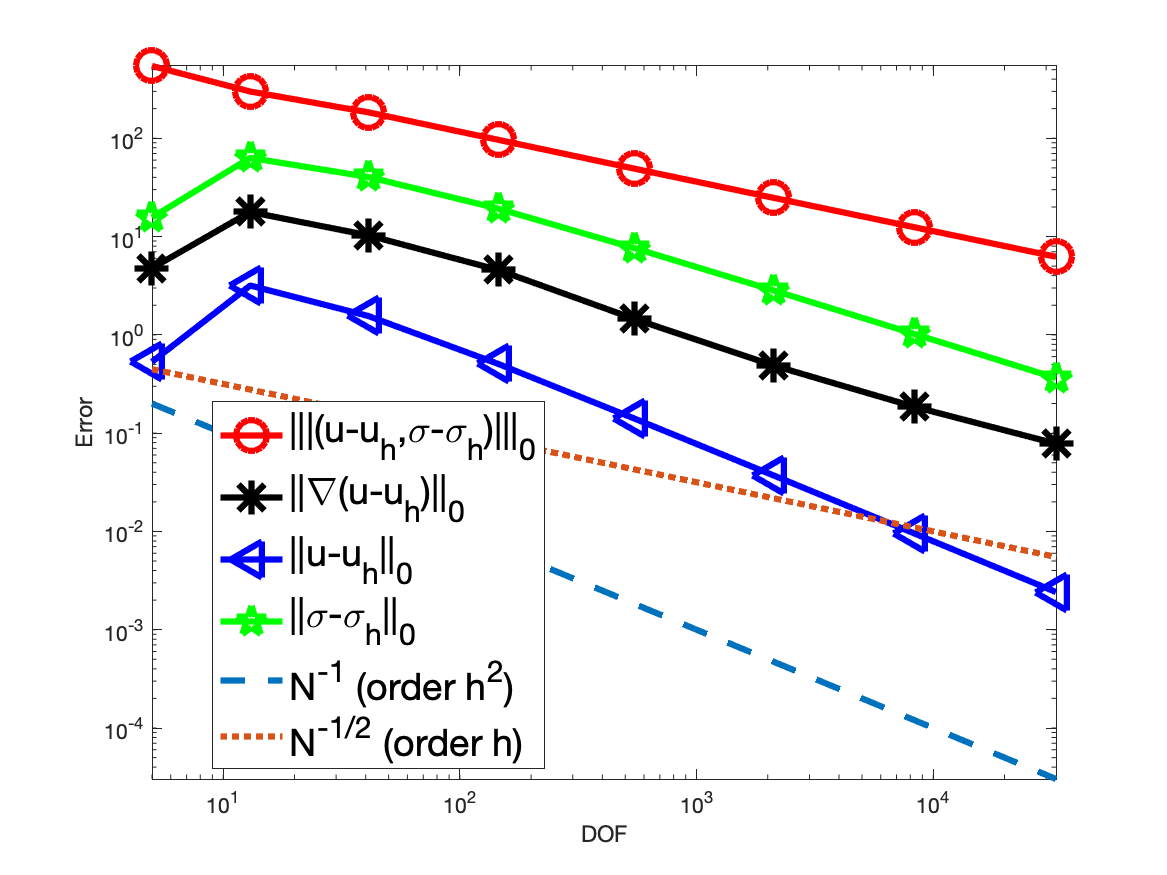}}\\
\subfigure[Discontinuous continuous coefficients $A_3$]{
\includegraphics[width=0.41\linewidth]{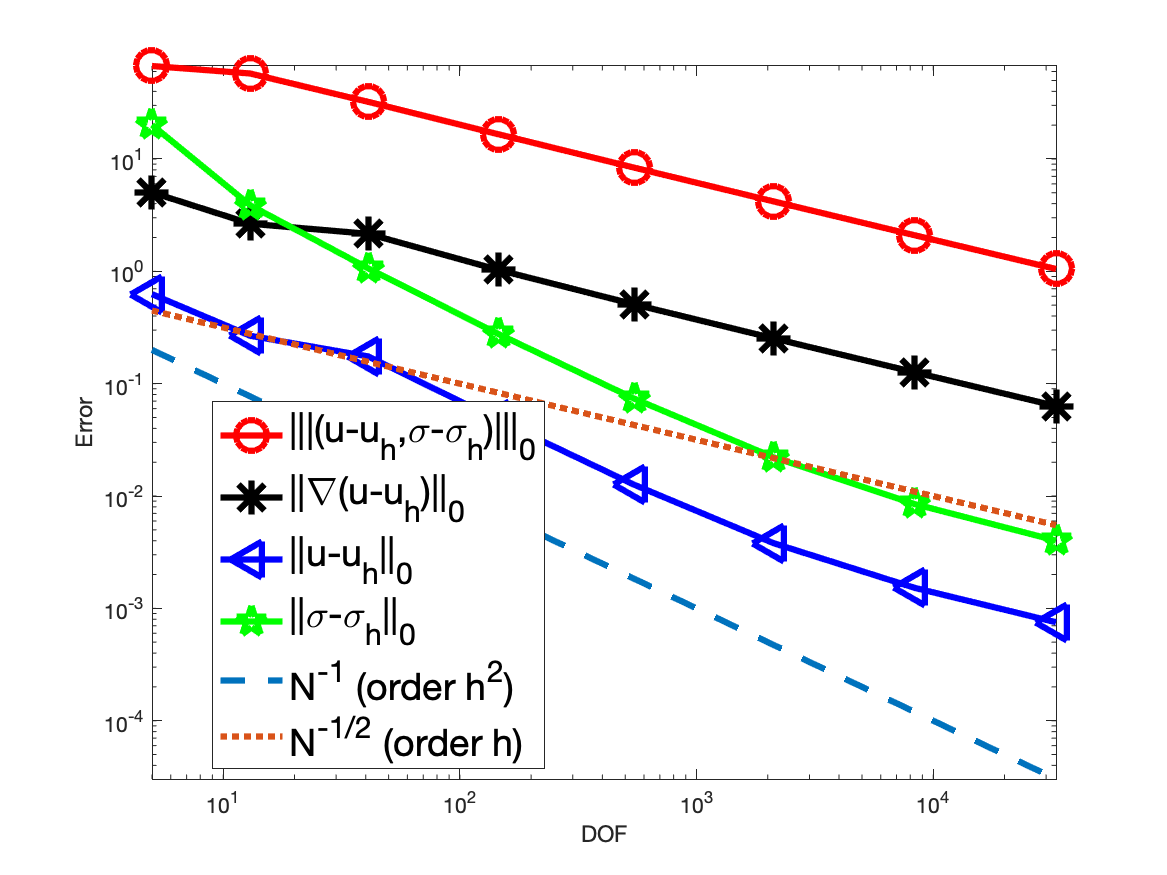}}
\subfigure[Degenerate continuous coefficients $A_4$]{
\includegraphics[width=0.41\linewidth]{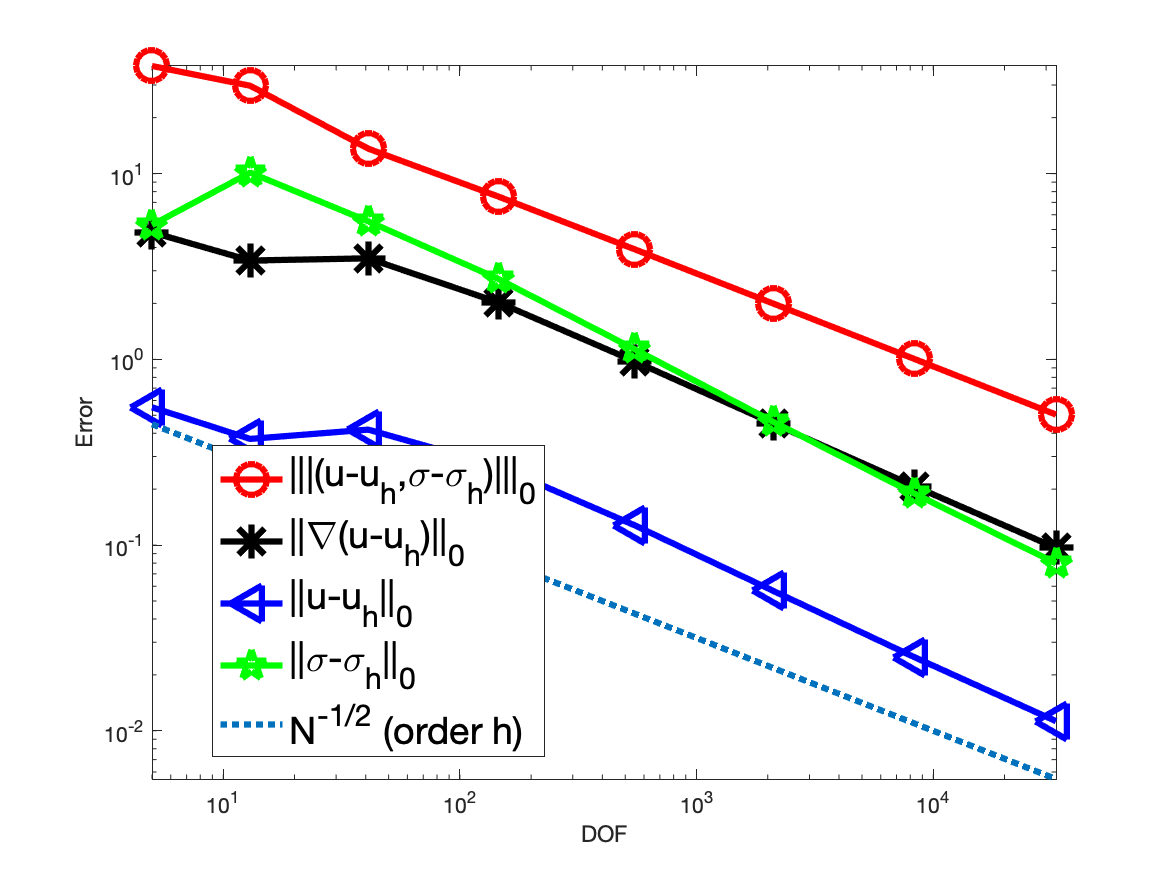}}
\caption{Convergence histories for $L^2$-LSFEM with $S_{1,0}\times S_1^2$  smooth solution}
 \label{con_LSFEM_P1}
\end{figure}

From Fig. \ref{con_LSFEM_P1}, we can see that for a problem with a smooth solution, the convergence order of the  error in the LSFEM norm $\tri(u-u_h,\bsigma-\bsigma_h) \tri_0$ is always one for all cases if the  $L^2$-LSFEM with $S_{1,0}\times S_1^2$ approximation is used, which is optimal and compatible with the theoretical analysis.

For the error of $u$ in $H^1$-semi norm $\|\nabla (u-u_h)\|_0$, the order is the optimal one for all cases. For the $L^2$-norm error $\|u-u_h\|_0$, the H\"older continuous and  uniformly continuous problems have an optimal order two, while the discontinuous case has an order slightly less than two and the degenerate case has an order slightly bigger than one. These results are not covered by the theoretical analysis, but the numerical experiments suggest that the errors in $L^2$-norm are more sensitive to the smoothness of the coefficients, while the errors in $H^1$-norm is less sensitive. 

For the error of $\bsigma$ in $L^2$ norm $\|\bsigma-\bsigma_h\|_0$, we do not have theoretical analysis, and the numerical results do show that the convergence order depends on the problem. The observed order is between one, which is the approximation order of $\|\nabla (u-v_h)\|_0$, and two, which is the optimal interpolation order of $\|\bsigma-\btau_h\|_0$.  

\subsubsection{Weighted-LSFEM}
\begin{figure}[!htb]
\centering 
\subfigure[H\"older continuous coefficient $A_1$]{
\includegraphics[width=0.4\linewidth]{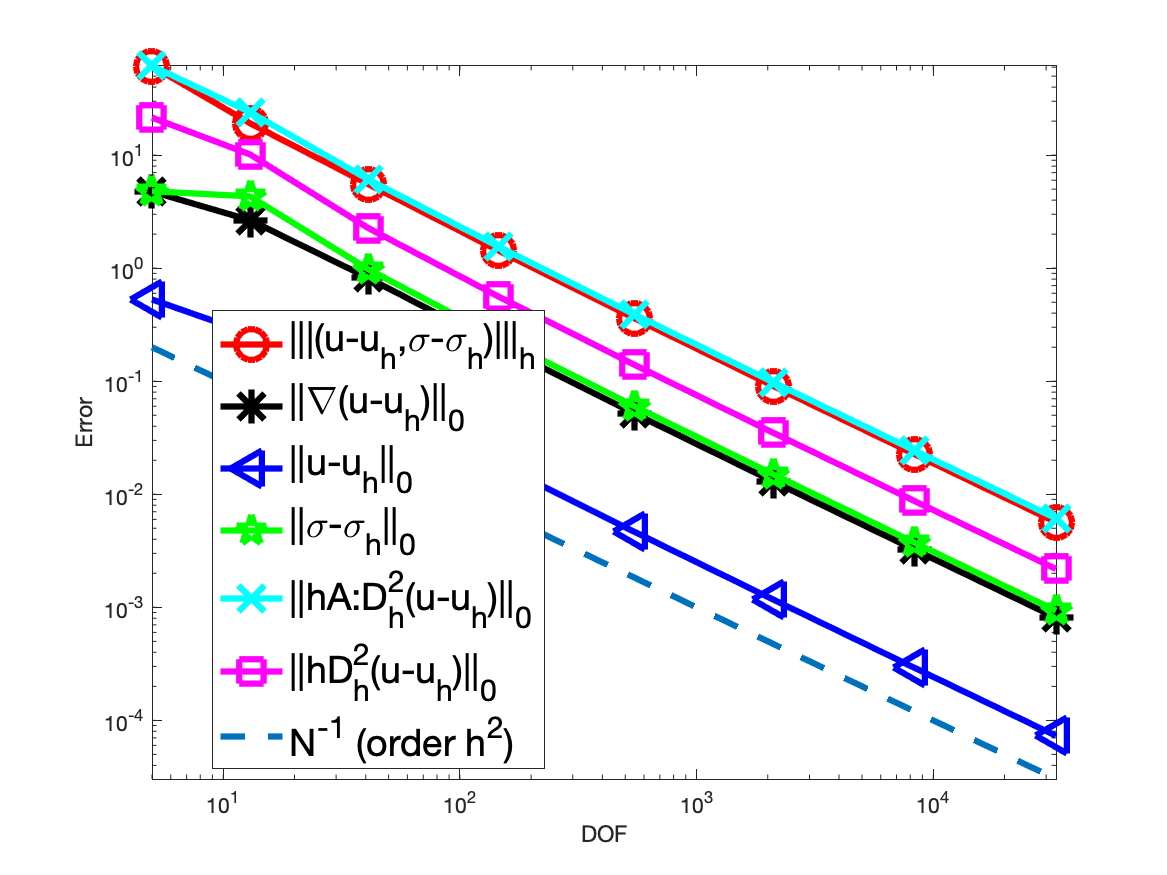}}
\subfigure[Uniformly continuous coefficient $A_2$]{
\includegraphics[width=0.4\linewidth]{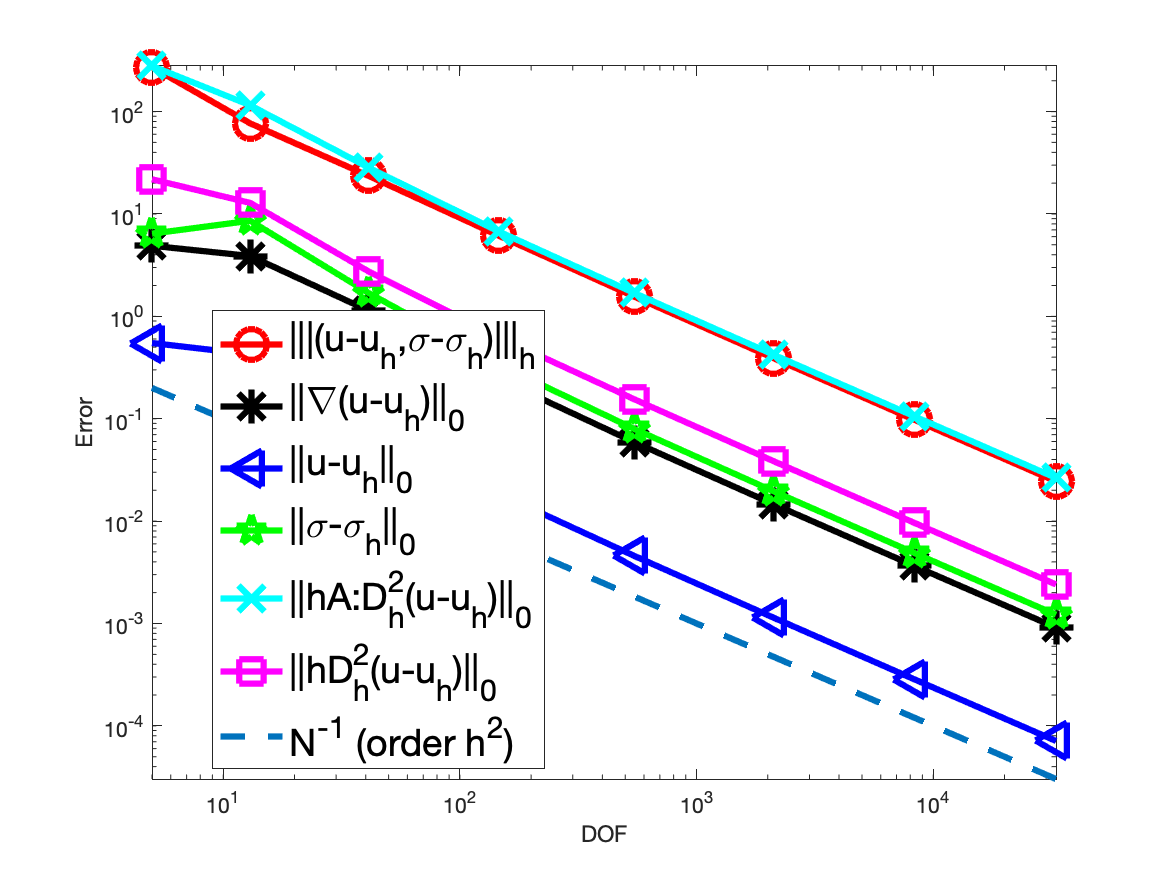}}\\
\subfigure[Discontinuous coefficient $A_3$]{
\includegraphics[width=0.4\linewidth]{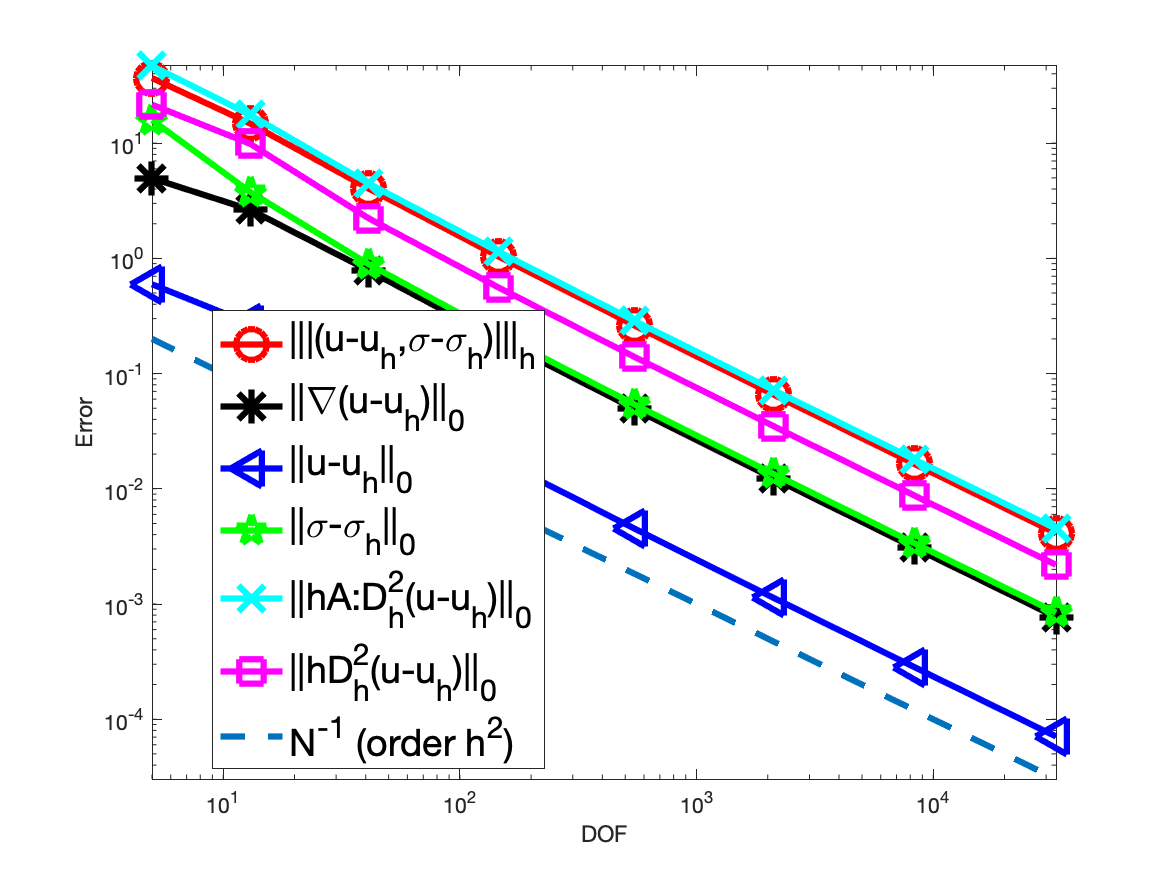}}
\subfigure[Degenerate coefficient $A_4$]{
\includegraphics[width=0.4\linewidth]{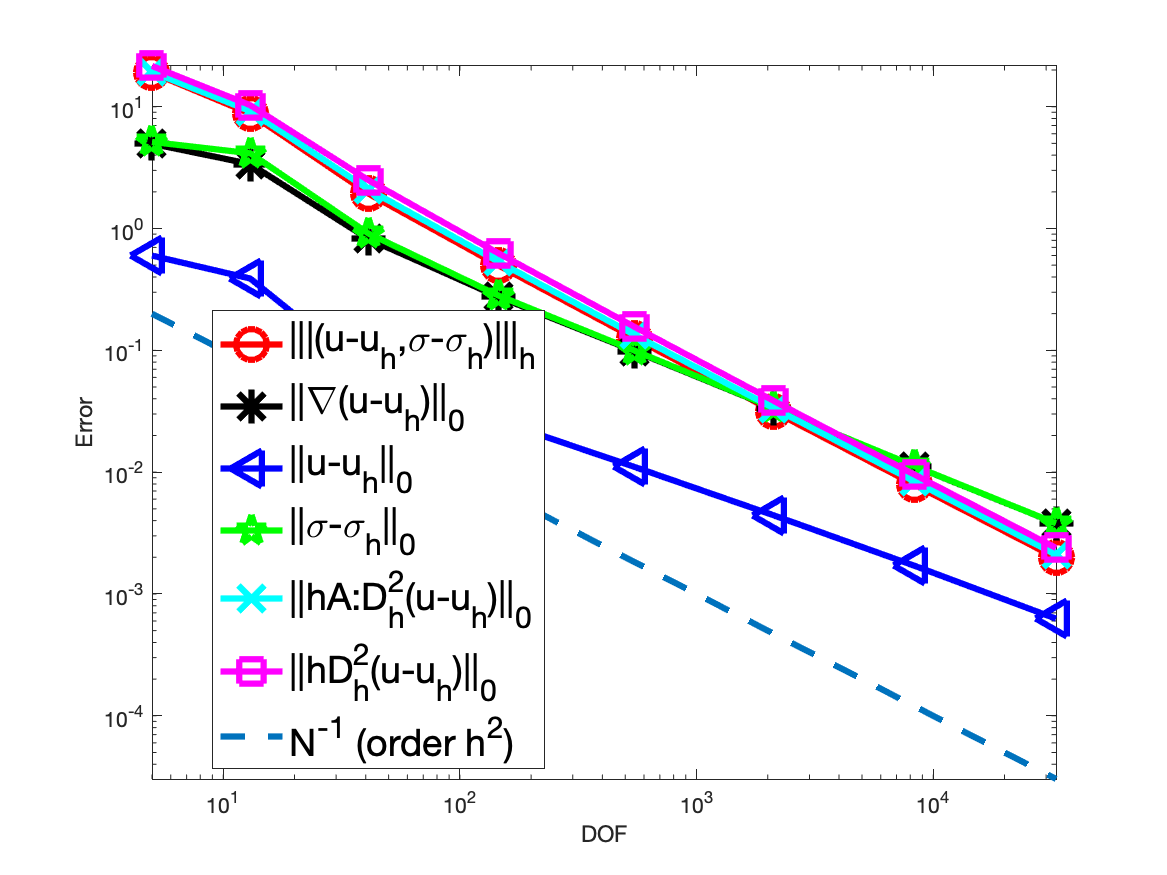}}
%
\subfigure[H\"older continuous coefficient $A_1$]{
\includegraphics[width=0.4\linewidth]{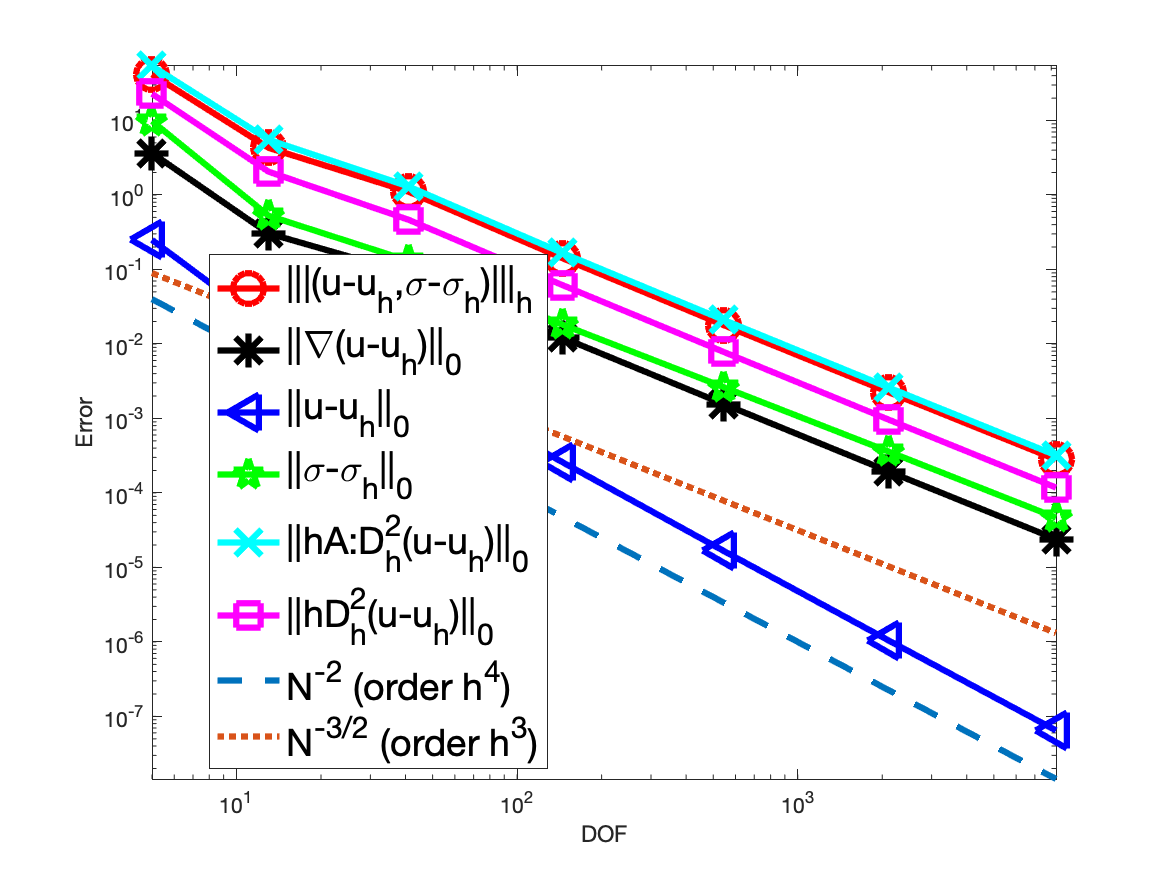}}
\subfigure[Uniformly continuous coefficient $A_2$]{
\includegraphics[width=0.4\linewidth]{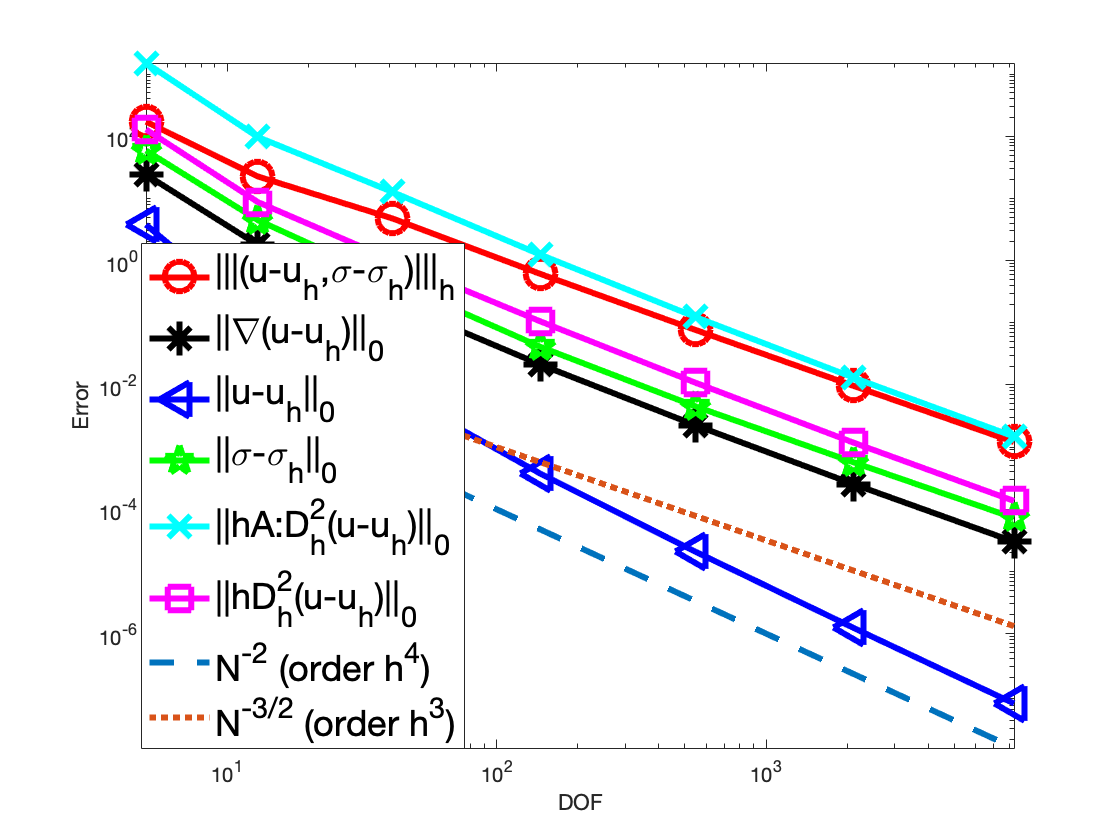}}\\
\subfigure[Discontinuous coefficient $A_3$]{
\includegraphics[width=0.4\linewidth]{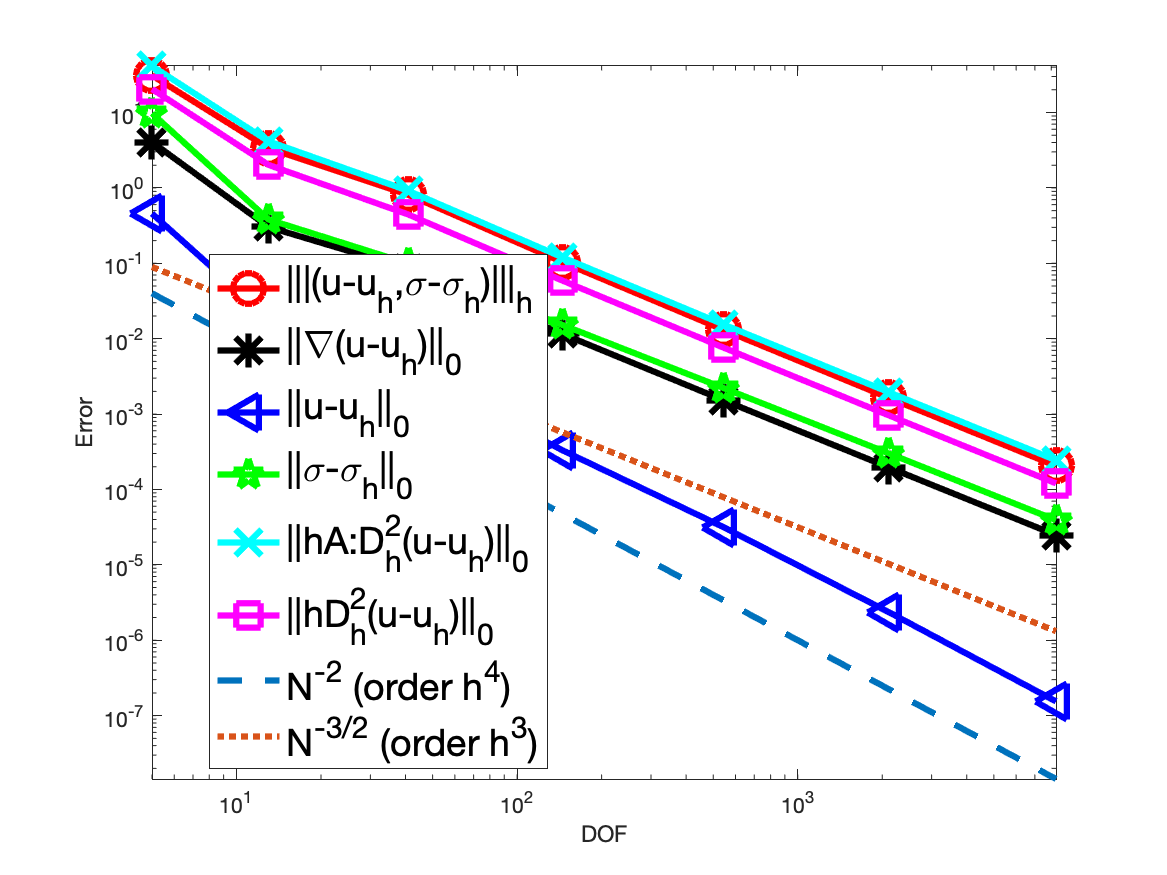}}
\subfigure[Degenerate coefficient $A_4$]{
\includegraphics[width=0.4\linewidth]{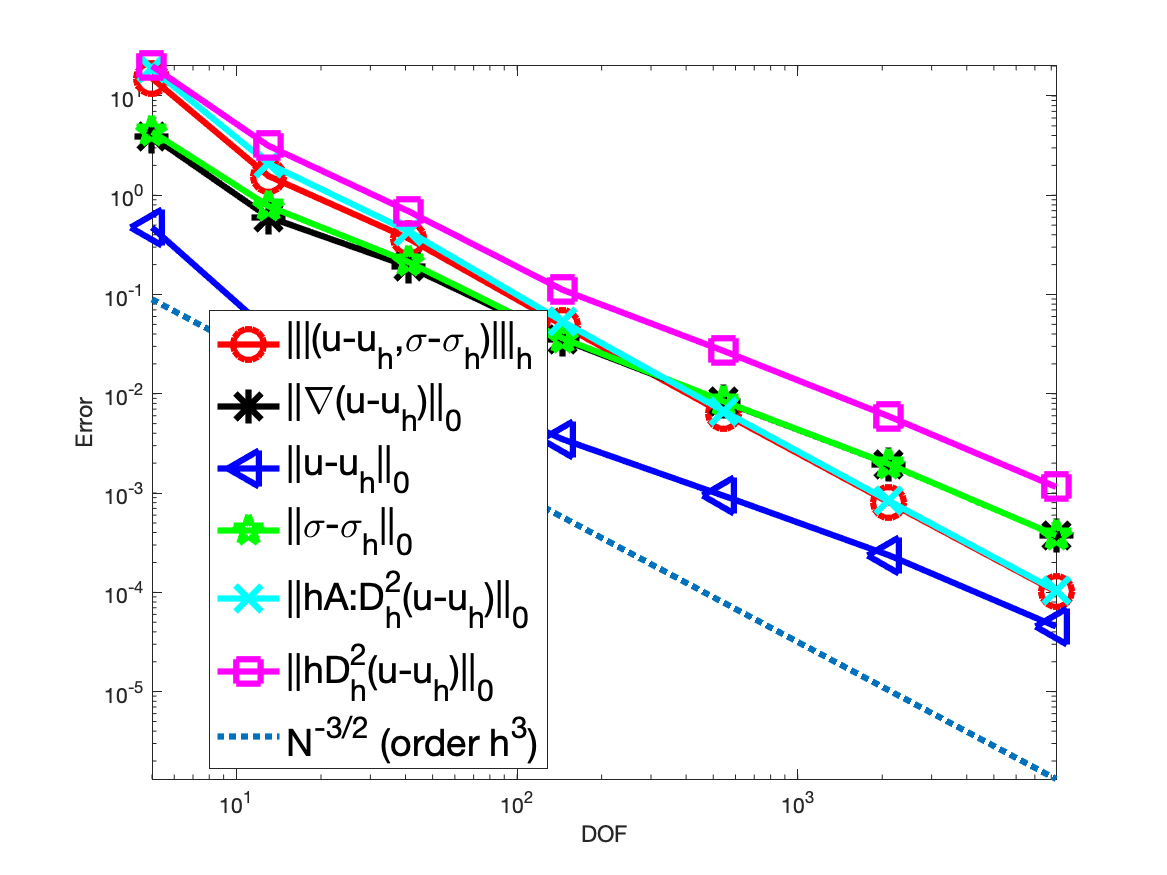}}
\caption{Convergence histories for $S_{k,0}\times S_{k-1}^2$  weighted-LSFEM with a smooth solution, a-d (k=2), e-h(k=3)}
\label{con_LSFEM_P2P3}
\end{figure}

In Fig. \ref{con_LSFEM_P2P3}, we show the numerical results for the smooth solution problems using the-weighted LSFEM with $S_{2,0}\times S_1^2$ and $S_{3,0}\times S_2^2$ approximations. 

As discussed earlier, the error in the LSFEM norm $\tri(u-u_h,\bsigma-\bsigma_h) \tri_h$ and $\|A\ratio D^2_h(u-u_h)\|_0$ are of optimal orders and compatible with the theoretical analysis.

For all non-degenerate cases, $\|\nabla (u-u_h)\|_0$, $\|\bsigma-\bsigma_h\|_0$, $\|u-u_h\|_0$, and 
$\|hD_h^2(u-u_h)\|_0$ are of optimal orders.

For the degenerate case, for $k=2$, $\|\nabla (u-u_h)\|_0 =O(h^{1.5})$ and $\|\bsigma-\bsigma_h\|_0 = O(h^{1.5})$ are less than optimal,  $\|u-u_h\|_0 = O(h^{1.4})$ is even worse than the $H^1$-semi norm error, but $\|hD_h^2(u-u_h)\|_0$ is of optimal order 2. For $k=3$, $\|\nabla (u-u_h)\|_0$, $\|\bsigma-\bsigma_h\|_0$, and $\|u-u_h\|_0$, are of order 2.4, which are less than the optimal order; $\|hD_h^2(u-u_h)\|_0$ is of order $2.3$, which is also less than the optimal order.

In conclusion,  we find that for the degenerate case, $\|\nabla (u-u_h)\|_0$, $\|\bsigma-\bsigma_h\|_0$, $\|u-u_h\|_0$, and $\|hD_h^2(u-u_h)\|_0$ are often worse than optimal order, while the other cases are fine with a smooth solution.
%

\subsection{A Discontinuous Coefficients Problem from \cite{SS:13}}
Let $\O= (-1,1)^2$ and the coefficient matrix be the discontinuous matrix $A_3$. Choose the righthand side $f$ so that the exact solution is
\beq
u(x,y) =xy (e^{1-|x|}-1)(e^{1-|y|}-1).
\eeq
Note that both the solution $u$ and the gradient $\nabla u$ are piecewise smooth and continuous.

\begin{figure}[!htb]
\centering 
\subfigure[$L^2$-LSFEM with $S_{1,0}\times S_1^2$]{
\includegraphics[width=0.31\linewidth]{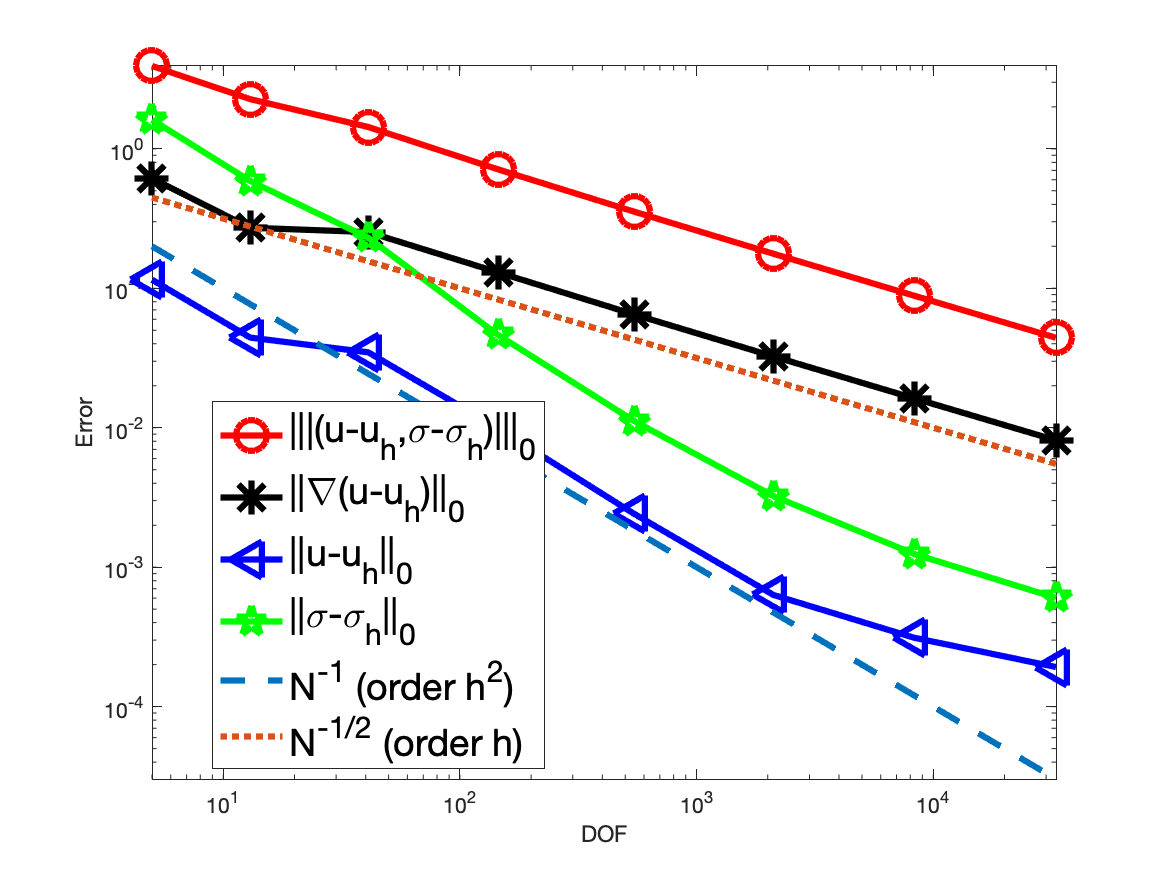}}
\subfigure[Weighted-LSFEM with $S_{2,0}\times S_1^2$]{
\includegraphics[width=0.31\linewidth]{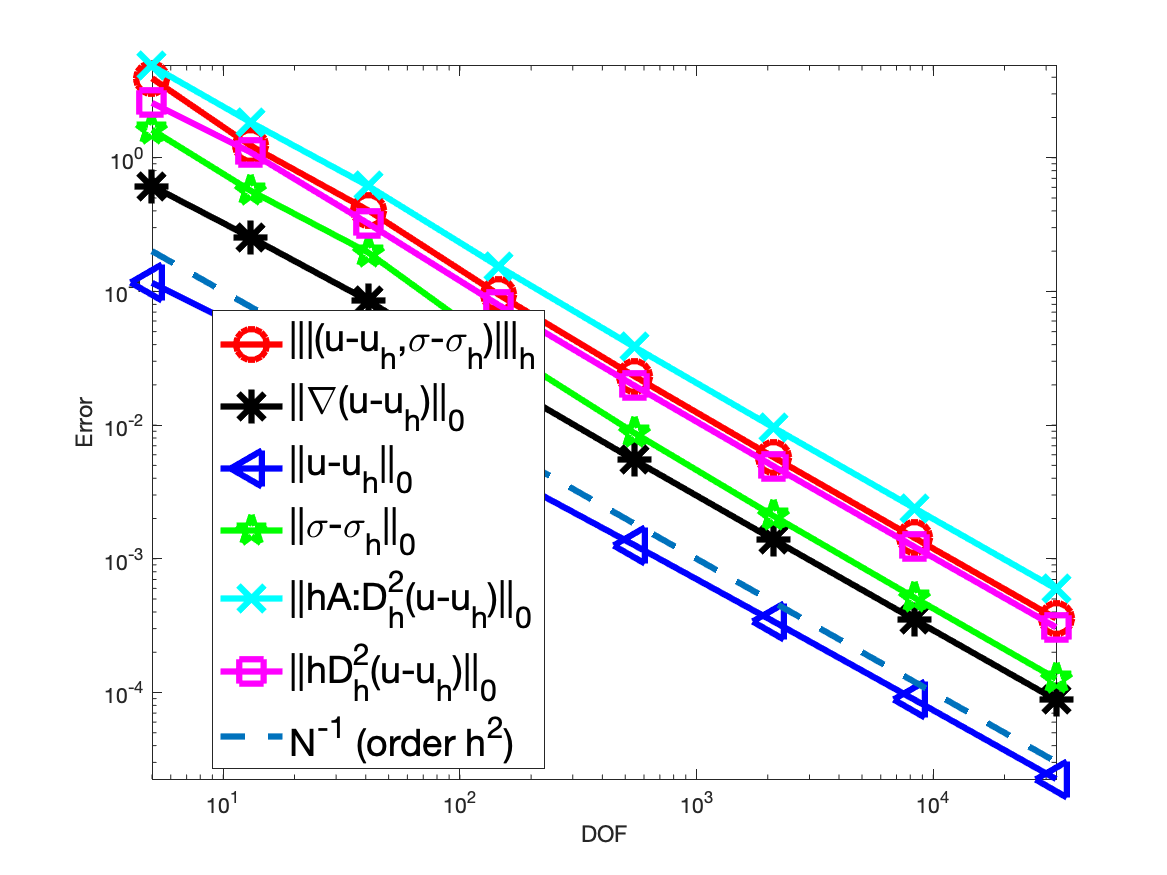}}
\subfigure[Weighted-LSFEM with $S_{3,0}\times S_2^2$]{
\includegraphics[width=0.31\linewidth]{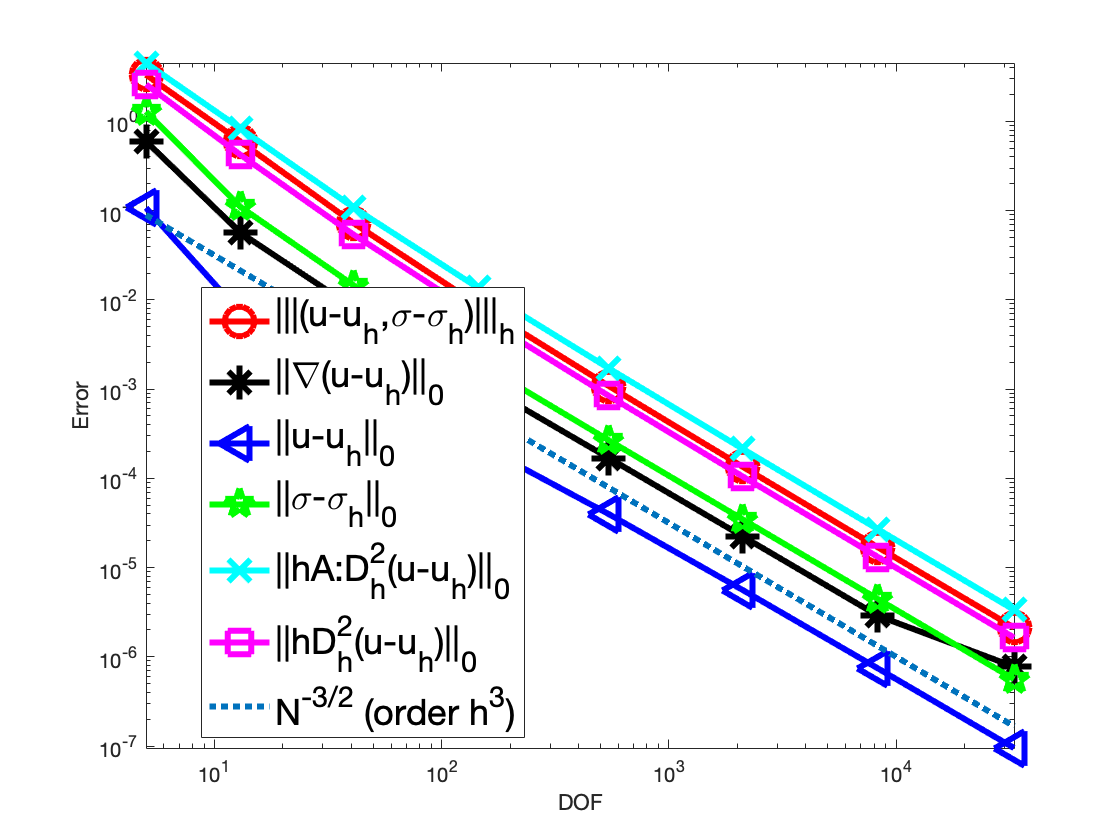}}
\caption{Convergence histories for discontinuous coefficient problem with $u(x,y) = xy (e^{1-|x|}-1)(e^{1-|y|}-1)$}
 \label{con_LSFEM_dis}
\end{figure}

In Fig. \ref{con_LSFEM_dis}, we show  the numerical results with uniform refinements using the $S_{1,0}\times S_1^2$ $L^2$-LSFEM and  $S_{2,0}\times S_1^2$ and $S_{3,0}\times S_2^2$ weighted-LSFEM.

As examples before, $\tri(u-u_h,\bsigma-\bsigma_h) \tri$ and $\|\nabla (u-u_h)\|_0$  
are optimal and compatible with the theoretical analysis.

Due to the discontinuity in the coefficients, the error of $u$ in $L^2$ norm $\|\nabla (u-u_h)\|_0$ behaves differently for different cases. The $S_{1,0}\times S_1^2$ $L^2$-LSFEM has order two at the beginning but reduces to order one later. The $S_{2,0}\times S_1^2$ weighted-LSFEM has a consistent order two, which is optimal. The $S_{3,0}\times S_2^2$ weighted-LSFEM has order three, which loses an order compared to the smooth coefficient case. 

For other norms, the convergence orders are the same as the optimal ones for smooth problems. 

In a summary, for this piecewise smooth solution example with discontinuous coefficients, excpet for the $L^2$-error of the $S_{3,0}\times S_2^2$ weighted-LSFEM,  we observe convergence results very similar to the global smooth solution example. 


\subsection{A singular solution example from \cite{FHN:17}}
Let the coefficient matrix $A=A_2$ and $\O = (0,1/2)^2$. The exact solution is chosen to be
$$
u = (x^2+y^2)^{7/8}.
$$
The solution is not in $H^2$ and has a singularity at the origin.

\begin{figure}[!htb]
\centering 
\subfigure[$L^2$-LSFEM with $S_{1,0}\times S_1^2$]{
\includegraphics[width=0.31\linewidth]{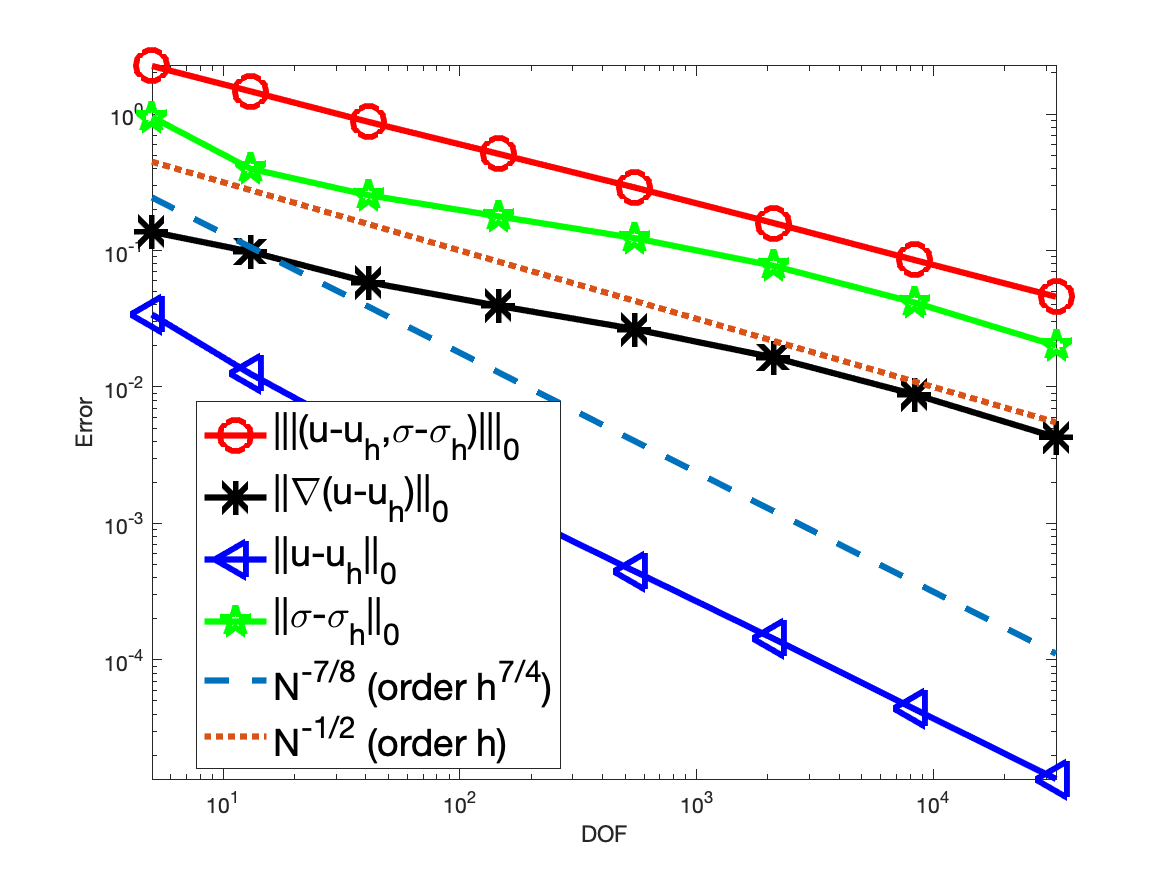}}
\subfigure[Weighted-LSFEM with $S_{2,0}\times S_1^2$]{
\includegraphics[width=0.31\linewidth]{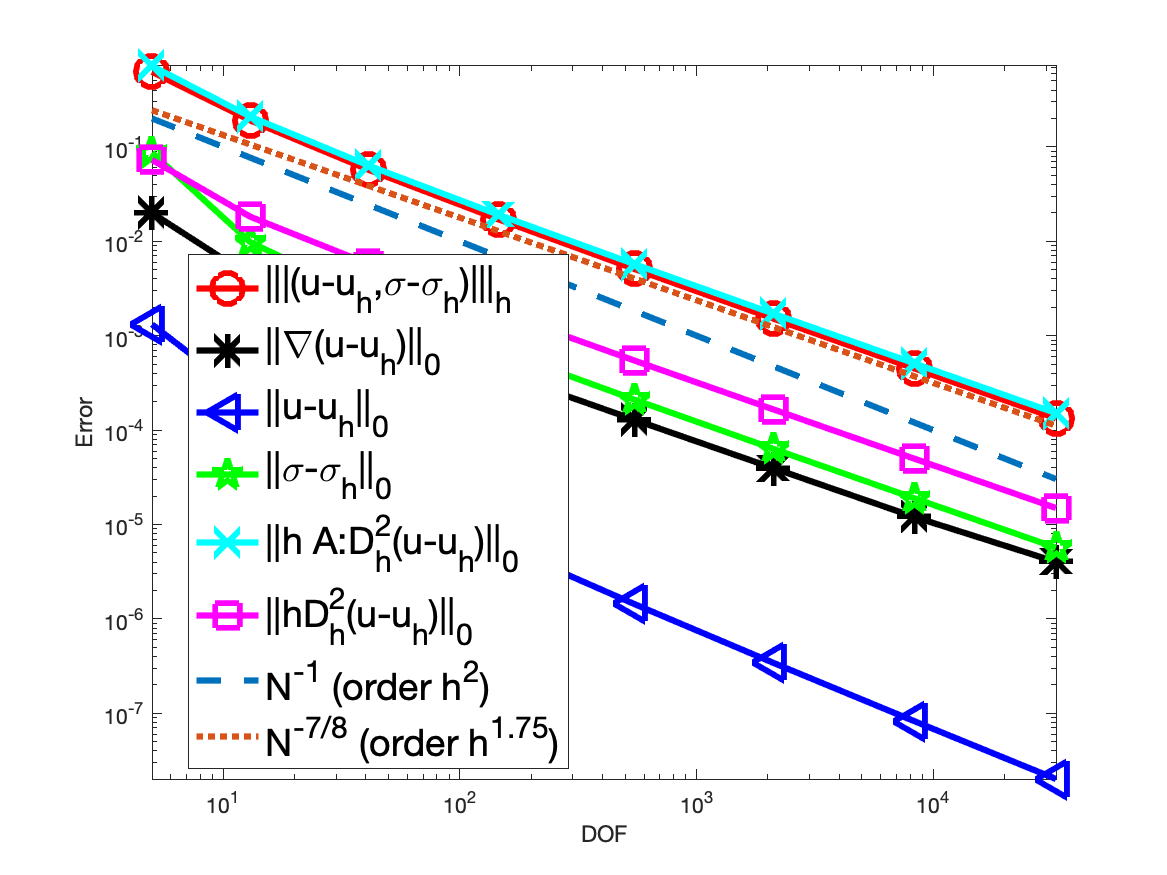}}
\subfigure[Weighted-LSFEM with $S_{3,0}\times S_2^2$]{
\includegraphics[width=0.31\linewidth]{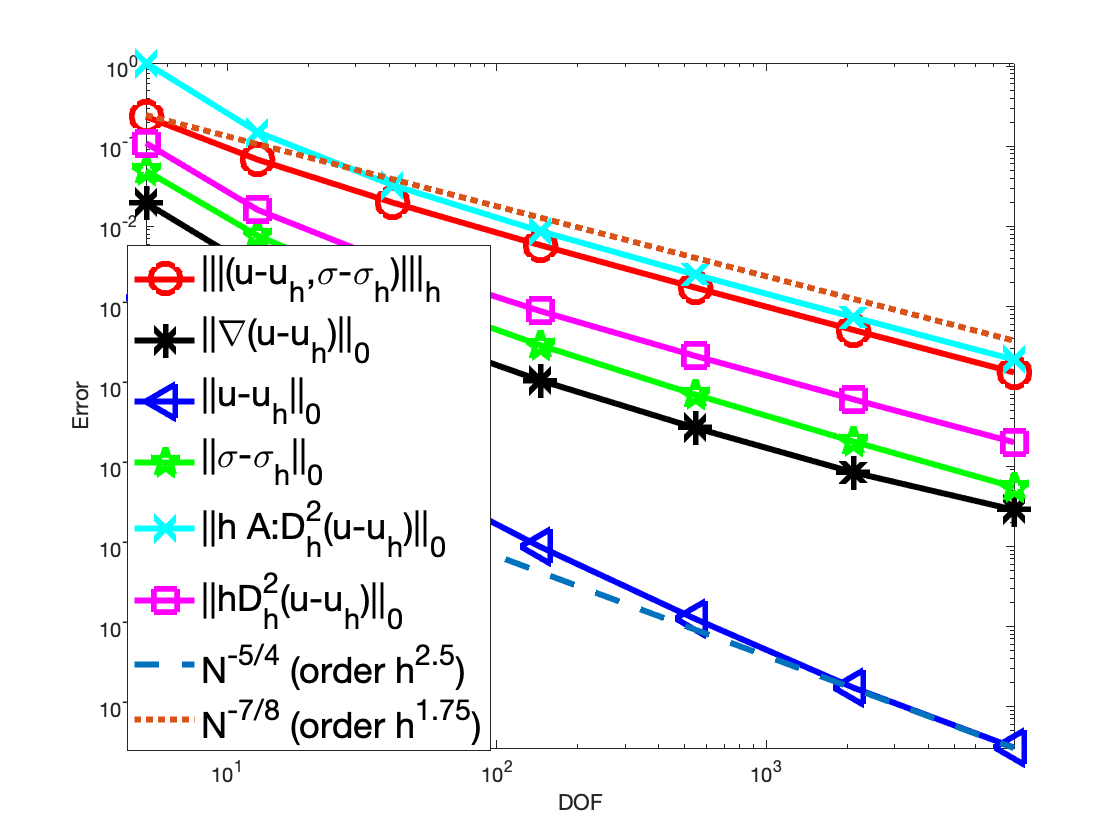}}
\caption{Convergence histories for uniformly continuous coefficient problem with $u = (x^2+y^2)^{7/8}$ with uniform mesh refinements}
 \label{con_LSFEM_uniform_x74}
\end{figure}

In Fig. \ref{con_LSFEM_uniform_x74}, we show  the numerical results with uniform refinements using the $S_{1,0}\times S_1^2$ $L^2$-LSFEM and  $S_{2,0}\times S_1^2$ and $S_{3,0}\times S_2^2$ weighted-LSFEM. 

For the $S_{1,0}\times S_1^2$ $L^2$-LSFEM, $\tri(u-u_h,\bsigma-\bsigma_h) \tri_0$, $\|\nabla(u-u_h)\|_0$, and $\|\bsigma-\bsigma_h\|_0$ converge at an order close to $0.9$, and $\|u-u_h\|_0$ converges around an order of $1.7$. These are due to the singularity.

The weighted-LSFEMs convergence at an order $7/4$ for $\tri(u-u_h,\bsigma-\bsigma_h) \tri$,  $\|h:A D_h^2 (u-u_h)\|_0$, and $\|h:D_h^2 (u-u_h)\|_0$, which is the optimal order with respect to the regularity. The errors  $\|\nabla (u-u_h)\|_0$ and $\|\bsigma-\bsigma_h\|_0$ converge at a rate $1.6$. The error of $u$ in $L^2$ norm $\|u-u_h\|_0$ is close to $2$ for the $S_{2,0}\times S_1^2$ weighted-LSFEM, while the $S_{3,0}\times S_2^2$ weighted-LSFEM has an order $2.5$,

\begin{figure}[!htb]
\centering 
\subfigure[$L^2$-LSFEM with $S_{1,0}\times S_1^2$]{
\includegraphics[width=0.41\linewidth]{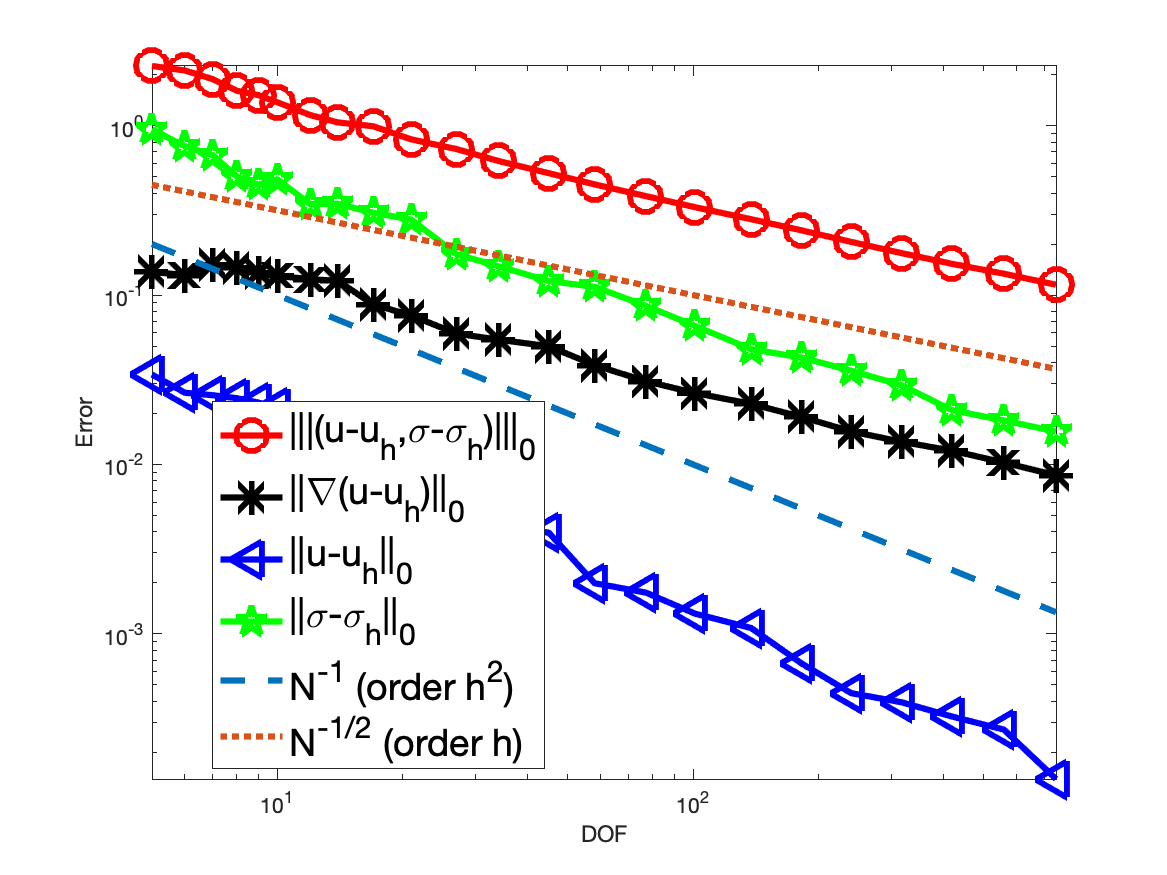}}
\subfigure[Weighted-LSFEM with $S_{2,0}\times S_1^2$]{
\includegraphics[width=0.41\linewidth]{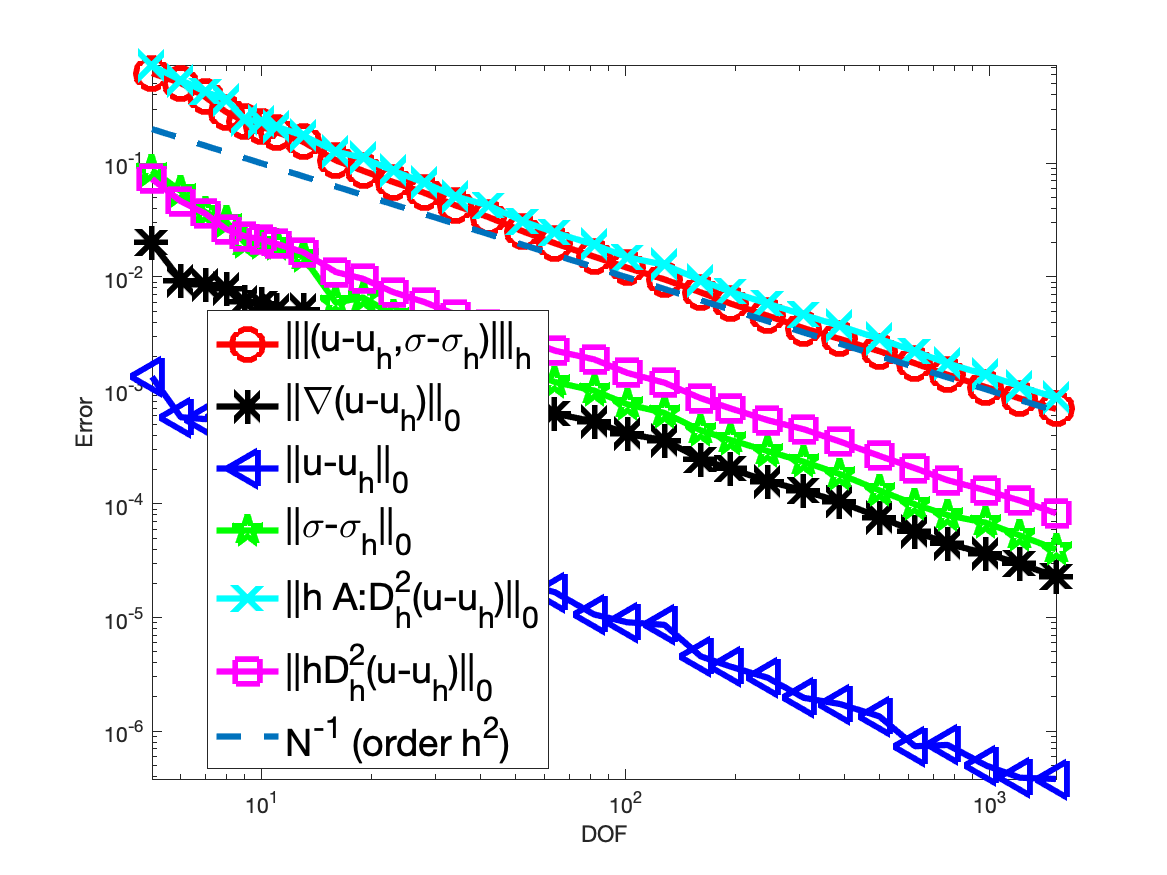}}
\subfigure[Weighted-LSFEM with $S_{3,0}\times S_2^2$]{
\includegraphics[width=0.41\linewidth]{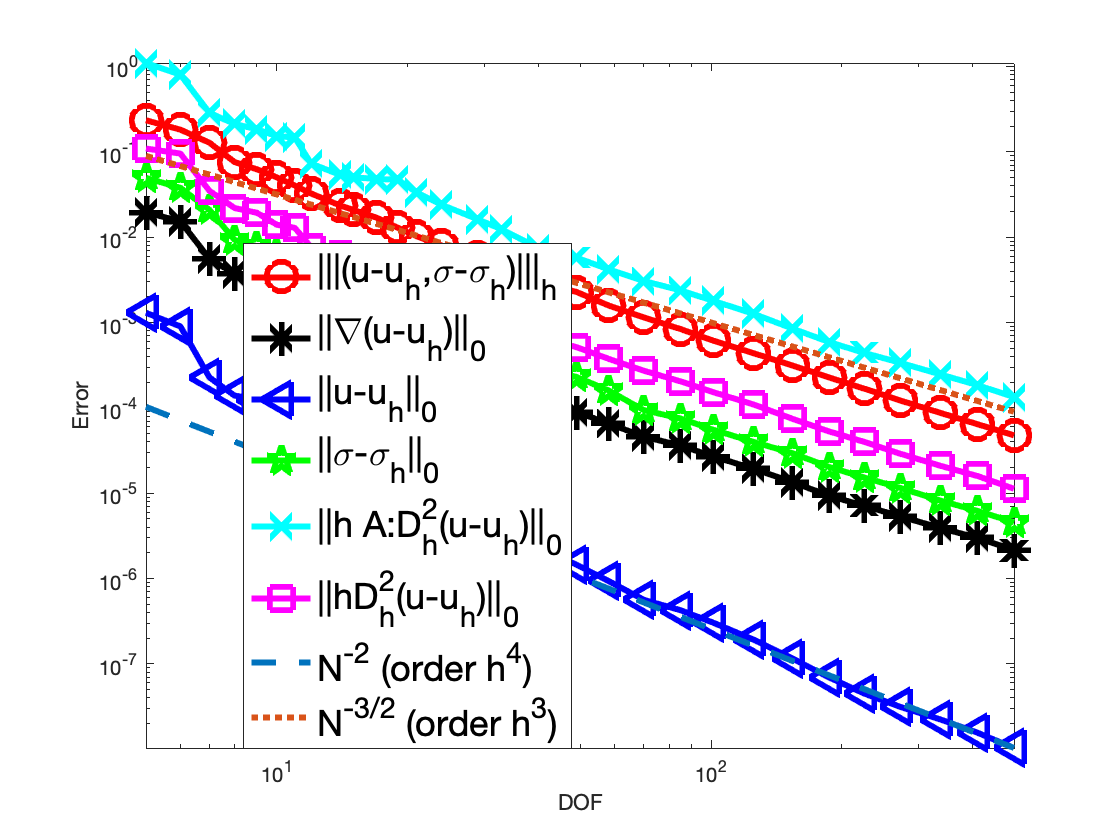}}
\subfigure[A refined mesh generated by adaptive $L^2$-LSFEM]{
\includegraphics[width=0.41\linewidth]{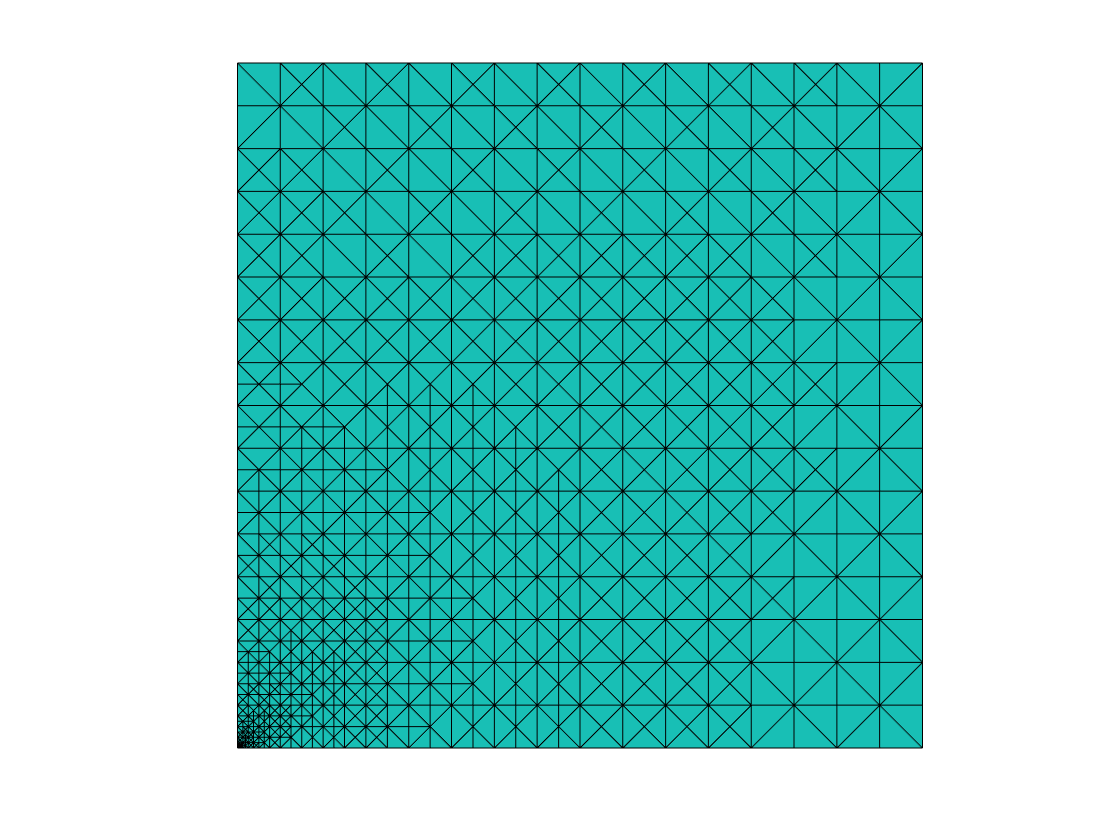}}
\caption{Convergence histories for uniformly continuous coefficient problem with $u = (x^2+y^2)^{7/8}$ with adaptive mesh refinements}
 \label{con_LSFEM_uniform_x74_afem}
\end{figure}

On the (a), (b), and (c) of Fig. \ref{con_LSFEM_uniform_x74_afem}, we show  the numerical results with adaptive refinements using the adaptive $S_{1,0}\times S_1^2$ $L^2$-LSFEM and  $S_{2,0}\times S_1^2$ and $S_{3,0}\times S_2^2$ weighted-LSFEMs. Not surprisingly, all convergences orders behave as if the solution and the matrix are smooth. On the (d) of Fig. \ref{con_LSFEM_uniform_x74_afem}, we show a refined mesh generated by adaptive $L^2$-LSFEM, we clearly see the refinements around the singularity. 

\subsection{A singular solution example with a degenerate matrix from \cite{FHN:17}}
Let the coefficient matrix $A=A_4$ and $\O = (0,1)^2$. The righthand side $f$ is $0$ and exact solution is chosen to be
$$
u = x^{4/3}-y^{4/3}.
$$
The solution is singular along the $x$ and $y$ axises. 

\begin{figure}[!htb]
\centering 
\subfigure[$L^2$-LSFEM with $S_{1,0}\times S_1^2$]{
\includegraphics[width=0.31\linewidth]{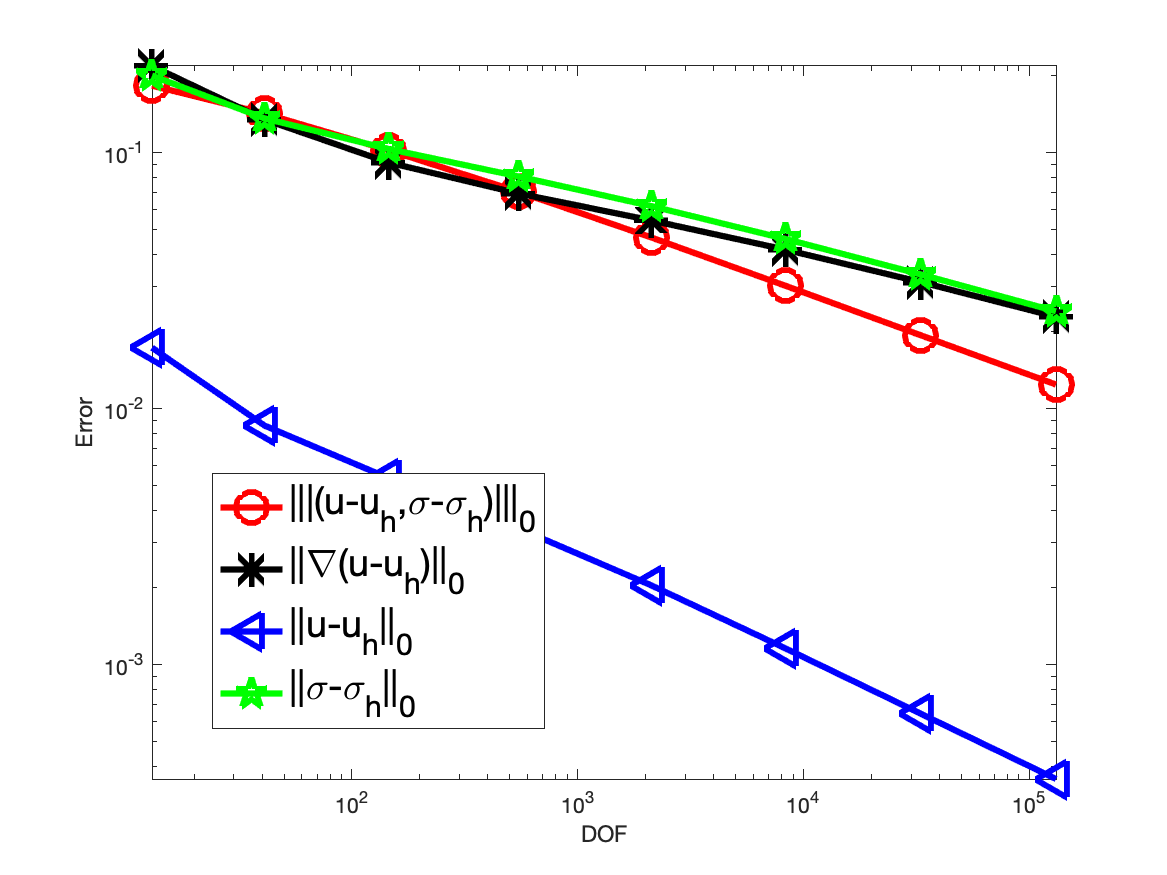}}
\subfigure[Weighted-LSFEM with $S_{2,0}\times S_1^2$]{
\includegraphics[width=0.31\linewidth]{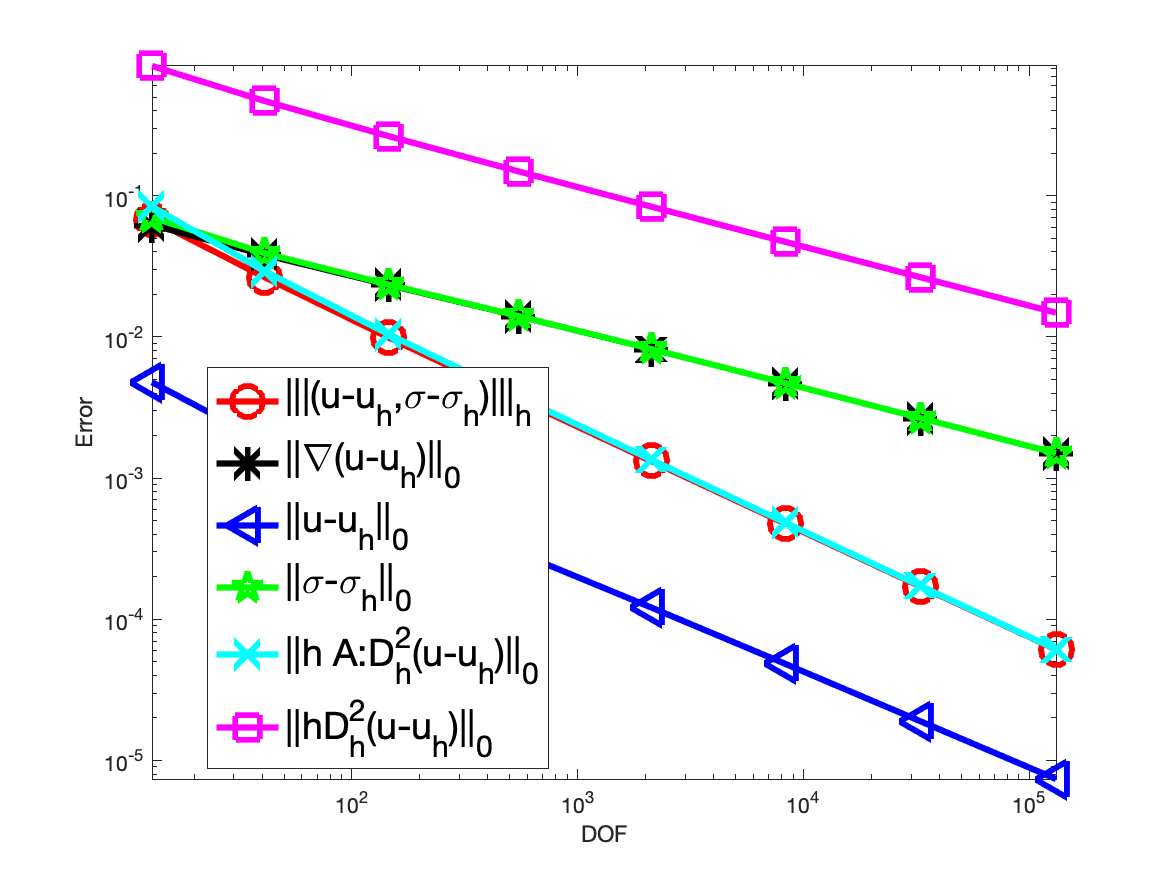}}
\subfigure[Weighted-LSFEM with $S_{3,0}\times S_2^2$]{
\includegraphics[width=0.31\linewidth]{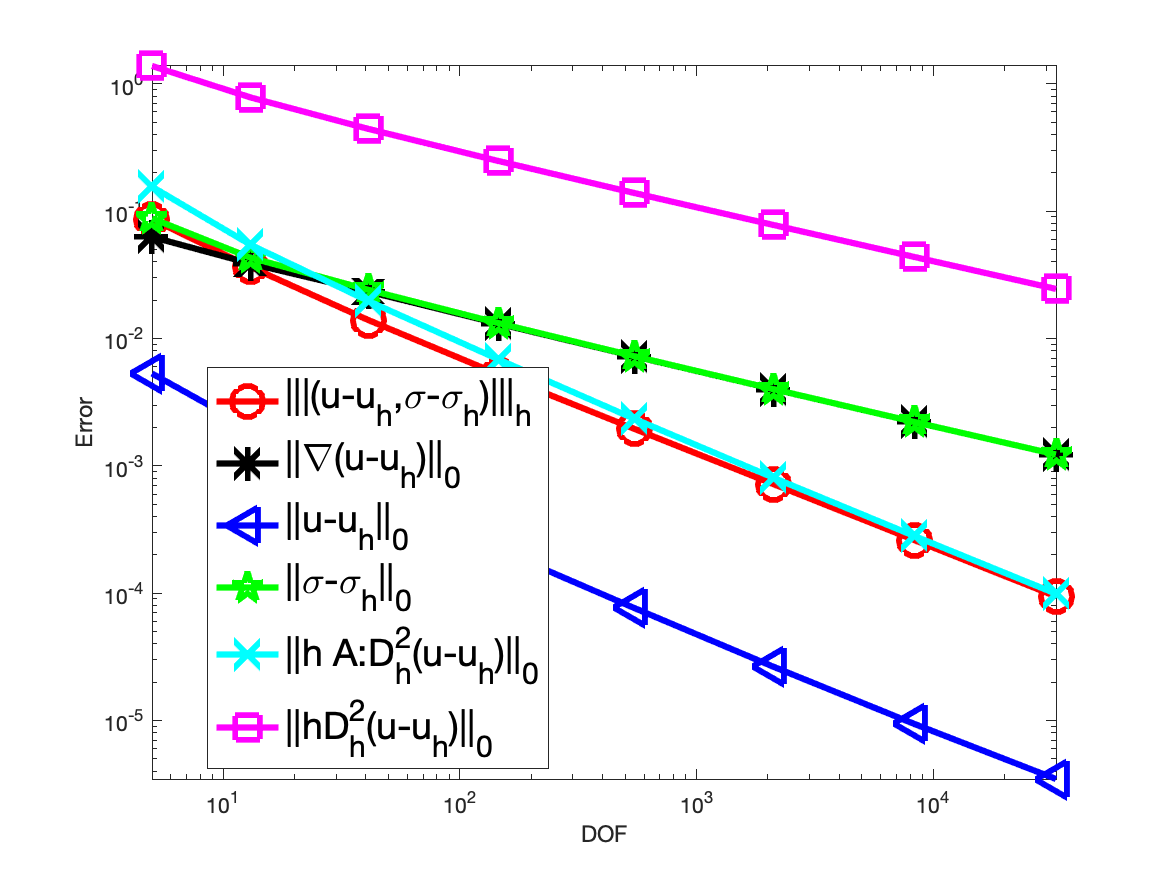}}
\caption{Convergence histories for a degenerate coefficient problem with $u =x^{4/3}-y^{4/3}$}
 \label{con_LSFEM_deg_x43}
\end{figure}

In Fig. \ref{con_LSFEM_deg_x43}, we show  the numerical results with uniform refinements using the $S_{1,0}\times S_1^2$ $L^2$-LSFEM and  $S_{2,0}\times S_1^2$ and $S_{3,0}\times S_2^2$ weighted-LSFEM. 

For the $S_{1,0}\times S_1^2$ $L^2$-LSFEM, $\tri(u-u_h,\bsigma-\bsigma_h) \tri_0$ has an order $0.63$, $\|\nabla (u-u_h)\|_0$ $\|\bsigma-\bsigma_h\|_0$ have an order $0.45$,  and $\|u-u_h\|_0$ has an order $0.85$.

The weighted LSFEMs have a convergence order $1.5$ for $\tri(u-u_h,\bsigma-\bsigma_h) \tri_h$ and $\|hA\ratio D_h^2(u-u_h)\|_0$, a convergence order $0.84$ for $\|\nabla (u-u_h)\|_0$ and $\|\nabla (u-u_h)\|_0$, and a convergence order $1.4$ for $\|u-u_h\|_0$. This again suggests that for the low regularity problem with uniformly mesh refinements, the higher order method is unnecessary for the weighted-LSFEM. 

Not surprisingly, due to the degenerate nature of $A$, $\|hD_h^2(u-u_h)\|_0$ is only of order $0.83$, which is much worse than order $1.5$ of $\|hA\ratio D_h^2(u-u_h)\|_0$. 

\begin{figure}[!htb]
\centering 
\subfigure[$L^2$-LSFEM with $S_{1,0}\times S_1^2$]{
\includegraphics[width=0.41\linewidth]{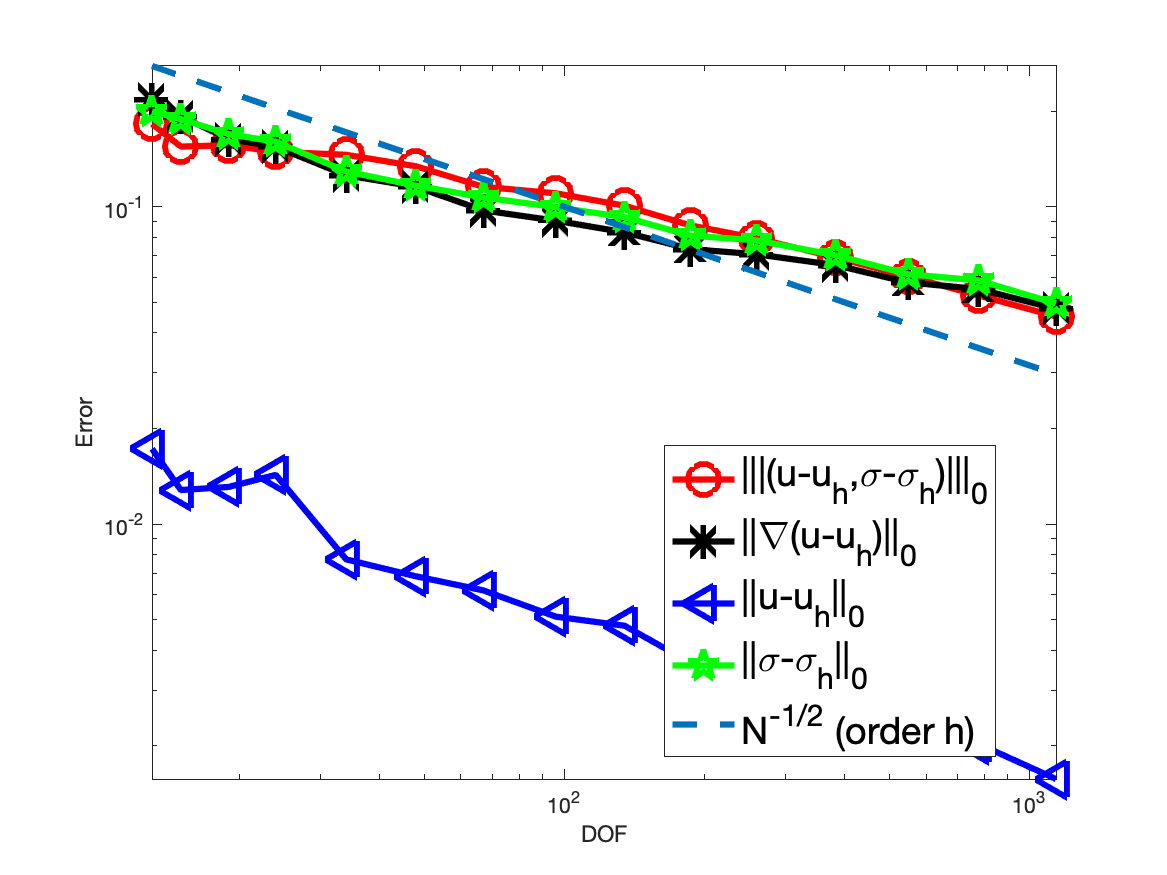}}
\subfigure[Weighted-LSFEM with $S_{2,0}\times S_1^2$]{
\includegraphics[width=0.41\linewidth]{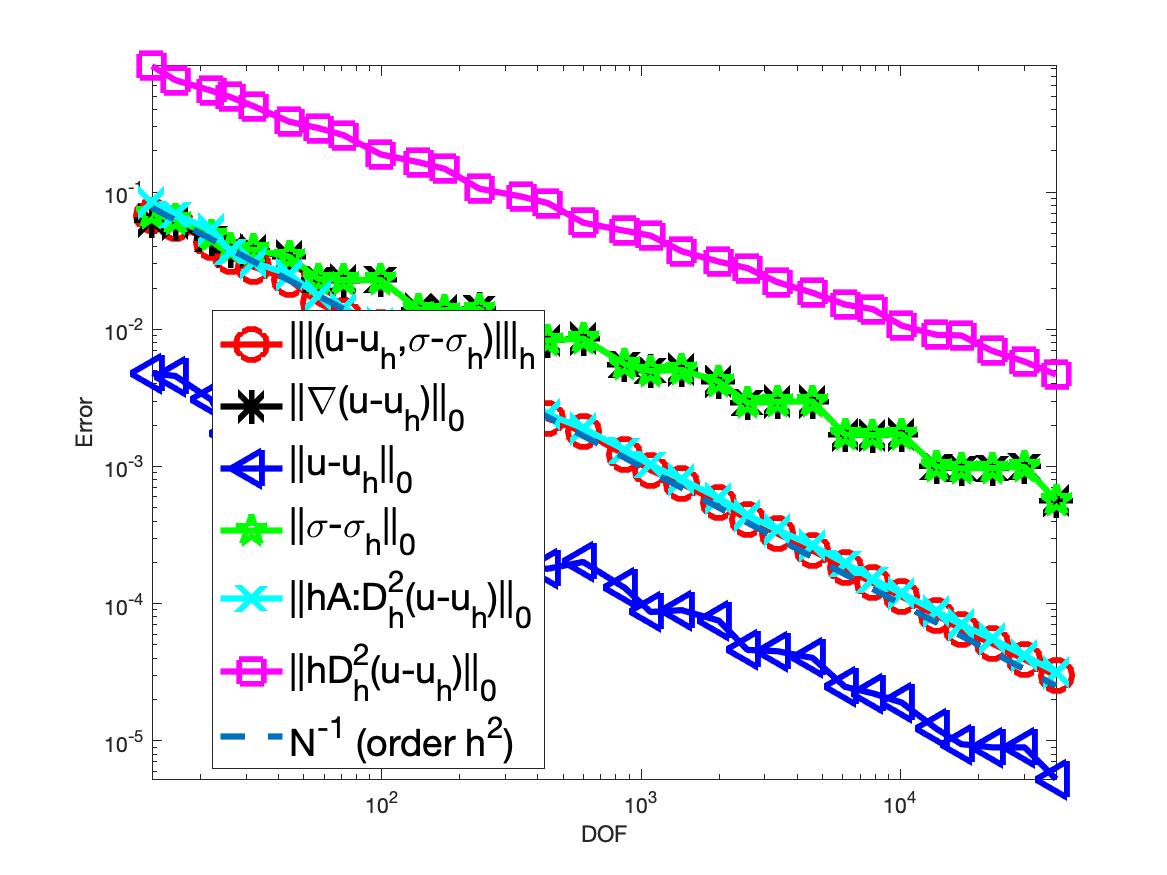}}
\subfigure[Weighted-LSFEM with $S_{3,0}\times S_2^2$]{
\includegraphics[width=0.41\linewidth]{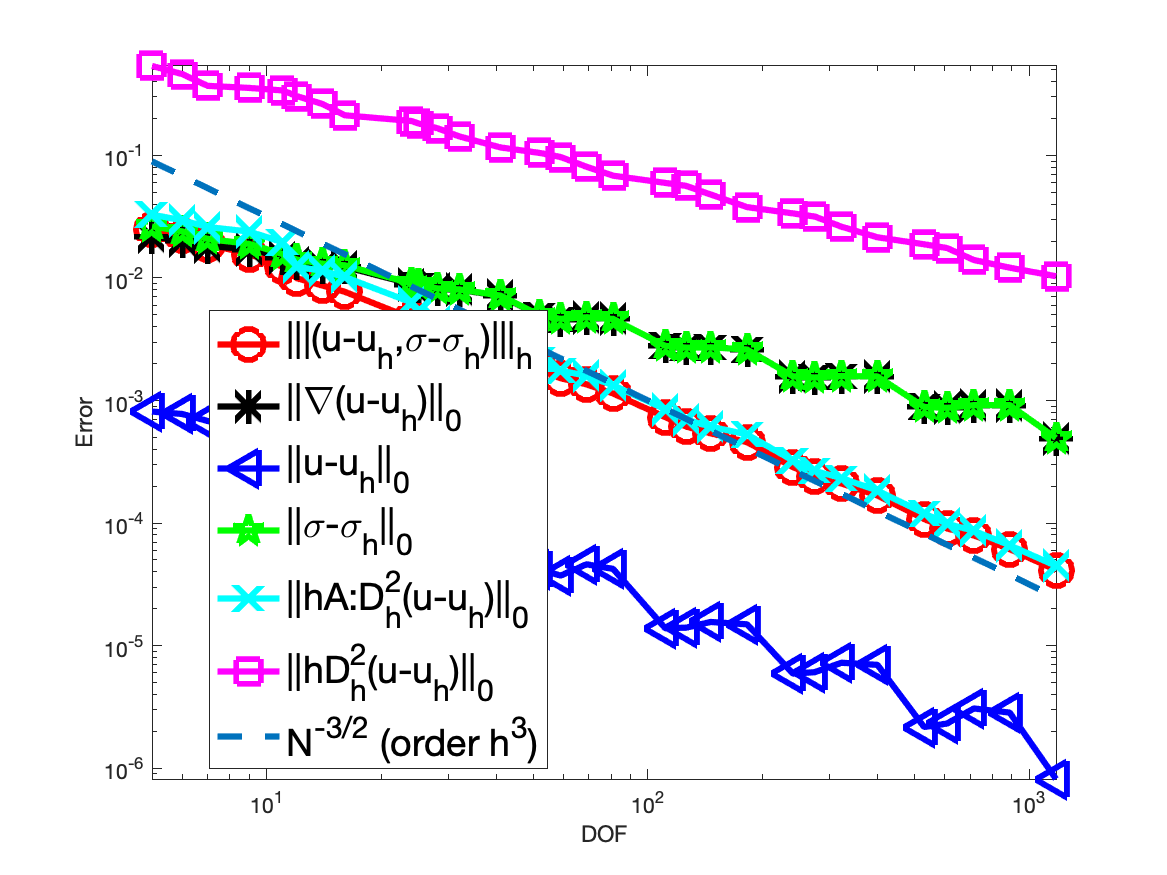}}
\subfigure[A refined mesh generated by adaptive $L^2$-LSFEM ]{
\includegraphics[width=0.41\linewidth]{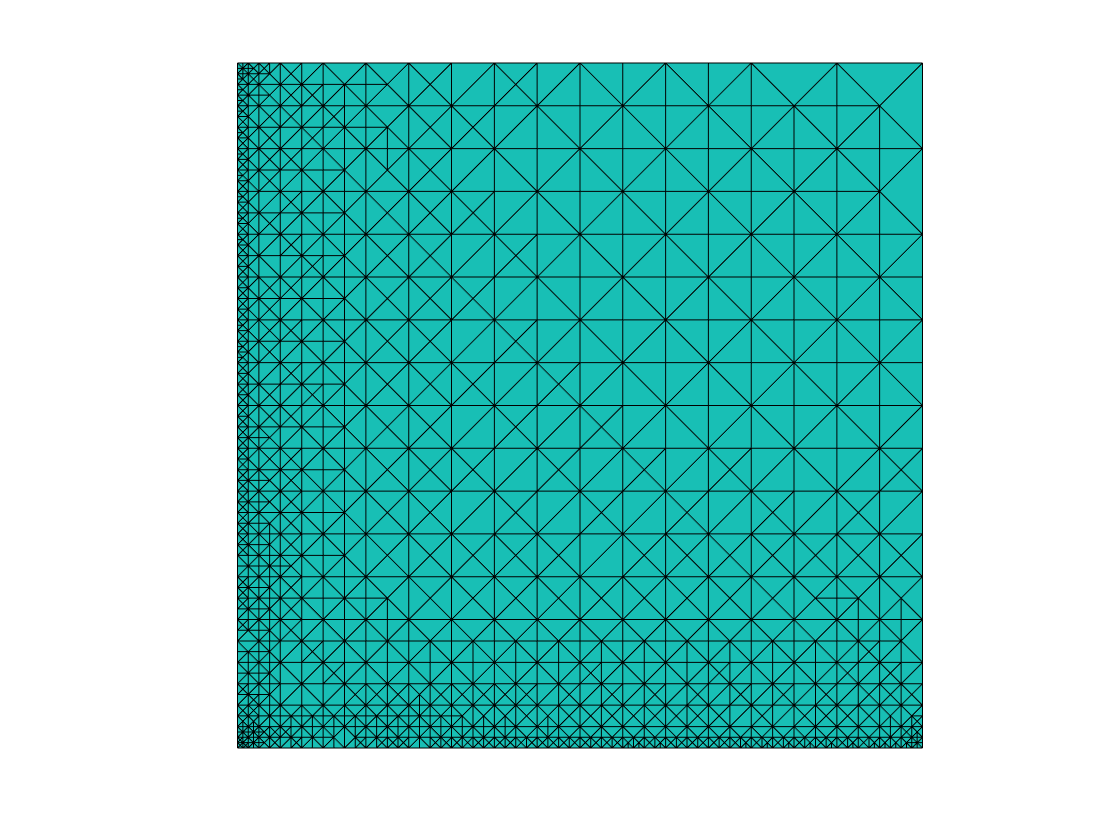}}
\caption{Convergence histories for a degenerate coefficient problem with $u =x^{4/3}-y^{4/3}$ with adaptive mesh refinements}
 \label{con_LSFEM_afem_x43}
\end{figure}

On (a), (b), and (c) of Fig. \ref{con_LSFEM_afem_x43}, we show  the numerical results with adaptive refinements using the adaptive $S_{1,0}\times S_1^2$ $L^2$-LSFEM and  $S_{2,0}\times S_1^2$ and $S_{3,0}\times S_2^2$ weighted-LSFEMs. 

For the adaptive $S_{1,0}\times S_1^2$ $L^2$-LSFEM, all norms except $\|u-u_h\|_0$ converge at a rate less than one , this is partly due to the fact the term $\|A\ratio \nabla \bsigma_h -f\|_0$ requires high regularity.  The error $\|u-u_h\|_0$  converges at a rate one, which is also less than optimal two, this is also partly due to the degenerate nature of the matrix $A$.

For the adaptive weighted-LSFEM, $\tri(u-u_h,\bsigma-\bsigma_h) \tri_h$ and $\|hA\ratio D_h^2(u-u_h)\|_0$ converge at the optimal order $k$. All other norms converge with a lower than optimal rate due the degenerate $A$. The error $\|u-u_h\|_0$ for $k=3$ only has a rate about $2$, which is also worse than the optimal $3$. That said, the higher order method behaves much better than the lower order methods with the adaptive methods.

On (d) of Fig. \ref{con_LSFEM_afem_x43}, we show an adaptively refined mesh generated by the adaptive $L^2$-LSFEM. It is clear from the graph that many refinements are along the $x$ and $y$ axises and the origin.

%
%
%
%
%

\subsubsection{L-shaped domain problem}
We choose the L-shaped domain $\O = (-1,1)^2 \backslash [0,1)\times (-1,0]$. The exact solution $u$ in polar coordinates is
\beq
u(r,\theta) = r^{2/3} \sin (2/3 \theta).
\eeq
It is easy to check that $\Delta u =0$.
We consider several different cases:
$$
A_5 = 
\left(
\begin{array}{cc} 
	5r^{1/2}+1  & r^2/2\\
  	r^2/2  &  5r^{1/2}+1 
   \end{array}
\right), \quad
A_6 = 
\left(
\begin{array}{cc} 
	5 -1/\ln(r) & r^2/2\\
  	r^2/2  &  5 -1/\ln(r)
   \end{array}
\right),
$$

$$
\quad \mbox{and} \quad
A_7 =
\left(
\begin{array}{cc} 
	2  & r^2 \frac{xy}{|xy|} \\
  	r^2 \frac{xy}{|xy|} &  2 
   \end{array}
\right).
$$

The coefficients of the matrices $A_5$, $A_6$, and $A_7$ are H\"older continuous, uniformly continuous, and discontinuous, respectively.  We choose $a_{11}=a_{22}$ since we want to use the fact $u_{xx} + u_{yy} =0$, and we choose $a_{12}=a_{21}$ with an $r^2$ factor to make sure the righthand side $f = - a_{12}u_{xy}- a_{21}u_{yx}$ belongs to $L^2$.

\begin{figure}[!htb]
\centering 
\subfigure[H\"older continuous coefficient $A_5$]{
\includegraphics[width=0.41\linewidth]{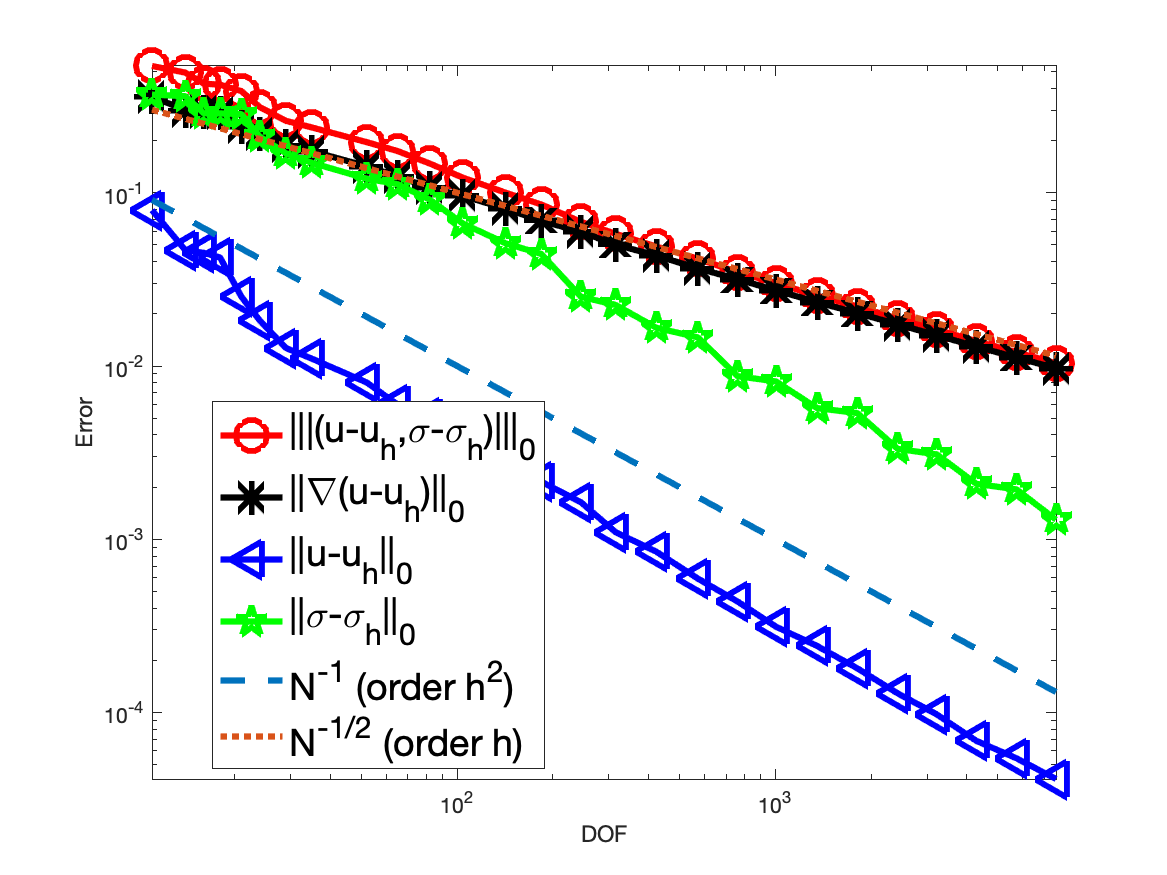}}
\subfigure[Uniformly continuous coefficient $A_6$]{
\includegraphics[width=0.41\linewidth]{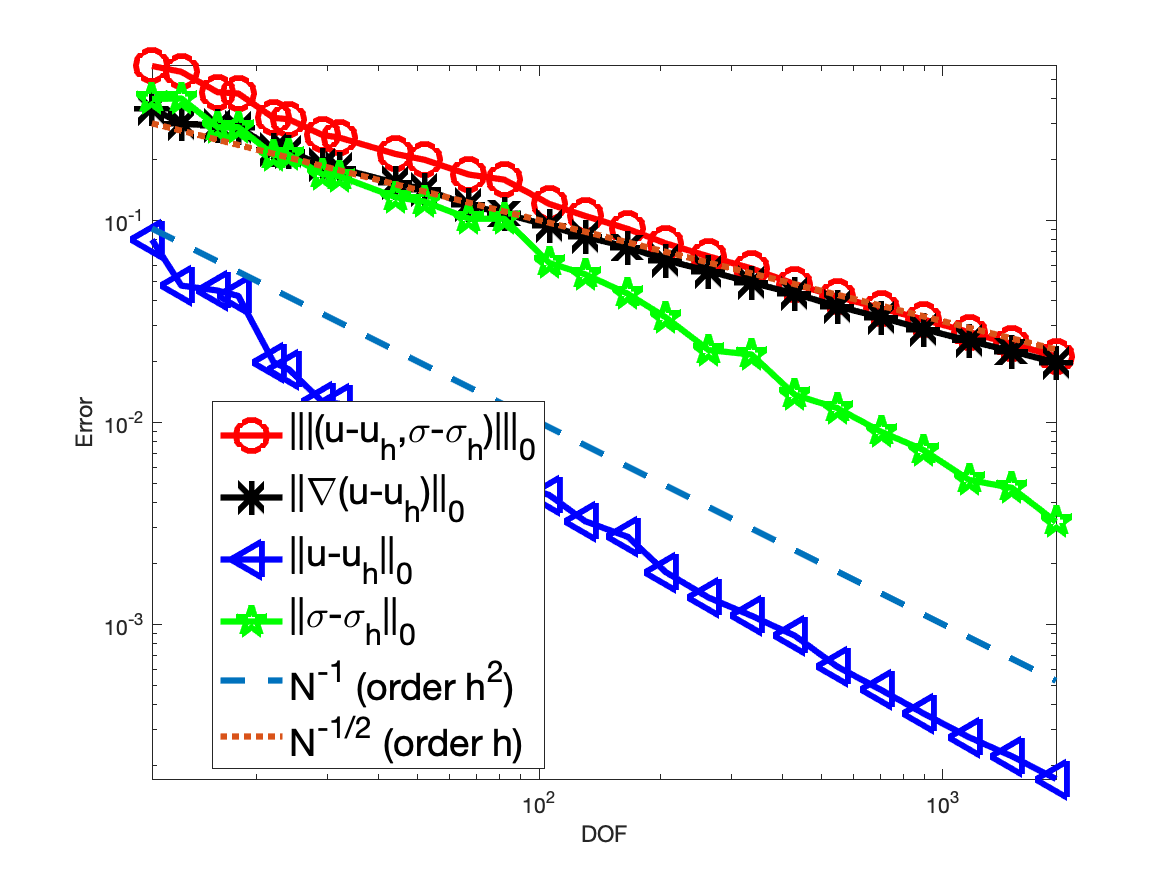}}
\subfigure[Discontinuous coefficient $A_7$]{
\includegraphics[width=0.41\linewidth]{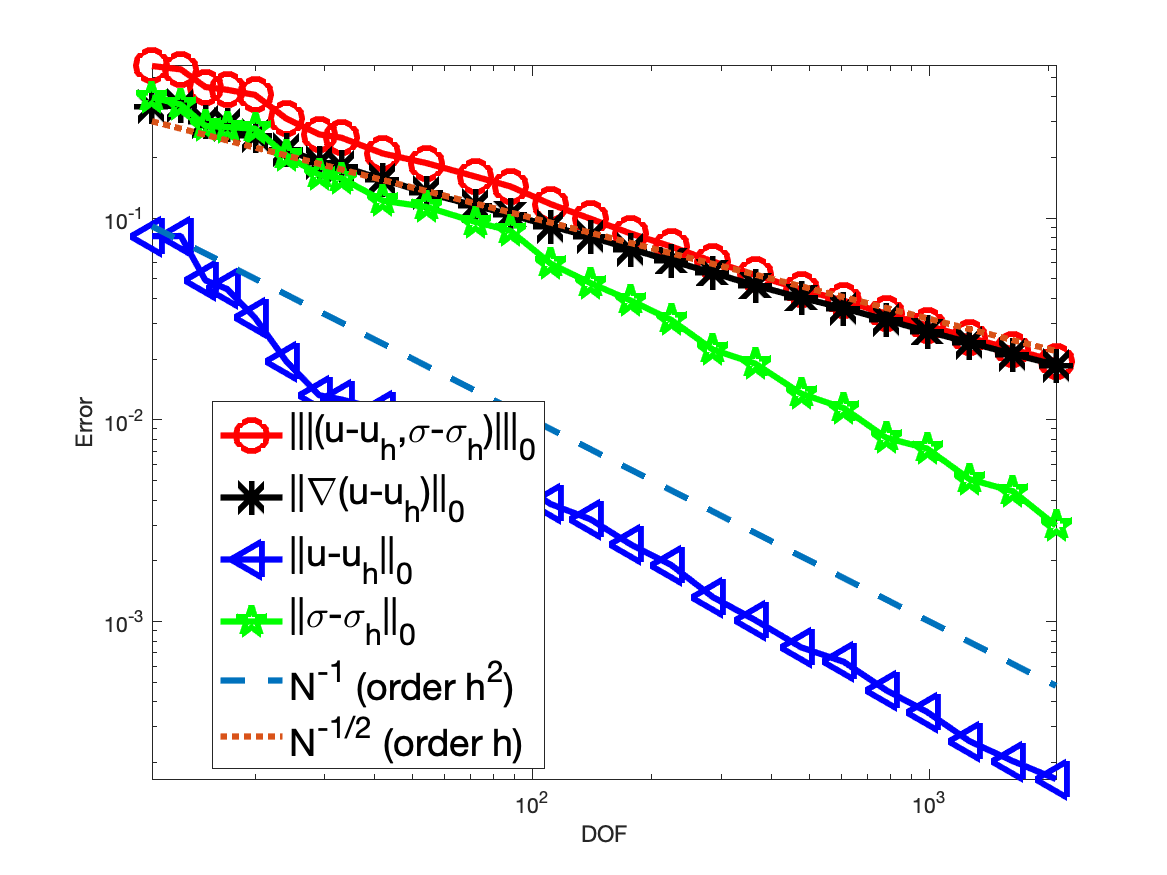}}
\subfigure[A refined mesh when $||\nabla  (u-u_h)||_0 \leq 0.01$, $N=7618$, $A_5$]{
\includegraphics[width=0.41\linewidth]{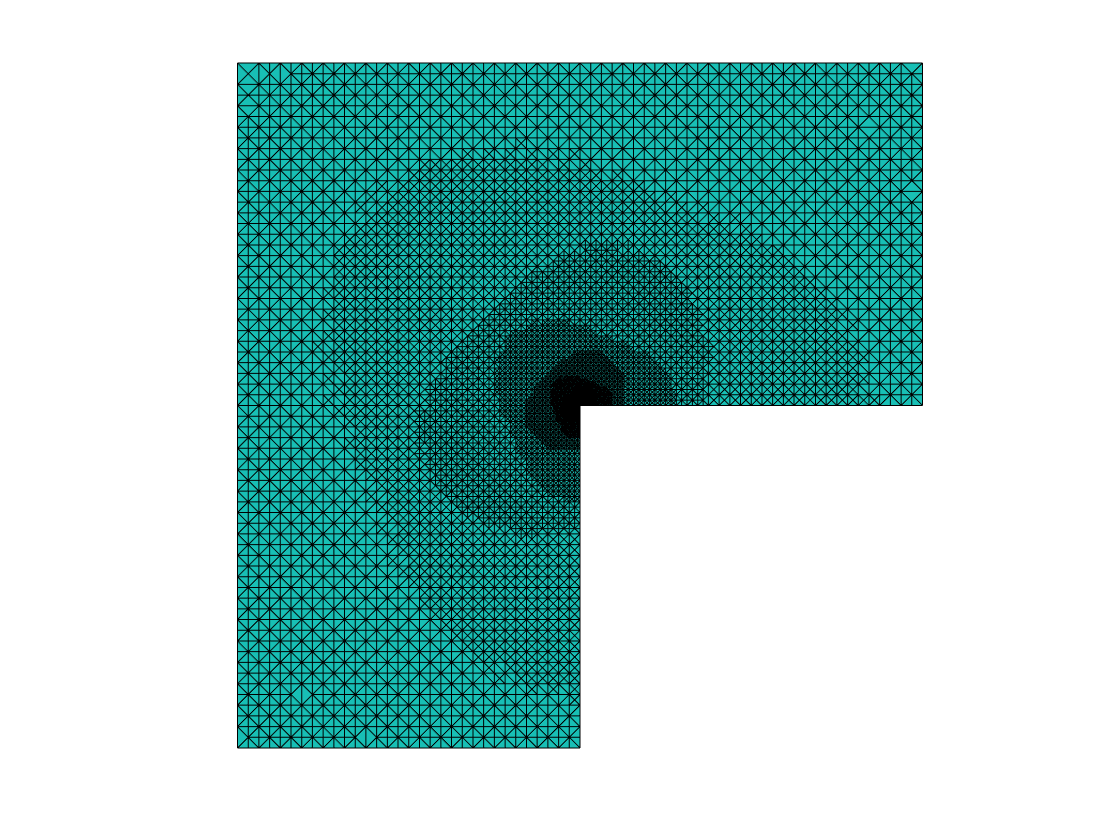}}
\caption{Convergence histories for adaptive $L^2$-LSFEM with $S_{1,0}\times S_1^2$  for the L-shaped problem}
 \label{con_LSFEM_S1_Lshape}
\end{figure}

\begin{figure}[htb]
\centering 
\subfigure[H\"older continuous coefficient $A_5$]{
\includegraphics[width=0.41\linewidth]{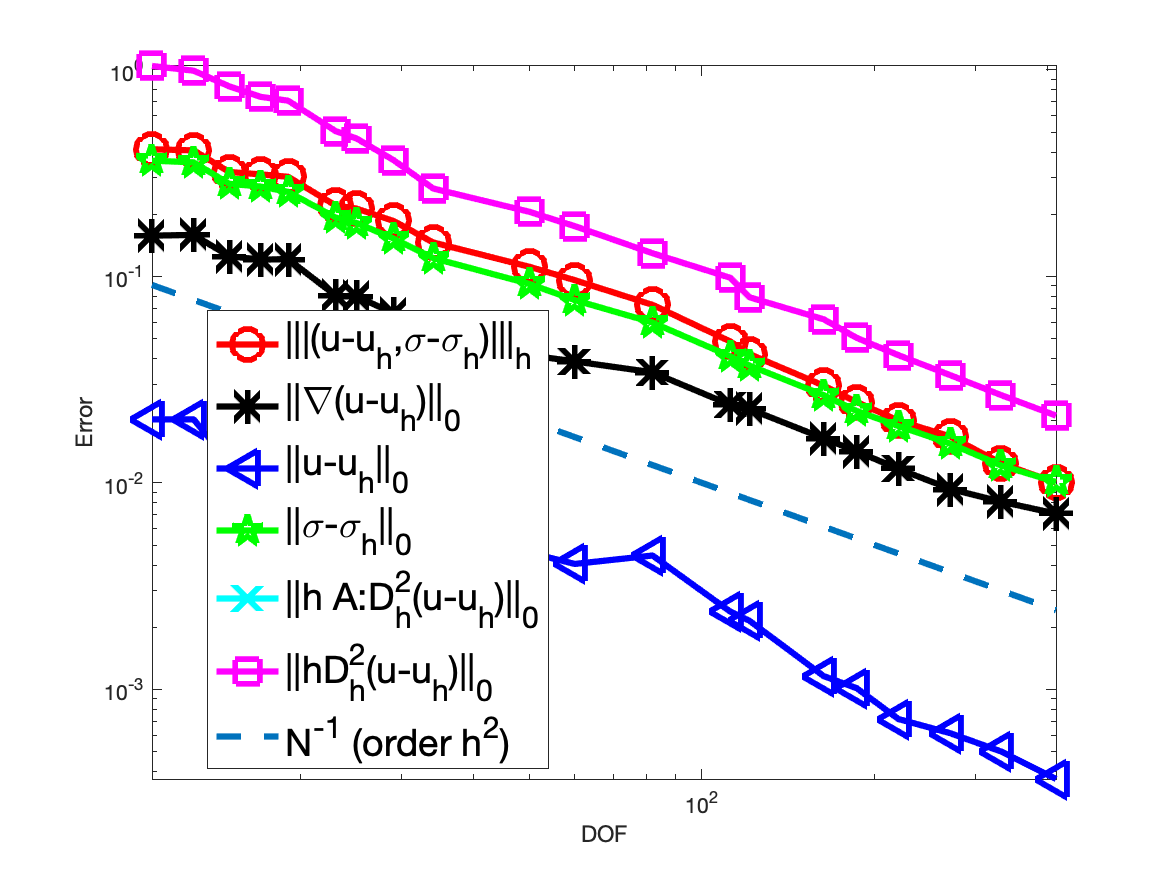}}
\subfigure[Uniformly continuous coefficient $A_6$]{
\includegraphics[width=0.41\linewidth]{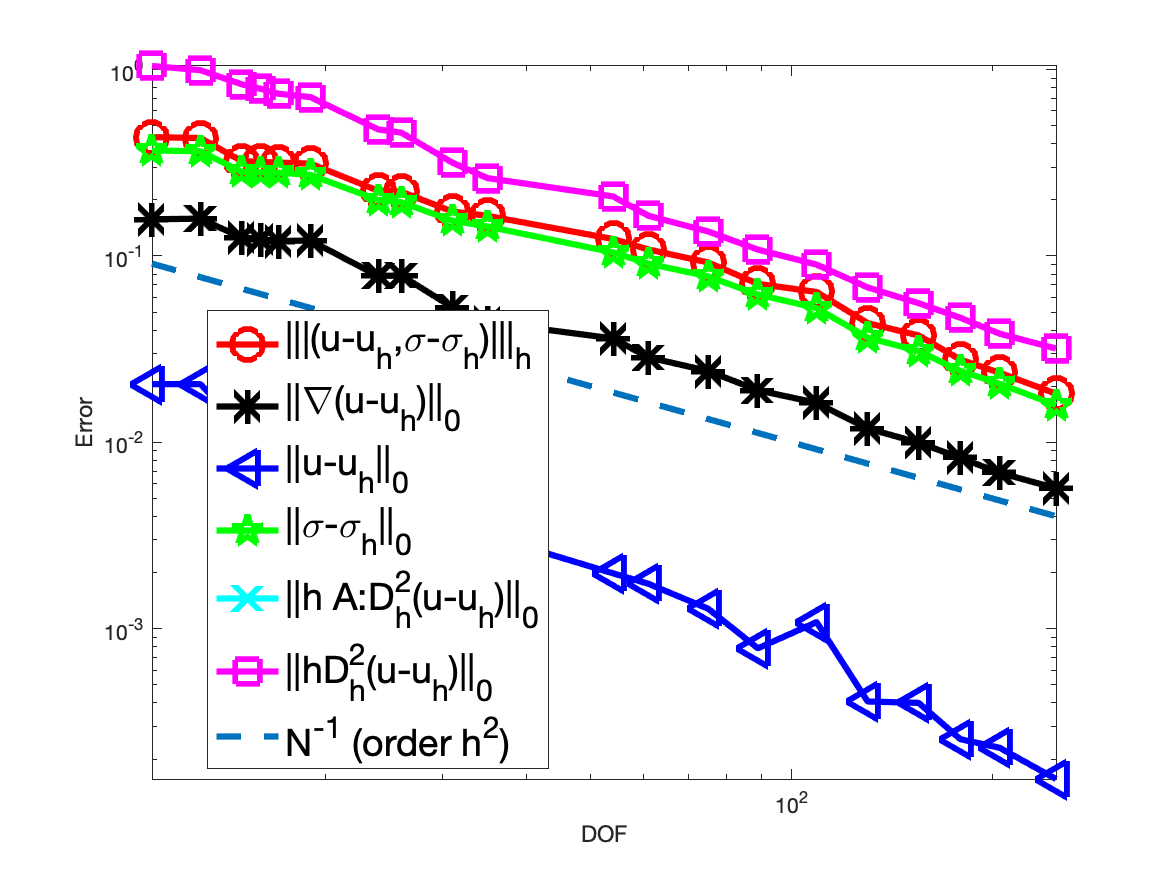}}
\subfigure[Discontinuous coefficient $A_7$]{
\includegraphics[width=0.41\linewidth]{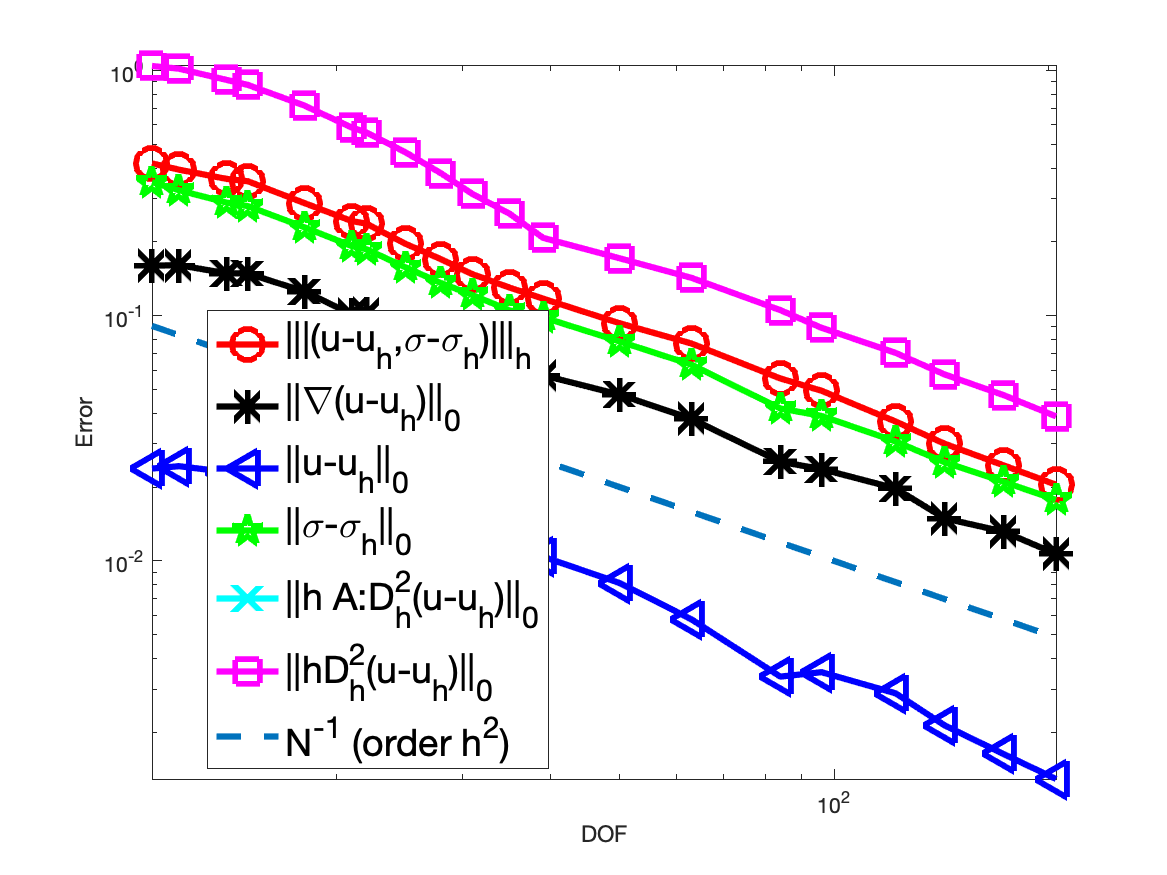}}
\subfigure[A refined mesh when $||\nabla  (u-u_h)||_0 \leq 0.01$, $N=271$, $A_5$]{
\includegraphics[width=0.41\linewidth]{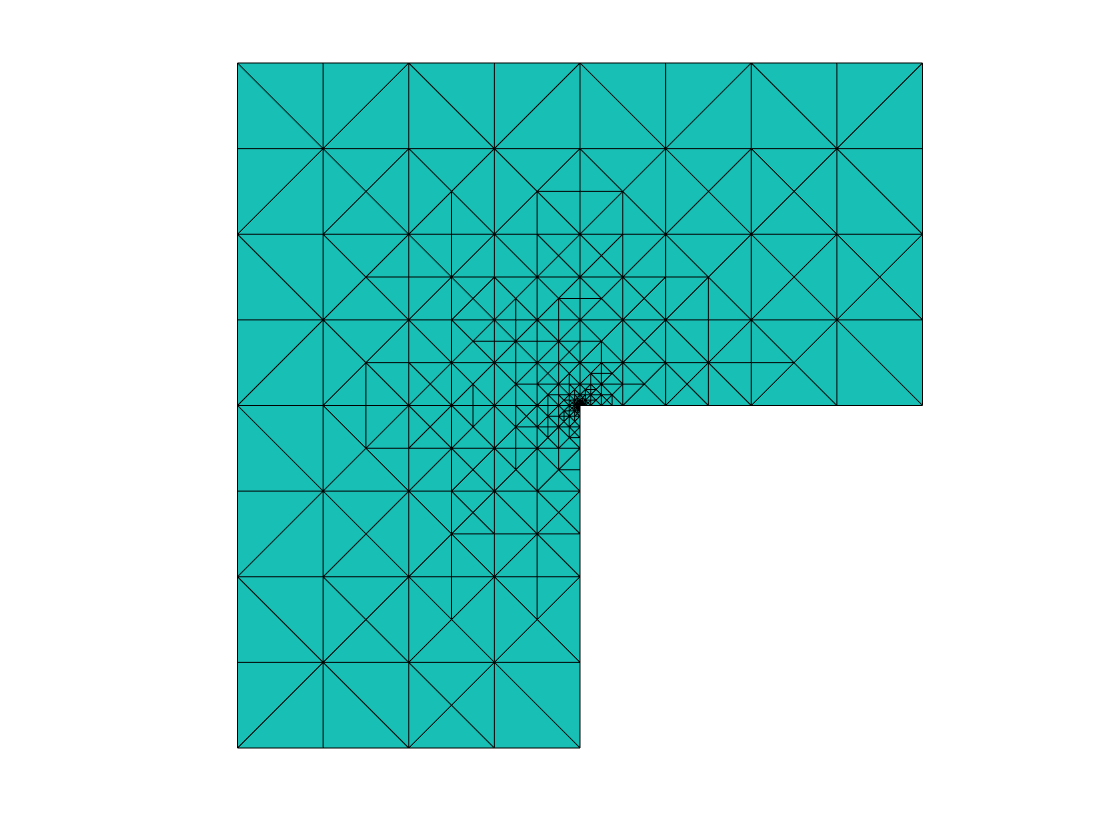}}
\caption{Convergence histories for adaptive weighted-LSFEM with $S_{2,0}\times S_1^2$  for the L-shaped problem}
 \label{con_LSFEM_S2_Lshape}
\end{figure}

\begin{figure}[htb]
\centering 
\subfigure[H\"older continuous coefficient $A_5$]{
\includegraphics[width=0.41\linewidth]{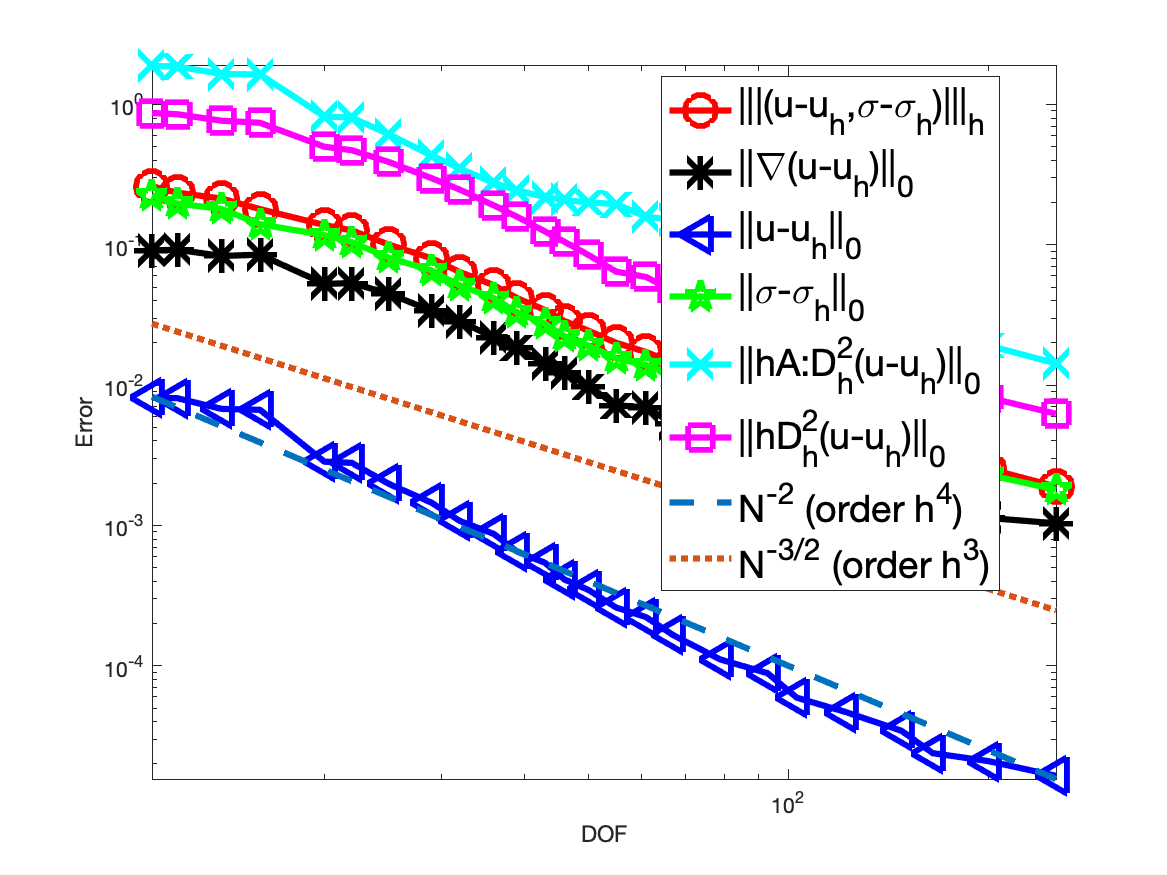}}
\subfigure[Uniformly continuous coefficient $A_6$]{
\includegraphics[width=0.41\linewidth]{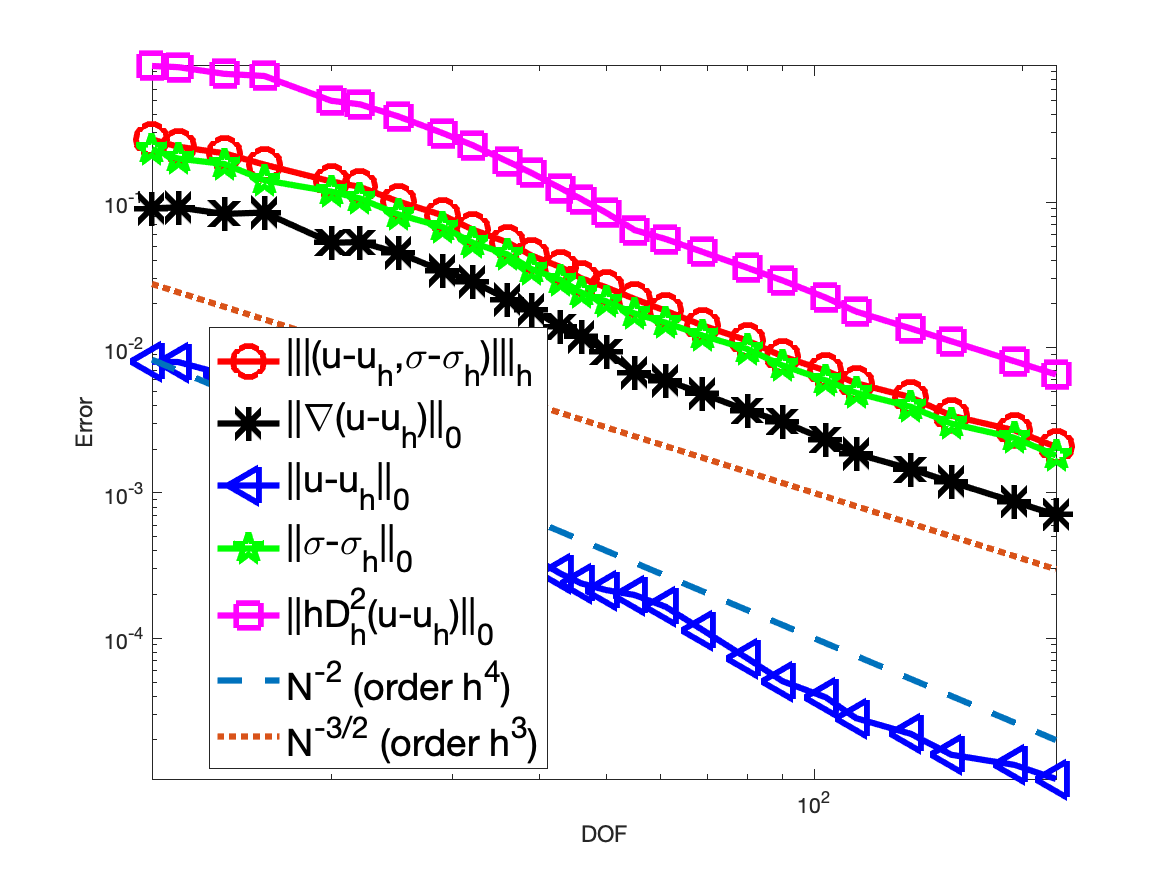}}
\subfigure[Discontinuous coefficient $A_7$]{
\includegraphics[width=0.41\linewidth]{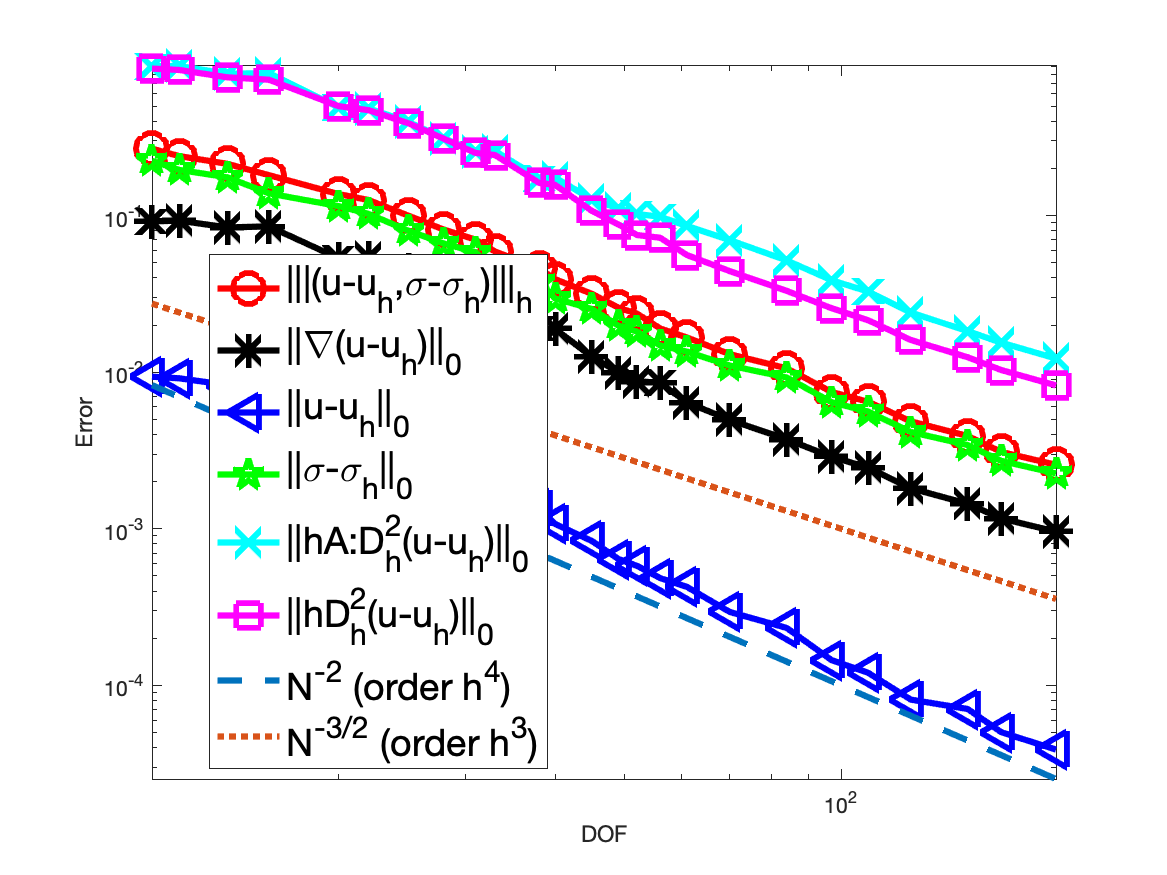}}
\subfigure[A refined mesh when $||\nabla  (u-u_h)||_0 \leq 0.01$, $N=50$, $A_5$]{
\includegraphics[width=0.41\linewidth]{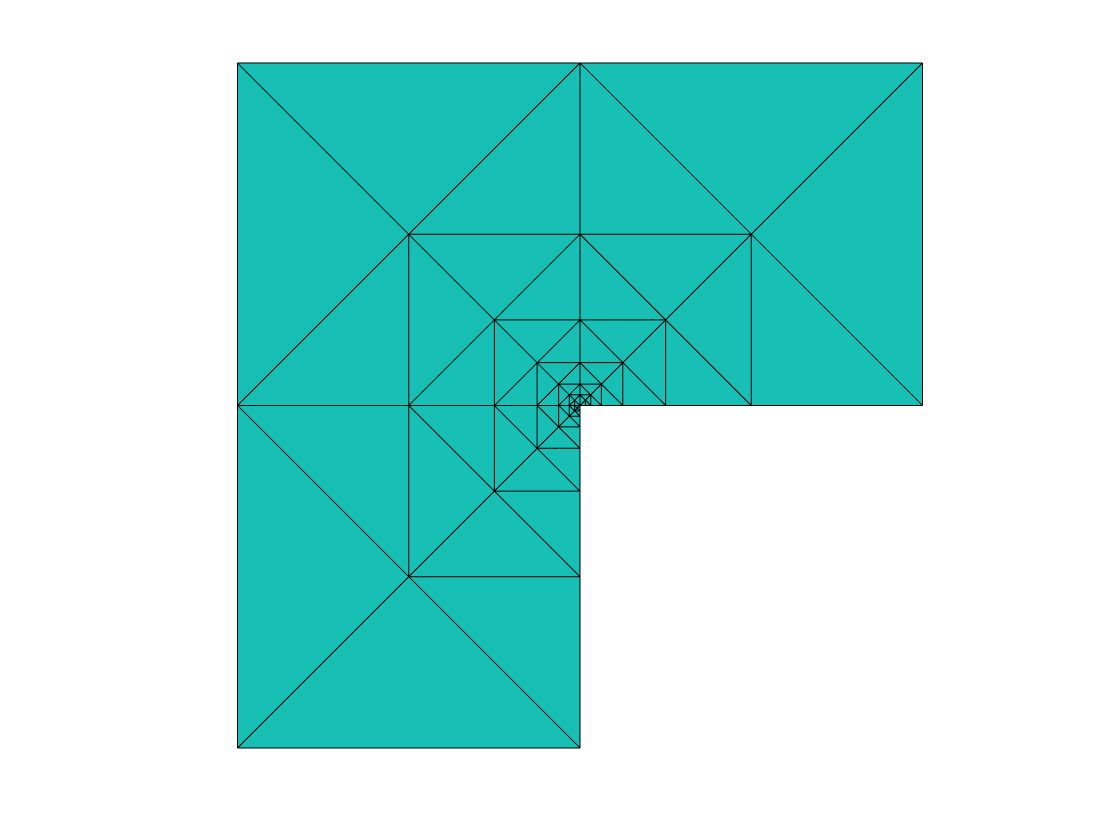}}
\caption{Convergence histories for adaptive weighted-LSFEM with $S_{3,0}\times S_2^2$  for the L-shaped problem}
 \label{con_LSFEM_S3_Lshape}
\end{figure}

In Fig. \ref{con_LSFEM_S1_Lshape}- Fig. \ref{con_LSFEM_S3_Lshape}, we show  the numerical results for the L-shaped problems with adaptive refinements using the adaptive $S_{1,0}\times S_1^2$ $L^2$-LSFEM and  $S_{2,0}\times S_1^2$ and $S_{3,0}\times S_2^2$ weighted-LSFEMs. All convergences orders are now optimal as if the solution and the matrix are smooth. On the (d) of Fig. \ref{con_LSFEM_S1_Lshape}- Fig. \ref{con_LSFEM_S3_Lshape}, we show the mesh when the $||\nabla  (u-u_h)||_0 \leq 0.01$ with the matrix $A_5$. It is also very clear that unlike the methods with a uniform mesh, high order methods combined with adaptive mesh refinements are superior compared with  lower oder methods. For our test problem with the H\"older continuous coefficients $A_5$, the weighted-LSFEM with $S_{3,0}\times S_2^2$ only uses  $50$ mesh points while the $L^2$-LSFEM with  $S_{1,0}\times S_1^2$ requires $7618$ nodes to reduce the $H^1$-semi norm of the error of $u$ to be less than $0.01$.

\end{document}